\documentclass{article}
\usepackage{graphicx} 
\usepackage{amsthm} 
\usepackage{amssymb}
\usepackage[
monochrome
]{xcolor}
\usepackage{amsmath} 
\usepackage{hyperref} 
\usepackage{bbm} 
\usepackage[a4paper, total={6in, 9in}]{geometry}
\usepackage{relsize} 
\usepackage{dirtytalk} 

\theoremstyle{plain}
\newcounter{stmcounter}[section]

\newtheorem{corollary}[stmcounter]{Corollary}
\newtheorem{theorem}[stmcounter]{Theorem}

\newtheorem{proposition}[stmcounter]{Proposition}
\newtheorem{lemma}[stmcounter]{Lemma}
\newtheorem{definition}[stmcounter]{Definition}

\newtheorem{remark}[stmcounter]{Remark}
\newtheorem{assumption}[stmcounter]{Assumption}

\title{Well-Posedness of the generalised Dean--Kawasaki Equation with correlated noise on bounded domains}
\author{Shyam Popat\thanks{Mathematical Institute, University of Oxford, Oxford OX2 6GG, UK (popat@maths.ox.ac.uk)}}

\date{\today}

\begin{document}

\maketitle


\small
\begin{center}
ABSTRACT. In this paper, we extend the notion of stochastic kinetic solutions introduced in \cite{fehrman2024well} to establish the well-posedness of stochastic kinetic solutions of generalised Dean--Kawasaki equations with correlated noise on bounded, $C^2$-domains with Dirichlet boundary conditions.  The results apply to a wide class of non-negative boundary data, which is based on certain a priori estimates for the solutions, that encompasses all non-negative constant functions including zero and all smooth functions bounded away from zero.
\end{center}
\vspace{10pt}
\textbf{MSC2020 subject classification:} \hspace{5pt}60H15, 82B21, 82B31, 35Q70\\
\textbf{Keywords:}\hspace{5pt} Stochastic partial differential equation, Dean--Kawasaki equation, Dirichlet boundary conditions
\normalsize

\section{Introduction}
We will consider the well-posedness of the generalised Dean--Kawasaki initial boundary value problem on a $C^2$ and bounded domain $U\subset\mathbb{R}^d$,
\begin{equation}\label{generalised Dean--Kawasaki Equation Stratonovich}
   \begin{cases}
    \partial_t\rho=\Delta\Phi(\rho)-\nabla\cdot(\sigma(\rho)\circ\dot{\xi}^F+ \nu(\rho)), & \text{on} \hspace{5pt} U\times(0,T],\\
\Phi(\rho)=\bar{f},&\text{on}\hspace{5pt} \partial U\times[0,T],\\
    \rho(\cdot,t=0)=\rho_0 ,&\text{on}\hspace{5pt} U\times\{t=0\}.
   \end{cases} 
\end{equation}
The main uniqueness and existence assumptions for the non-linear functions $\Phi,\sigma$ and $\nu$ are given in Assumptions \ref{assumptions on coefficients for uniqueness} and \ref{existence assumptions} respectively. 
The assumptions allow us to consider a wide range of relevant stochastic PDE such as the full range of fast diffusion and porous medium equations, i.e. $\Phi(\rho)=\rho^m$ for every $m\in(0,\infty)$ and the degenerate square root $\sigma(\rho)=\sqrt{\rho}$.
Subsequently we will refer to this choice of $\Phi$ and $\sigma$, and the alternative choice $\sigma(\rho)=\Phi^{1/2}(\rho)$ as the model case.
The Stratonovich noise $\circ \dot{\xi}^F$ is white in time and sufficiently regular in space, see Definition \ref{definition noise}.\\
We will prove the existence and uniqueness of a stochastic kinetic solutions of \eqref{generalised Dean--Kawasaki Equation Stratonovich}, in the sense of Definition \ref{definition of stochastic kinetic solution of generalised Dean--Kawasaki equation} below.
In the definition we do not insist that $\rho$ is continuous all the way to the boundary and so we \textcolor{red}{are unable to} define the boundary condition in a pointwise sense.
We only require the solution $\rho$ to be locally in the Sobolev space $W^{1,2}(U)$, i.e. when cutoff away from zero and infinity, and otherwise just to be $L^1(U)$ regular.
Hence the boundary condition in \eqref{generalised Dean--Kawasaki Equation Stratonovich} is expressed indirectly via $\Phi(\rho)$ and is done using the trace theorem, see \textcolor{red}{\cite[Sec.~5.5]{evans2022partial}}.
Furthermore, the boundary condition does not depend on time, $\bar{f}=\bar{f}(x^*)$ for $x^*\in\partial U$.
The restriction on the boundary data $\bar{f}$ that we are able to handle comes from assuming finiteness of the various boundary terms that appear in the a priori energy estimates.\\
A detailed summary outlining various reasons why one would be interested in studying stochastic PDE such as \eqref{generalised Dean--Kawasaki Equation Stratonovich} is given in \textcolor{red}{\cite[Sec.~1.3]{fehrman2024well}}.
To briefly summarise, they arise as fluctuating continuum models for interacting particle systems, see \cite{giacomin1998deterministic}, and they can also be used to describe the hydrodynamic limit of simple particle processes, for instance the simple exclusion process, see \cite{quastel1999large}, and the zero range process, see \textcolor{red}{\cite[Sec.~4.3]{dirr2016entropic}}.\\
In \cite{fehrman2024well}, Fehrman and Gess prove the well-posedness of equations such as \eqref{generalised Dean--Kawasaki Equation Stratonovich} on the torus.
Motivated by the particle system application, we will follow and extend the techniques introduced there in order to extend the well-posedness to a bounded domain.
In this context the Dirichlet boundary condition enables one to model sources and sinks.

\subsection{Background and relevant literature}
The Dean--Kawasaki equation was derived independently in the physics literature by Dean \cite{dean1996langevin} and Kawasaki \cite{kawasaki1994stochastic}.
In \cite{dean1996langevin}, Dean considered an $N$ particle system $\{X^i\}_{i=1}^N$ following Langevin dynamics with pairwise interaction potential $V^1$, 
\[\dot{X}_t^i=\frac{1}{N}\sum_{j=1}^N V^1(X_t^i-X_t^j)+\sqrt{2}\dot{\beta^i_t},\]
where $\{\beta^i\}_{i=1}^N$ are independent Brownian motions.
\textcolor{red}{It was shown that empirical measure $\rho^N_t:=\frac{1}{N}\sum_{i=1}^N\delta_{X_t^i}$ of the particle system satisfies an equation whose quadratic variation agrees with those of solutions to} 
\begin{equation}\label{dean's equation}
    \partial_t\rho^N=\textcolor{red}{\frac{\alpha}{2}}\Delta \rho^N-\nabla\cdot\left(\rho^N\nabla V^1\ast\rho^N\right)-\frac{\sqrt{2}}{\sqrt{N}}\nabla\cdot\left(\sqrt{\rho^N}\xi\right),
\end{equation}
where $\xi$ is space\textcolor{red}{-}time white noise. 
\textcolor{red}{That is to say, the empirical measure of the particle system and equation \eqref{dean's equation} can be seen to agree as semimartingales.}
The irregularity of space\textcolor{red}{-}time white noise implies that equations like \eqref{dean's equation} and \eqref{generalised Dean--Kawasaki Equation Stratonovich} with space\textcolor{red}{-}time white noise are not renormalisable using Hairer's regularity structures \cite{hairer2014theory} or Gubinelli, Imkeller and Perkowski's paracontrolled distributions \cite{gubinelli2015paracontrolled}, even in dimension $d=1$.
Since the derivative hits space\textcolor{red}{-}time white noise in \eqref{dean's equation}, one can only consider solutions $\rho^N$ in the space of distributions rather than functions, in which case the square root $\sqrt{\rho^N}$ as well as the product $\sqrt{\rho^N}\xi$ have no \textcolor{red}{pathwise classical meaning}.
\textcolor{red}{One might hope that there exist non-atomic solutions to equation \eqref{dean's equation} or \eqref{generalised Dean--Kawasaki Equation Stratonovich} with space\textcolor{red}{-}time white noise.
However, it has been shown by Konarovskyi, Lehmann and von Renesse in \cite[Thm.~2.2]{konarovskyi2019dean} and \cite[Thm.~1]{konarovskyi2020dean} that martingale solutions to equation \eqref{dean's equation} exist if and only if $\alpha=N\in\mathbb{N}$ and the initial condition is atomic, that is $\rho^N_0=\frac{1}{N}\sum_{i=1}^N\delta_{x_i}$. 
In this case, martingale solutions to \eqref{dean's equation} are in fact atomic empirical measures.}
Furthermore, they show that one at least needs to add a drift correction term into the equations to obtain non-trivial solutions.
It is for this reason that the Dean--Kawasaki equation is described as a \say{rigid mathematical object} in \cite{cornalba2023density}.
However, this rigidity can be overcome when suitable regularisations are considered, for instance smoothing or truncating the noise.
We now discuss previous work on regularised versions of the Dean--Kawasaki equation based on relevance with the present work.\\
As mentioned above, this work primarily follows the techniques presented in Fehrman and Gess, \cite{fehrman2024well}.
There the authors consider a white in time, spatially correlated noise, see Definition \ref{definition noise} below.
Even though the noise is sufficiently regular in space the well-posedness of equation \eqref{generalised Dean--Kawasaki Equation Stratonovich} is tricky due to the square root singularity $\sigma(\rho)=\sqrt{\rho}$.
Indeed, the It\^o-to-Stratonovich conversion of \eqref{generalised Dean--Kawasaki Equation Stratonovich} introduces a term with factor $(\sigma'(\rho))^2\nabla\rho$ that creates a singularity at zero, see derivation of equation \eqref{generalised Dean--Kawasaki equation Ito} and Remark \ref{remark highlighting singularity in ito correction} below.
The lack of local integrability of $\log(\rho)$ means that even a weak solution to \eqref{generalised Dean--Kawasaki Equation Stratonovich} cannot be considered. 
Instead, the authors consider stochastic kinetic solutions, see Definition \ref{definition of stochastic kinetic solution of generalised Dean--Kawasaki equation} below, where the compact support in the velocity variable restricts the solution away from its zero set, see Remark \ref{remark about compact support of test function in defn of stochastic kinetic solution}.
\textcolor{red}{In fact, the Stratonovich noise in \eqref{generalised Dean--Kawasaki Equation Stratonovich} is essential to the well-posedness, as the above It\^o-to-Stratonovich term cancels exactly with a \say{bad} It\^o correction term whenever It\^o's formula is applied, for instance see derivation of the kinetic equation \cite[Eq.~(3.1)]{fehrman2024well} and subsequent derivation, as well $L^2(U)$ a priori energy estimates, see computation between equation \eqref{ito formula in energy estimate} and \eqref{equation for L2 estimate} below.}\\
The kinetic formulation and the notion of kinetic solutions was first introduced in the PDE setting by Lions, Perthame and Tadmor \cite{lions1994kinetic} where the authors considered the kinetic formulation for multidimensional scalar conservation laws.
They chose the term \say{kinetic equation} due to its analogy with the classical kinetic models such as
 Boltzmann or Vlasov models, see \cite{cercignani1988boltzmann}.
Later, Chen and Perthame \cite{chen2003well} were able to  treat non-isotropic degenerate parabolic--hyperbolic equations.
A new chain rule type condition introduced there allowed kinetic equations to be well defined, even when the macroscopic fluxes are not locally integrable.\\
Before proceeding, it is worth mentioning the relationship between stochastic kinetic solutions and weak solutions in the context of the Dean--Kawasaki equation.
When the diffusion coefficient $\sigma$ is sufficiently smooth so that we can make sense of weak solutions to \eqref{generalised Dean--Kawasaki Equation Stratonovich}, a weak solution is a stochastic kinetic solution, see \textcolor{red}{\cite[Prop.~5.21]{fehrman2024well}}.
Conversely, under additional assumptions on $\sigma$, stochastic kinetic solutions are also weak solutions, see \textcolor{red}{\cite[Cor.~5.31]{fehrman2024well}}.\\
Other related works are that of Fehrman and Gess \cite{fehrman2019well,fehrman2021path} where they prove the path by path well-posedness of equations like \eqref{generalised Dean--Kawasaki Equation Stratonovich} via a kinetic approach, motivated by the theory of stochastic viscosity solutions, see \cite{lions1998fully,lions1998fully1,lions2000fully}, and stochastic conservation laws, see \cite{lions2013scalar,lions2014scalar,gess2014scalar,gess2017stochastic}. 
In the recent work of Clini \cite{clini2023porous} the result was extended to a smooth bounded domain with zero Dirichlet boundary conditions in the case of the porous media and fast diffusion equations, i.e. $\Phi(\rho)=\rho^m, \nu(\rho)=0$ in \eqref{generalised Dean--Kawasaki Equation Stratonovich}.
\textcolor{red}{The pathwise well-posedness in the above works is proven via rough path techniques, which imposes extra regularity on the diffusion coefficient $\sigma$ that is needed to overcome the roughness of the noise $\xi^F$.
For instance, in the case of Brownian noise, $\sigma$ is required to be six times continuously differentiable, so the works fall well outside of the critical square root case.} \\
Another extension of \cite{fehrman2024well} was that of Wang, Wu and Zhang \cite{wang2022dean}. 
The authors are able to show the existence of renormalised kinetic solutions to non-local equations such as \eqref{dean's equation} where we have a convolution, rather than the local interaction $\nu(\rho)$ in \eqref{generalised Dean--Kawasaki Equation Stratonovich} considered here and in \cite{fehrman2024well}.  
There the interaction kernel ($V^1$ in \eqref{dean's equation}) is assumed to satisfy the Ladyzhenskaya-Prodi-Serrin (LPS) condition, a regularity condition first studied in the context of Navier-Stokes equations in  \cite{prodi1959teorema, serrin1961interior, ladyzhenskaya1967uniqueness} and applied to SDEs and distributional dependent SDEs in \cite{krylov2005strong,rockner2021well} respectively. 
The main difficulty lies in the proof of existence, where the lack of uniform $L^p(\mathbb{T}^d)$ estimates of weak solutions of the regularised equation implies that the kinetic measures of the regularised equation are not uniformly bounded over $[0,T]\times \mathbb{T}^d\times \mathbb{R}$.
One needs to instead use entropy estimates, similar to Proposition \ref{entropy estimate} below, to show the tightness of the kinetic measures of the approximate equations.\\
Another approach was by Dareiotis and Gess \cite{dareiotis2020nonlinear} where they constructed probabilistic solutions to \eqref{generalised Dean--Kawasaki Equation Stratonovich} via an entropy formulation.
The advantage of the kinetic approach considered here and in \cite{fehrman2024well} over the entropy approach is that, by the precise identification of the kinetic defect measure, the kinetic approach can handle $L^1(U)$ integrable initial data, whereas the entropy approach requires $L^m(U)$ integrable initial data in the model case $\Phi(\xi)=\xi^m$.
Furthermore, the approach in \cite{dareiotis2020nonlinear} required $C^{1,\alpha}$-regularity from the noise coefficient $\sigma$, and therefore remained far from the critical square root case $\sigma(\rho)=\sqrt{\rho}$.\\
Djurdjevac, Kremp and Perkowski in \cite{djurdjevac2022weak} were able to prove the existence of strong solutions to equations like \eqref{generalised Dean--Kawasaki Equation Stratonovich} in the case of truncated noise and mollified square root.
The proof follows by a variational approach for a transformed equation and energy estimates for the approximate (Galerkin projected) system.
However, again it is worth mentioning that the approach can only handle smooth coefficients and so also cannot handle the square root singularity.\\
Also relevant is the work of Martini and Mayorcas \cite{martini2022additive}, where local well-posedness of equations like \eqref{generalised Dean--Kawasaki Equation Stratonovich} with space\textcolor{red}{-}time white noise is proven when the square root $\sigma(\rho)=\sqrt{\rho}$ is replaced by $\sigma=\sqrt{\rho_{det}}$ where $\rho_{\det}$ solves a related deterministic PDE.
The proof is via paracontrolled distribution theory.\\
Finally, we mention that there is an extensive literature surrounding the well-posedness of stochastic nonlinear diffusion equations on a smooth, bounded domains in $\mathbb{R}^d$ with zero Dirichlet boundary conditions.
In \cite{barbu2006weak}, Barbu, Bogachev, Da Prato and R\"ockner are able to show the existence of weak solutions of additive equations of the form
\[\partial_t \rho=\Delta\Phi(\rho)+\sigma(\rho)\dot{W}_t,\]
where $W_t$ is cylindrical Brownian motion,  $\sigma=\sqrt{Q}$ in the case that $Q$ is linear, non-negative and of finite trace and $\Phi'$ satisfies a polynomial growth condition.
See \cite{barbu2009existence} for a proof of strong solutions in the porus medium case  with Lipschitz $\sigma$.
Well-posedness of similar equations with multiplicative noise for dimension $1\leq d \leq 3$ was shown by Barbu, Da Prato and R\"ockner \cite{barbu2008existence,barbu2008some}.

\subsection{Organisation and main contributions}
In Section \ref{section 2 Setup and Kinetic formulation} we setup the problem.
Firstly, in Section \ref{section 2.1 Definition of the noise} we give the definition and assumptions on the nose $\xi^F$ and subsequently in Section \ref{secion 2.2 Definition of stochastic kinetic solution} we define a stochastic kinetic solution.
The setup is analogous to \textcolor{red}{\cite[Sec.~ 2 and 3]{fehrman2024well}}, the main difference being that we incorporate the boundary condition in point two of the definition of stochastic kinetic solution, Definition \ref{definition of stochastic kinetic solution of generalised Dean--Kawasaki equation}. \\
In Section \ref{section 3.1 Assumptions and definition of convolution kernels and cutoff functions} we state the required assumptions for uniqueness as well as the definitions of convolution kernels and cutoff functions. 
Given that we are working on a bounded domain, in order to avoid boundary terms that we do not have control over appearing when integrating by parts, our solution concept Definition \ref{definition of stochastic kinetic solution of generalised Dean--Kawasaki equation} is modified so that our test functions are also compactly supported in space.
When formalising this in the uniqueness proof, this amounts to multiplying the test functions in \cite{fehrman2024well} by a new spatial cutoff $\iota_\gamma$, see Definition \ref{definition of convolution kernels and cutoff functions}.
It is worthwhile to mention that the restriction of working on a $C^2$ domain $U$ arises so that we are able to differentiate the distance function that forms part of the definition of the spatial cutoff, as well as to apply Sobolev embedding theorems later on.
In Section \ref{section 3.3 Uniqueness proof} the uniqueness is proven.
The main novelty in the proof is in the techniques used to bound the new terms arising when the gradient hits the spatial cutoff.\\
In Section \ref{section 4.1 a priori estimates for regularised equation} we begin by proving $L^2(U)$ a priori energy estimates of an appropriately regularised version of \eqref{generalised Dean--Kawasaki Equation Stratonovich} in Proposition \ref{energy estimate of regularised equation proposition}.
We use this to prove further space-time regularity results in the remainder of the section.
 In Section \ref{section4.2 entropy estimate} we prove an entropy estimate for the equation, similar to \textcolor{red}{\cite[Prop.~5.18]{fehrman2024well}}.
A localised version of the entropy estimate is used to prove a bound for the decay of the kinetic measure at zero, a statement needed in the uniqueness proof but proved in Section \ref{section 4.3 decay of kinetic measure at zero} for convenience.   
All of the estimates follow from applying It\^o's formula. 
An important novelty is that we introduce harmonic PDEs (see for example Definitions \ref{definition of functions g and h_M} and \ref{definition of v delta})  with carefully chosen boundary data that ensures we obtain functions which vanish along the boundary when applying It\^o's formula, and so we can integrate by parts without worrying about additional boundary terms.
The boundary regularity assumptions needed for our energy estimates will impose restrictions the boundary data $\bar{f}$ that we can consider.\\
The rest of the existence arguments, including tightness and compactness arguments of Section \ref{section 4.2 existence of solution} follow directly from arguments from \textcolor{red}{\cite[Sec.~5]{fehrman2024well}} so we are brief and simply include the main ideas for completeness.


\section{Setup and Kinetic formulation} \label{section 2 Setup and Kinetic formulation}

\subsection{Definition of the noise}\label{section 2.1 Definition of the noise}
We briefly introduce and state the assumptions needed for the noise term $\xi^F$. 
Both the definition and assumptions are identical to those introduced in \textcolor{red}{\cite[Sec.~2]{fehrman2024well}}.

\begin{definition}[The noise $\xi^F$]\label{definition noise}
Let $F:=\{f_k : U\to\mathbb{R}\}_{k\in\mathbb{N}}$ be a sequence of continuously differentiable functions and $\{B^k:[0,T]\to\mathbb{R}^d\}_{k\in\mathbb{N}}$  a sequence of independent, d-dimensional Brownian motions on a filtered probability space $(\Omega,\mathcal{F},(\mathcal{F}_t)_{t\in[0,T]},\mathbb{P})$.
The noise $\xi^F$, superscripted by $F$ to denote dependence on $\{f_k\}_k$, is defined by 
\[\xi^F: U\times[0,T]\to\mathbb{R}^d,\hspace{20pt}\xi^F(x,t):=\sum_{k=1}^\infty f_k(x) B^k_t.\]
\end{definition}

For ease of notation let us define three quantities related to the spatial component of the noise,
\[ F_1:U\to\mathbb{R}
\hspace{10pt} \text{defined by} \hspace{10pt}F_1(x):=\sum_{k=1}^\infty f_k^2(x);\]
\[ F_2:U\to\mathbb{R}^d
\hspace{10pt} \text{defined by} \hspace{10pt}F_2(x):=\sum_{k=1}^\infty f_k(x)\nabla f_k(x)=\frac{1}{2}\sum_{k=1}^\infty \nabla f_k^2(x);
\]
\[ F_3:U\to\mathbb{R}
\hspace{10pt} \text{defined by} \hspace{10pt}F_3(x):=\sum_{k=1}^\infty |\nabla f_k(x)|^2.\]
We need to make the following assumptions on the noise.

\begin{assumption}[Assumption on noise]\label{Assumption on noise Fi}
Suppose that $F_i$, $i=1,2,3$ are continuous on $U$. Furthermore assume $\nabla\cdot F_2$ is bounded on $U$.
\end{assumption}

Note that by H\"older's inequality, the boundedness of $F_1$ and $F_3$ imply the partial sums of $F_2$ are absolutely convergent. 
The fact that $F_1$ is bounded implies that the noise $\xi^F$ is almost surely finite and bounded in $L^2(U)$.

\subsection{Definition of stochastic kinetic solution}\label{secion 2.2 Definition of stochastic kinetic solution}
We begin by re-writing equation \eqref{generalised Dean--Kawasaki Equation Stratonovich} using It\^o noise.
By Definition \ref{definition noise} of the noise we have
\begin{align*}
     \partial_t\rho&=\Delta\Phi(\rho)-\nabla\cdot(\sigma(\rho)\circ\dot{\xi}^F+ \nu(\rho))=\Delta\Phi(\rho)-\nabla\cdot\nu(\rho)-\sum_{k=1}^\infty\nabla\cdot(\sigma(\rho)f_k\circ dB_t^k).
\end{align*}
Denoting $F_{\sigma,k}(\xi,x):=\sigma(\xi)f_k(x)$ for fixed $x\in U$, the It\^o-to-Stratonovich conversion formula, see \textcolor{red}{\cite[Sec.~3.3]{oksendal2013stochastic}}, the chain rule and product rule give formally that
\begin{align}\label{generalised Dean--Kawasaki equation Ito}
    \partial_t\rho&=\Delta\Phi(\rho)-\nabla\cdot(\sigma(\rho)\dot{\xi}^F+ \nu(\rho))+ \frac{1}{2}\sum_{k=1}^\infty\nabla\cdot\left(\frac{\partial F_{\sigma,k}(\rho,x)}{\partial B^k}\right)\nonumber\\
    &=\Delta\Phi(\rho)-\nabla\cdot(\sigma(\rho)\dot{\xi}^F+ \nu(\rho))+ \frac{1}{2}\sum_{k=1}^\infty\nabla\cdot\left(f_k\sigma'(\rho)\frac{\partial\rho}{\partial B^k}\right)\nonumber\\
    &=\Delta\Phi (\rho) -\nabla\cdot (\sigma(\rho) \dot{\xi}^F + \nu(\rho)) + \frac{1}{2}\sum_{k=1}^\infty\nabla\cdot(f_k\sigma'(\rho)\nabla(f_k\sigma(\rho))) \nonumber\\
     &=\Delta\Phi (\rho) -\nabla\cdot (\sigma(\rho) \dot{\xi}^F + \nu(\rho)) + \frac{1}{2}\nabla\cdot(F_1[\sigma'(\rho)]^2\nabla\rho + \sigma'(\rho)\sigma(\rho)F_2),
\end{align}
which we will equivalently sometimes write in the formal SDE notation as
\[d\rho_t=\Delta\Phi (\rho)\,dt -\nabla\cdot (\sigma(\rho) \,d{\xi}^F + \nu(\rho)\,dt) + \frac{1}{2}\nabla\cdot(F_1[\sigma'(\rho)]^2\nabla\rho + \sigma'(\rho)\sigma(\rho)F_2)\,dt.\]
The below remark illustrates how to interpret integrals involving the divergence of the It\^o noise in \eqref{generalised Dean--Kawasaki equation Ito}.
We use it when interpreting the kinetic equation \eqref{kinetic equation} below.

\begin{remark}
    For $\mathcal{F}_t$-adapted processes $g_t\in L^2(\Omega\times[0,T];L^2(U))$ and $h_t\in L^2(\Omega\times[0,T];H^1(U))$ and any $t\in[0,T]$ we define
    \[\int_0^t\int_U g_s\nabla\cdot(h_s\,d\xi^F)=\sum_{k=1}^\infty\left(\int_0^t\int_U g_s f_k\nabla h_s\cdot\,dB^k_s+\int_0^t\int_U g_s h_s\nabla f_k\cdot\,dB^k_s \right).\]
\end{remark}

\begin{remark}\label{remark highlighting singularity in ito correction}
In the model case $\sigma(\rho)=\rho^{1/2}$, the first correction term arising in the It\^o equation \eqref{generalised Dean--Kawasaki equation Ito} is $\frac{1}{8}\nabla\cdot(F_1\rho^{-1}\nabla\rho)=\frac{1}{8}\nabla\cdot(F_1\nabla \log(\rho))$.
If the solution $\rho$ approaches $0$ at any time the above term is a singular. In fact it is not even clear how we can define the notion of a weak solution since we do not know if $\log(\rho)$ is locally integrable.
\end{remark}

We now turn our attention to providing the kinetic formulation for the generalised Dean--Kawasaki equation \eqref{generalised Dean--Kawasaki equation Ito}. 
By now the kinetic formulation is well understood, and for the Dean--Kawasaki equation the full derivation can be found in \textcolor{red}{\cite[Sec.~3]{fehrman2024well}}. 
We briefly describe the motivation below.
Suppose that for a convex function $S\in C^2(\mathbb{R};\mathbb{R})$ we were interested in properties of functions of the solution $S(\rho)$, where $\rho$ solves equation \eqref{generalised Dean--Kawasaki equation Ito}.
To derive the equation satisfied by $S(\rho)$ one applies It\^o's formula 
\begin{equation*}\label{ito formula for function of solution}
    dS(\rho)=S'(\rho)d\rho + \frac{1}{2}S''(\rho)d\langle\rho\rangle.
\end{equation*} 
However, one cannot do this directly since $\rho$ is not regular enough to apply It\^o's formula, and instead we work on the level of the regularised equation.

\begin{definition}[Regularised equation]\label{definition regularised generalised Dean--Kawasaki equation}
For every $\alpha\in(0,1)$, the regularised generalised Dean--Kawasaki equation $\rho^\alpha$ is \textcolor{red}{a solution to the equation}
\begin{equation}\label{Regularised equation}
 \partial_t\rho^\alpha=\Delta\Phi (\rho^\alpha)+\alpha\Delta \rho^\alpha -\nabla\cdot (\sigma(\rho^\alpha) \dot{\xi}^F + \nu(\rho^\alpha)) + \frac{1}{2}\nabla\cdot(F_1[\sigma'(\rho^\alpha)]^2\nabla\rho^\alpha + \sigma'(\rho^\alpha)\sigma(\rho^\alpha)F_2).
\end{equation} 
\end{definition}

After obtaining an equation for $S(\rho^\alpha)$, one aims to re-write the equation in terms of the kinetic function.

\begin{definition}[Kinetic function]\label{definition of kinetic function}
  For every $\alpha\in(0,1)$, given a non-negative solution $\rho^\alpha$ of equation \eqref{Regularised equation}, define the kinetic function $\chi:U\times [0,T] \times \mathbb{R}\to \{0,1\}$ by
  \[\chi(x,t,\xi)=\mathbbm{1}_{\{0\leq \xi \leq \rho^\alpha(x,t)\}}.\]
   By an abuse of notation we do not write the $\alpha$ dependence on the left hand side.
  In the sequel we will only ever use $\chi$ to denote the kinetic function of solutions to~\eqref{generalised Dean--Kawasaki equation Ito}.
  The kinetic function can equivalently be viewed as $\chi:\mathbb{R}\times\mathbb{R}\to\{0,1\}$, $\chi=\chi(\rho^\alpha,\xi)$.
\end{definition}

In Lions, Perthame and Tadmor \cite{lions1994kinetic} the kinetic function is called the velocity distribution or velocity profile since there they view $\xi$ as a velocity variable.
Here we will adopt the same nomenclature.
By analogy with the theory of gasses, $\chi$ can be called a pseudo-Maxwellian.
It is not then difficult to obtain a distributional equation for the kinetic function by using identities such as, if $S(0)=0$,
\[S(\rho(x,t))=\int_\mathbb{R}S'(\xi)\chi(x,\xi,t)\,d\xi.\]
Finally, in taking the regularisation limit ($\alpha\to0$), one needs to control a term containing $\alpha|\nabla\rho^\alpha|^2$ where we have the competing decay of $\alpha$ and divergence of $|\nabla\rho^\alpha|$ in the limit. 
On the level of the kinetic equation, in the theory of entropy solutions, see for instance \textcolor{red}{\cite[Sec.~2]{chen2003well}}, this is precisely quantified by the kinetic measure $q$.

\begin{definition}[Kinetic measure]\label{definition kinetic measure}
    Let $(\Omega,\mathcal{F},(\mathcal{F}_t)_{t\in[0,\infty)},\mathbb{P})$ be a filtered probability space.\\
    A kinetic measure $q$ is a map from $\Omega$ to the set of non-negative, locally finite measures on $U\times(0,\infty)\times[0,T]$ such that, for every $\psi\in C_c^\infty(U\times(0,\infty))$ we have
    \[(\omega,t)\in(\Omega,[0,T])\to\int_0^t\int_{\mathbb{R}}\int_U\psi(x,\xi)\, q(dx,d\xi,dt)(\omega)=:\int_0^t\int_{\mathbb{R}}\int_U\psi(x,\xi)\, dq(x,\xi,t)(\omega)\]
    is $\mathcal{F}_t$ predictable.
\end{definition}

The resulting equation forms the basis of the definition of stochastic kinetic solution.
We encapsulate some of the properties of the kinetic measure in points three and four of the definition below.

\begin{definition}[Stochastic kinetic solution of \eqref{generalised Dean--Kawasaki equation Ito}]\label{definition of stochastic kinetic solution of generalised Dean--Kawasaki equation}
Let $\rho_0\in L^1(\Omega;L^1(U))$ be a non-negative $\mathcal{F}_0$ measurable initial condition. A stochastic kinetic solution of \eqref{generalised Dean--Kawasaki equation Ito} is a non-negative, almost surely continuous $L^1(U)$ valued $\mathcal{F}_t$-predictable function $\rho\in L^1(\Omega\times[0,T];L^1(U))$ that satisfies
\begin{enumerate}
    \item Integrability of flux: we have 
    \[\sigma(\rho)\in L^2(\Omega;L^2(U\times[0,T])) \hspace{15pt} \text{and}\hspace{15pt} \nu(\rho)\in L^1(\Omega;L^1(U\times[0,T];\mathbb{R}^d)).\]
    \item Boundary condition, local regularity of solution: for each $k\in\mathbb{N}$
     \[[(\Phi(\rho)\wedge k)\vee 1/k]-[(\bar{f}\wedge k)\vee 1/k]\in L^2(\Omega;L^2([0,T];H^{1}_0(U))).\]
    Furthermore, there exists a kinetic measure $q$ that satisfies:
    \item Regularity: almost surely, in the sense of non-negative measures,
    \[\delta_0(\xi - \rho)\Phi'(\xi)|\nabla \rho|^2\leq q \hspace{5pt} \text{on}\hspace{5pt}U\times(0,\infty)\times[0,T].\]
    \item Vanishing at infinity: we have
    \[\lim_{M\to\infty}\mathbb{E}\left(q(U\times[M,M+1]\times[0,T])\right)=0.\]
    \item The kinetic equation:
    for every $\psi\in C_c^\infty(U\times(0,\infty))$ and every $t\in(0,T]$, almost surely,
\begin{align}\label{kinetic equation}
&\int_{\mathbb{R}}\int_{U}\chi(x,\xi,t)\psi(x,\xi)\,dx\,d\xi=\int_{\mathbb{R}}\int_{U}\chi(x,\xi,t=0)\psi(x,\xi)\,dx\,d\xi \nonumber\\
&-\int_0^t \int_{U}\left(\Phi'(\rho)\nabla\rho+\frac{1}{2}F_1[\sigma'(\rho)]^2\nabla\rho +\frac{1}{2} \sigma'(\rho)\sigma(\rho)F_2\right)\cdot\nabla\psi(x,\xi)|_{\xi=\rho}\,dx\,ds\nonumber\\
    &-\int_0^t \int_{\mathbb{R}}\int_{U}d_\xi\psi(x,\xi)\,dq +\frac{1}{2}\int_0^t \int_{U}\left(
    \sigma'(\rho)\sigma(\rho)\nabla\rho\cdot F_2 + \sigma(\rho)^2F_3 \right)\partial_\xi\psi(x,\rho)\,dx\,ds\nonumber\\
    & -\int_0^t \int_{U}\psi(x,\rho)\nabla\cdot (\sigma(\rho) \,d{\xi}^F)\,dx -        \int_0^t \int_{U}\psi(x,\rho)\nabla\cdot\nu(\rho)\,dx\,ds.
\end{align}
\end{enumerate}
\end{definition}

We conclude the section with a few remarks on the above definition.

\begin{remark}\label{remark about notation in kinetic equation}
    In the kinetic equation \eqref{kinetic equation} we write $\nabla\psi(x,\xi)|_{\xi=\rho}$ to emphasise that we take gradient in the first component of $\psi$ rather than the full gradient of $\psi(x,\rho)$.\\
    We abuse notation and write $\partial_\xi\psi(x,\rho)$ to mean $\partial_\xi\psi(x,\xi)|_{\xi=\rho}$.
    \end{remark}

\begin{remark}\label{remark saying rho is locallhy h1}
Since we will assume $\Phi$ is strictly increasing (Assumption \ref{assumptions on coefficients for uniqueness}), the second point implies that locally $\rho\in H^1(U)$.
\end{remark}

\begin{remark}\label{remark about compact support of test function in defn of stochastic kinetic solution}
     In the kinetic equation it is essential that we integrate against test functions $\psi$ that are compactly supported in $U\times(0,\infty)$.
     Firstly, noting Remark \ref{remark highlighting singularity in ito correction}, the compact support in the velocity variable $\xi\in(0,\infty)$ implies that equation \eqref{kinetic equation} needs to hold only away from the zero set of the solution. 
     Secondly, the compact support in space implies that we can integrate by parts and without worrying about boundary terms. 
    The velocity and spatial compactness requirement motivates the introduction cutoff functions in Definition \ref{definition of convolution kernels and cutoff functions} that will be present in many of the choices of test function $\psi$ we make.
\end{remark}

\section{Uniqueness}\label{section 3 Uniqueness}
In Section \ref{section 3.1 Assumptions and definition of convolution kernels and cutoff functions} we begin with some assumptions on the coefficients $\Phi,\nu,\sigma$ of equation \eqref{generalised Dean--Kawasaki equation Ito} needed for uniqueness.
The assumptions are the same as in \cite{fehrman2024well} and allow us to consider the model cases.
We then introduce smoothing kernels and cutoff functions in Definition \ref{definition of convolution kernels and cutoff functions}.\\
The uniqueness is proved in Theorem \ref{main uniqueness theorem}. 
By taking limits of the various convolution kernels and cutoffs in the correct order, the only difference to the torus, \textcolor{red}{\cite[Thm.~4.7]{fehrman2024well}}, is the need to bound new terms arising when the gradient hits the spatial cutoff. 
In the proof we will use a result bounding the 
decay of the kinetic measure at zero, \textcolor{red}{\cite[Prop.~4.6]{fehrman2024well}}, which, for convenience, we prove below in Section \ref{section 4.3 decay of kinetic measure at zero}.

\subsection{Assumptions and definition of convolution kernels and cutoff functions}\label{section 3.1 Assumptions and definition of convolution kernels and cutoff functions}
We begin with the assumptions needed for uniqueness, identical to \textcolor{red}{\cite[Asm.~4.1]{fehrman2024well}}.

\begin{assumption}[Uniqueness assumptions]\label{assumptions on coefficients for uniqueness}
    Suppose $\Phi,\sigma\in C([0,\infty))$ and $\nu\in C([0,\infty);\mathbb{R}^d)$ satisfy the five assumptions:
    \begin{enumerate}
        \item We have $\Phi,\sigma\in C^{1,1}_{loc}([0,\infty))$ and $\nu\in C^1_{loc}([0,\infty);\mathbb{R}^d)$. 
        \item The function $\Phi$ is strictly increasing and starts at $0$: $\Phi(0)=0$ with $\Phi'>0$ on $(0,\infty)$. 
        \item At least linear decay of $\sigma^2$ at $0$: There exists a constant $c\in(0,\infty)$ such that 
        \[\limsup_{\xi\to 0^+}\frac{\sigma^2(\xi)}{\xi}\leq c.\]
        In particular this implies that $\sigma(0)=0$.
        \item Regularity of oscillations of of $\sigma^2$ at infinity: There is a $c\in[1,\infty)$ such that 
        \[\sup_{\xi'\in[0,\xi]}\sigma^2(\xi')\leq  c(1+\xi+\sigma^2(\xi)) \hspace{10pt} \text{for every}\hspace{3pt}\xi\in[0,\infty).\]
        \item Regularity of oscillations of of $\nu$ at infinity: There is a constant $c\in[1,\infty)$ such that 
        \[\sup_{\xi'\in[0,\xi]}|\nu(\xi')|\leq  c(1+\xi+|\nu(\xi)|)\hspace{10pt} \text{for every}\hspace{3pt}\xi\in[0,\infty).\]
    \end{enumerate}
\end{assumption}

We refer to \textcolor{red}{\cite[Rmk.~4.2]{fehrman2024well}} for a comprehensive discussion on the final two assumptions.
The assumptions are satisfied in the case that the functions $\sigma^2, \nu$ are  increasing, or are uniformly continuous, or grow linearly at infinity.
The assumption is more general than any of the above three examples and essentially amounts to a restriction on the growth of the magnitude of oscillations, rather than frequency of oscillations at infinity.\\
We now define the convolution kernels and cutoff functions required in the uniqueness proof.

\begin{definition}[Convolution kernels and cutoff functions]\label{definition of convolution kernels and cutoff functions}
\begin{itemize}
    \item Convolution kernel in space and velocity: for every $\epsilon,\delta\in(0,1)$ let $\kappa^\epsilon_d:U\to[0,\infty)$ and $\kappa_1^\delta:\mathbb{R}\to[0,\infty)$ be standard convolution kernels/ mollifiers of scale $\epsilon$ and $\delta$ on $U$ and $\mathbb{R}$ respectively.
    That is to say, let $\kappa_d\in C_c^\infty(\mathbb{R}^d), \kappa_1\in C_c^\infty(\mathbb{R})$ be non-negative and integrate to one.
    For $\epsilon,\delta\in(0,1)$ define
    \[\kappa^\epsilon_{d}(x)=\frac{1}{\epsilon^d}\kappa_d\left(\frac{x}{\epsilon}\right),\hspace{5pt}\kappa^\delta_{1}(\xi)=\frac{1}{\delta}\kappa_1\left(\frac{\xi}{\delta}\right).\]
Let $\kappa^{\epsilon,\delta}$ be defined by the product
\[\kappa^{\epsilon,\delta}(x,y,\xi,\eta):=\kappa^\epsilon_{d}(x-y)\kappa_1^\delta(\xi-\eta), \hspace{15pt} (x-y,\xi,\eta)\in U\times \mathbb{R}^2.\]
\item Cutoff of small velocity $\xi$:  for every $\beta\in(0,1)$ let $\phi_{\beta}:\mathbb{R}\to[0,1]$ be the unique non-decreasing piecewise linear function that satisfies
    \[\phi_\beta(\xi)=1 \hspace{5pt} \text{if}\hspace{5pt}\xi\geq\beta, \hspace{10pt}
    \phi_\beta(\xi)=0 \hspace{5pt} \text{if}\hspace{5pt}\xi\leq\beta/2, \hspace{10pt}
    \phi_\beta'\textcolor{red}{(\xi)}=\frac{2}{\beta}\mathbbm{1}_{\beta/2\leq\xi\leq\beta}.\]
    \item Cutoff of large velocity $\xi$: for every $M\in\mathbb{N}$ let $\zeta_M:\mathbb{R}\to[0,1]$ be the unique non-increasing piecewise linear function that satisfies
    \[\zeta_M(\xi)=1 \hspace{5pt} \text{if}\hspace{5pt}\xi\leq M, \hspace{10pt}
    \zeta_M(\xi)=0 \hspace{5pt} \text{if}\hspace{5pt} \xi\geq M+1, \hspace{10pt}
    \zeta_M'\textcolor{red}{(\xi)}=-\mathbbm{1}_{M\leq\xi\leq M+1}.\]
    \item Spatial cutoff around boundary: 
    for $\gamma$ sufficiently small (depending on the geometry of $U$), start by introducing the interior regions
    \[U_\gamma:=\{x\in U : d(x,\partial U)\geq \gamma\}\subset U,\hspace{20pt} \partial U_\gamma:=\{x\in U: d(x,\partial U)=\gamma\},\] 
    \textcolor{red}{where $d(x,\partial U)$denotes the usual minimum Euclidean distance from a point to a set.}\\
  The cutoff function is such that it takes the value $1$ in the interior of the domain, $0$ along the boundary, and linearly interpolates between the two. Explicitly we consider the function
    \begin{equation}\label{spatial cutoff}
        \iota_\gamma(x):=\frac{d(x,\partial U)\wedge \gamma}{\gamma}=
\begin{cases}
1, & \text{if} \hspace{3pt} d(x,\partial U)>\gamma,\\
\gamma^{-1}d(x,\partial U), & \text{if} \hspace{3pt} 0\leq d(x,\partial U)\leq\gamma.
\end{cases}
\end{equation}
\end{itemize}
\end{definition}
We will repeatedly use below that for any fixed $\gamma$, one can approximate the cutoff function above by a sequence of functions \textcolor{red}{that are compactly supported in $U$.} 
In this way we may abuse notation and describe $\iota_\gamma$ itself as being compactly supported \textcolor{red}{in $U$}.
\textcolor{red}{We first need to establish how to define the gradient of the spatial cutoff, stated below as a lemma without proof.}

\begin{lemma}[Derivative of spatial cutoff]\label{remark derivative of spatial cutoff}
    To define the spatial derivative of the function $\iota_\gamma$, we will differentiate the distance function.
    We know that the distance function is differentiable if and only if for every $x$ we can find a unique closest point $x^*:=\Pi_{\partial U}(x)$ on the boundary to $x$. 
   Looking at the definition of the cutoff \eqref{spatial cutoff}, we only want to differentiate the distance function for 
   $x\in U\setminus U_\gamma$, so it follows that we only need to assume this property for 
 points $x$ sufficiently close to the boundary.
\textcolor{red}{Since $U$ is a $C^2$ domain, we know that such a neighbourhood of the boundary exists due to \cite[pp.~153]{foote1984regularity}, but in general this will depend on the geometry of $U$.
For fixed domain $U$, define
\[\gamma_U:=max\{\tilde\gamma: \forall x\in\partial U_{\tilde\gamma}, \hspace{3pt} argmin(x,\partial U) \hspace{3pt} \text{is a singleton}\}.\]
For the remainder we consider $\iota_\gamma$ defined for $\gamma\in(0,\gamma_U)$.
} 
    In this range, letting $v_x$ denote the inward pointing unit normal at the boundary to point $x\in U$, with $x^*$ as above, the first derivative is given by
\[\nabla\iota_\gamma(x)=\gamma^{-1}\frac{x-x^*}{|x-x^*|}\mathbbm{1}_{U\setminus U_\gamma}(x):=\gamma^{-1}v_x\mathbbm{1}_{U\setminus U_\gamma}(x),\]
which in particular implies that the size of the first derivative is of the order $\gamma^{-1}$,
\[|\nabla\iota_\gamma(x)|=\gamma^{-1}\mathbbm{1}_{U\setminus U_\gamma}(x).\]
\end{lemma}

\subsection{Uniqueness proof}\label{section 3.3 Uniqueness proof}
Before proving the uniqueness theorem, we need to prove an integration by parts lemma against the kinetic function.
Since we will only deal with test functions that are compactly supported in space, the statement reads the same as the torus, see \textcolor{red}{\cite[Lem.~4.4]{fehrman2024well}}.
The lemma is an equality of vectors which formalises the distributional equality $\nabla_x\chi=\delta_0(\xi-\rho)\nabla\rho$ (an equality that is satisfied when both sides are integrated).
We will later use the lemma with test function being the convolution kernel $\psi= \kappa^{\epsilon,\delta}$.

\begin{lemma}[Integration by parts against kinetic function]\label{lemma integration by parts against kinetic function}
   Let $\psi\in C_c^\infty(U\times(0,\infty))$ be a compactly supported test function (in both arguments) and $\chi$ the kinetic function as defined in Definition \ref{definition kinetic measure}. Then
\[\int_\mathbb{R}\int_{U}\nabla_x\psi(x,\xi)\chi(x,\xi,r)\,dx\,d\xi=-\int_{U}\psi(x,\rho(x,r))\nabla\rho(x,r)\,dx.\]
\end{lemma}

\begin{proof}
 Let $\Psi:U\times(0,\infty)\to \mathbb{R}$ be a function satisfying $\partial_\xi \Psi(x,\xi)=\psi(x,\xi)$, $\Psi(x,0)=0$.
      By the definition of kinetic function, for any $r\in[0,T]$ we have
    \begin{align*}
\int_{U}\int_\mathbb{R}\nabla_x\psi(x,\xi)\chi(x,\xi,r)\,d\xi\,dx&=\int_{U}\int_0^{\rho(x,r)}\nabla_x\psi(x,\xi)\,d\xi\,dx\\
    &=\int_{U}\int_0^{\rho(x,r)}\partial_\xi\nabla_x\Psi(x,\xi)\,d\xi\,dx\\
    &=\int_{U}\nabla_x\Psi(x,\rho(x,r))-\nabla_x\Psi(x,0)\,dx\\
    &=\int_{U}\nabla\Psi(x,\rho(x,r))-\partial_\xi\Psi(x,\rho(x,r))\nabla\rho(x,r)\,dx\\
    &=\int_{\partial U}\Psi(x,\rho(x,r))\cdot\hat{\eta}\,dx
    -\int_{U}\psi(x,\rho(x,r))\nabla\rho(x,r)\,dx,
    \end{align*}
    where the final equality is due to the divergence theorem. 
       Note that the above is an equality of vectors, the first term on the right hand side in the final line denotes a vector where the \textcolor{red}{$i^{\text{th}}$} component is the function $\Psi$ dotted with the \textcolor{red}{$i^{\text{th}}$} direction outward pointing unit normal $\eta_i$, and it vanishes due to the compact support of $\psi$. 
          In the remainder of the paper we write $\hat{\eta}=(\hat{\eta}_i)_{i=1}^d$ to denote the outward pointing unit normal at the boundary $\partial U$.
\end{proof}

We are now in a position to prove the uniqueness of stochastic kinetic solutions of \eqref{generalised Dean--Kawasaki equation Ito}. 
For the proof we will assume the following decay of the kinetic measure at zero, which is proved in Section \ref{section 4.3 decay of kinetic measure at zero}.
  \begin{equation}\label{equation decay of kinetic measure}
\liminf_{\beta\to0}\left(\beta^{-1}q(U\times[\beta/2,\beta]\times[0,T])\right)=0.
  \end{equation}
\textcolor{red}{The proof is via an entropy estimate, and therefore in the theorem below we require the initial condition to live in the entropy space, defined analogously to \cite[Def.~5.16, Asm. 5.17,]{fehrman2024well}, see also Remark 5.15 there.
\begin{definition}[Entropy space] \label{definition entropy space}
    The space of non-negative, $L^1(U)$ functions with finite entropy is the space
    \[Ent(U):=\left\{\rho\in L^1(U): \rho\geq0 \hspace{3pt}\text{almost everywhere, with}\hspace{3pt}\int_U\rho \log(\rho)<\infty\right\}.\]
    We say that a function $\rho:\Omega\to L^1(U)\cap Ent(U)$ is in the space $L^1(\Omega;Ent(U))$ if $\rho$ is $\mathcal{F}_0$ measurable and 
    \[\mathbb{E} \left[\|\rho\|_{L^1(U)}+\int_U\rho \log(\rho)\right]<\infty.\]
\end{definition}}
\begin{theorem}\label{main uniqueness theorem}
Suppose that the coefficients $\Phi,\sigma,\nu$ of equation \eqref{generalised Dean--Kawasaki equation Ito} and the coefficients of noise $\xi^F$ satisfy Assumptions \ref{Assumption on noise Fi}, \ref{assumptions on coefficients for uniqueness}.
Suppose $\rho^1$ and $\rho^2$ are two stochastic kinetic solutions of \eqref{generalised Dean--Kawasaki equation Ito} in the sense of Definition \ref{definition of stochastic kinetic solution of generalised Dean--Kawasaki equation}, \textcolor{red}{with kinetic measures $q^1,q^2$ both satisfying the decay \eqref{equation decay of kinetic measure} and} $\mathcal{F}_0-$measurable initial data $\rho_0^1,\rho_0^2\in L^1(\Omega;Ent(U))$ respectively. 
Then almost surely
\[\sup_{t\in[0,T]}\|\rho^1(\cdot,t)-\rho^2(\cdot,t)\|_{L^1(U)}\leq \|\rho_0^1-\rho_0^2\|_{L^1(U)}.\]
\end{theorem}

\begin{remark}
    \begin{itemize}
        \item Note that the pathwise contraction property in the equation above implies the pathwise continuity of solutions with respect to the initial condition. 
        This is a stronger result than the uniqueness of equation \eqref{generalised Dean--Kawasaki equation Ito}.
        \item The proof follows along the same lines as on the torus, see \textcolor{red}{\cite[Thm.~4.7]{fehrman2024well}}. 
        Hence in the proof below we omit the majority of the bounds that just follow from there, instead focusing our attention on the new terms arising as a result of having to include the spatial cutoff $\iota_\gamma$ as part of the test function. 
        \item In the below proof and throughout the paper we use $c$ to denote a running constant and not specify what it depends on unless it is important.
    \end{itemize}
\end{remark}

\begin{proof}
    Let $\chi^1,\chi^2$ be the kinetic functions of $\rho^1,\rho^2$ respectively.
    For every $\epsilon,\delta\in(0,1)$, $i\in\{1,2\}$ and $\kappa^{\epsilon,\delta}$ as in Definition \ref{definition of convolution kernels and cutoff functions} define the smoothed kinetic functions
\[\chi^{\epsilon,\delta}_{t,i}(y,\eta):=(\chi^i(\cdot,\cdot,t)\ast \kappa^{\epsilon,\delta})(y,\eta), \hspace{5pt} t\in[0,T],y\in U, \eta\in\mathbb{R}.\]
We have by definition of convolution kernels that for $x,y\in U$ and for $\xi,\eta\in\mathbb{R}$
\[\nabla_x \kappa^\epsilon_d(y-x) = -\nabla_y \kappa^\epsilon_d(y-x), \hspace{10pt} \partial_\xi\kappa^\delta_1 (\eta-\xi)=-\partial_\eta\kappa^\delta_1 (\eta-\xi).\]
This implies, as a result of the kinetic equation \eqref{kinetic equation}, that for every $\epsilon,\delta\in(0,1)$ there is a subset of full probability such that we have for every $i\in\{1,2\}$, $t\in[0,T]$ and $(y,\eta)\in U_{2\epsilon}\times(2\delta,\infty)$ such that the convolution kernel is compactly supported
\small
\begin{align}\label{long expression for chi in uniqueness proof}
&\left.\chi^{\epsilon,\delta}_{s,i}(y,\eta)\right|_{s=0}^t:=\left.(\chi^i(\cdot,\cdot,s)\ast \kappa^{\epsilon,\delta})(y,\eta)\right|_{s=0}^t :=\left. \int_{\mathbb{R}}\int_U \chi^i(x,\xi,s)\kappa^{\epsilon,\delta}(y,x,\eta,\xi)\,dx\,d\xi\right|_{s=0}^t\nonumber\\
     &=\nabla_y\cdot\left(\int_0^t \int_{U}\left(\Phi'(\rho^i)\nabla(\rho^i)+\frac{1}{2}F_1[\sigma'(\rho^i)]^2\nabla\rho^i +\frac{1}{2} \sigma'(\rho^i)\sigma(\rho^i)F_2\right)\kappa^{\epsilon,\delta}(y,x,\eta,\rho^i)\,dx\,ds\right)\nonumber\\
    &+d_\eta\left(\int_0^t \int_{\mathbb{R}}\int_{U}\kappa^{\epsilon,\delta}(y,x,\eta,\xi)\,dq^i\right) -\frac{1}{2}\partial_\eta\left(\int_0^t \int_{U}\left(\sigma'(\rho^i)\sigma(\rho^i)\nabla\rho^i\cdot F_2 + \sigma(\rho^i)^2F_3 \right)\kappa^{\epsilon,\delta}(y,x,\eta,\rho^i)\,dx\,ds\right)\nonumber\\
    & -\int_0^t \int_{U}\kappa^{\epsilon,\delta}(x,y,\rho,\eta)\nabla\cdot (\sigma(\rho^i) \,d{\xi}^F)\,dx -        \int_0^t \int_{U}\kappa^{\epsilon,\delta}(x,y,\rho,\eta)\nabla\cdot\nu(\rho^i)\,dx\,ds.
\end{align}
\normalsize
To find an expression of the difference in solutions, we want to deal with a regularised version of
\small
 \begin{align}\label{rigorous equation for difference of solutions in uniqueness proof}
 &\left.\int_{U}|\rho^1(x,s)-\rho^2(x,s)|\,dx\right|_{s=0}^t=\left.\int_{U} \int_{\mathbb{R}}\chi^1(\xi,\rho(x,s))+\chi^2(\xi,\rho(x,s)) -2\chi^1(\xi,\rho(x,s))\chi^2(\xi,\rho(x,s))\,d\xi\,dx\right|_{s=0}^t.
 \end{align}
 \normalsize
  We begin by treating the regularised version of the first two terms on the right hand side of equation \eqref{rigorous equation for difference of solutions in uniqueness proof}.
  Testing equation \eqref{long expression for chi in uniqueness proof} against smooth approximations of the product of cutoff functions $\zeta_M\phi_\beta \iota_\gamma$, which are smooth and compactly supported, and subsequently taking the limit of the approximations yields 
  \small
\begin{align*}
    &\left.\int_{\mathbb{R}}\int_{U}\chi^{\epsilon,\delta}_{s,i}(y,\eta)\zeta_M(\eta)\phi_\beta(\eta) \iota_\gamma(y)\,dy\,d\eta\right|_{s=0}^t=\\
     &-\int_{\mathbb{R}}\int_0^t \int_{U^2}\left(\Phi'(\rho^i)\nabla\rho^i +\frac{1}{2}F_1[\sigma'(\rho^i)]^2\nabla\rho^i +\frac{1}{2} \sigma'(\rho^i)\sigma(\rho^i)F_2\right) \kappa^{\epsilon,\delta}(y,x,\eta,\rho^i)\zeta_M(\eta)\phi_\beta(\eta) \cdot\nabla\iota_\gamma\,dy\,dx\,ds\,d\eta\nonumber\\
    &-\int_0^t \int_{\mathbb{R}^2}\int_{U^2}\kappa^{\epsilon,\delta}(y,x,\eta,\xi)\partial_\eta(\zeta_M(\eta)\phi_\beta(\eta)) \iota_\gamma(y)\,dq^i\,dy\,d\eta\\ 
    &+\frac{1}{2}\int_{\mathbb{R}}\int_0^t \int_{U^2}\left(
    \sigma'(\rho^i)\sigma(\rho^i)\nabla\rho^i\cdot F_2 + \sigma(\rho^i)^2F_3 \right)\kappa^{\epsilon,\delta}(y,x,\eta,\rho^i)\partial_\eta(\zeta_M(\eta)\phi_\beta(\eta)) \iota_\gamma(y)\,dy\,dx\,ds\,d\eta\nonumber\\
    & -\int_{\mathbb{R}}\int_0^t \int_{U^2}\zeta_M(\eta)\phi_\beta(\eta) \iota_\gamma(y)\kappa^{\epsilon,\delta}(y,x,\eta,\rho^i)\nabla\cdot (\sigma(\rho^i) \,d{\xi}^F)\,dy\,dx\,ds\,d\eta\\
    &-        \int_{\mathbb{R}}\int_0^t \int_{U^2}\zeta_M(\eta)\phi_\beta(\eta) \iota_\gamma(y)\kappa^{\epsilon,\delta}(y,x,\eta,\rho^i)\nabla\cdot\nu(\rho^i)\,dy\,dx\,ds\,d\eta.
\end{align*}
\normalsize
For convenience we split this up into three parts,
\begin{align*}    \left.\int_{\mathbb{R}}\int_{U}\chi^{\epsilon,\delta}_{s,i}(y,\eta)\zeta_M(\eta)\phi_\beta(\eta) \iota_\gamma(y)\,dy\,d\eta\right|_{s=0}^t= I^{i,cut}_t+I^{i,mart}_t+I^{i,cons}_t,
\end{align*}
with the cutoff term being the first three lines on the right hand side, the martingale term being the noise term on the fourth line and the conservative term being the term in the final line.
Note in particular that the first terms on the right hand side containing the derivative of the spatial cutoff $\nabla\iota_\gamma$ is a new term compared to the torus case  which a priori diverges like $\gamma^{-1}$.\\
To obtain an expression for the final term (the mixed term) in \eqref{rigorous equation for difference of solutions in uniqueness proof}, we introduce the notation $(x,\xi)\in U \times \mathbb{R}$ for arguments in $\chi^1$ and related quantities and $(x',\xi')\in U \times \mathbb{R}$ for arguments of $\chi^2$ and related quantities. For brevity we also introduce the notation
\[\bar{k}_{s,1}^{\epsilon,\delta}(x,y,\eta):=\kappa^{\epsilon,\delta}(x,y,\eta,\rho^1(x,s)), \hspace{15pt} \bar{k}_{s,2}^{\epsilon,\delta}(x',y,\eta):=\kappa^{\epsilon,\delta}(x',y,\eta,\rho^2(x',s)).\]
In the below computations, since we smoothed the kinetic function, we are allowed to use it as part of an admissible test function.
The stochastic product rules tells us that almost surely we have, for $\beta\in(0,1)$, $M\in\mathbb{N}$, $\gamma$ sufficiently small depending on the domain $U$, $\delta\in(0,\beta/4)$, $\epsilon\in(0,\gamma/4)$,
\small
\begin{align}\label{equation for stochastic product rule for uniqueness proof}   &\left.\int_{\mathbb{R}}\int_{U}\chi^{\epsilon,\delta}_{s,1}(y,\eta) \chi^{\epsilon,\delta}_{s,2}(y,\eta) \zeta_M(\eta)\phi_\beta(\eta) \iota_\gamma(y)\,dy\,d\eta\right|_{s=0}^t\nonumber\\
    &=\int_0^t\int_{\mathbb{R}}\int_U \left(\chi^{\epsilon,\delta}_{s,1}(y,\eta) d\chi^{\epsilon,\delta}_{s,2}(y,\eta) +  \chi^{\epsilon,\delta}_{s,2}(y,\eta) d\chi^{\epsilon,\delta}_{s,1}(y,\eta)+ 
    d\langle\chi^{\epsilon,\delta}_{1},\chi^{\epsilon,\delta}_{1}\rangle_s(y,\eta)\right)\zeta_M(\eta)\phi_\beta(\eta) \iota_\gamma(y)\,dy\,d\eta \nonumber\\
    &=\int_{\mathbb{R}}\int_U \left(\chi^{\epsilon,\delta}_{s,1}(y,\eta) \left.\left[\chi^{\epsilon,\delta}_{s,2}(y,\eta)\right|_{s=0}^t\right] +  \chi^{\epsilon,\delta}_{s,2}(y,\eta) \left.\left[\chi^{\epsilon,\delta}_{s,1}(y,\eta)\right|_{s=0}^t\right]\right.\nonumber\\
    &\hspace{160pt}+\left.
\left.\left[\langle\chi^{\epsilon,\delta}_{1},\chi^{\epsilon,\delta}_{2}\rangle_s(y,\eta)\right|_{s=0}^t\right] \right) \zeta_M(\eta)\phi_\beta(\eta) \iota_\gamma(y)\,dy\,d\eta.
    \end{align}
    \normalsize
Using equation \eqref{long expression for chi in uniqueness proof} we can write the first term on the final line of \eqref{equation for stochastic product rule for uniqueness proof} as
\small
\begin{align*}
    &\int_{\mathbb{R}}\int_U   \chi^{\epsilon,\delta}_{s,1}(y,\eta) \left.\left[\chi^{\epsilon,\delta}_{s,2}(y,\eta)\right]\right|_{s=0}^t\zeta_M(\eta)\phi_\beta(\eta) \iota_\gamma(y)\,dy\,d\eta\\
    &= \int_{\mathbb{R}}\int_U   \zeta_M(\eta)\phi_\beta(\eta) \iota_\gamma(y)\chi^{\epsilon,\delta}_{s,1}(y,\eta)\left[\nabla_y\cdot \left(\int_0^t\int_{U}\left(\Phi'(\rho^2)\nabla\rho^2\right)\bar{k}^{\epsilon,\delta}_{s,2}\,dx'\,ds\right)\nonumber\right.\\
    &+ \nabla_y \cdot\left(\int_0^t\int_{U}\left(\frac{1}{2}F_1[\sigma'(\rho^2)]^2\nabla\rho^2 +\frac{1}{2} \sigma'(\rho^2)\sigma(\rho^2)F_2\right)\bar{k}^{\epsilon,\delta}_{s,2}\,dx'\,ds\right)\nonumber\\
    &+\partial_\eta \left(\int_0^t\int_{\mathbb{R}}\int_{U}\kappa^{\epsilon,\delta}(x',y,\xi,\eta)\,dq^2\right)-\frac{1}{2}\partial_\eta \left(\int_0^t\int_{U}\left(
    \sigma'(\rho^2)\sigma(\rho^2)\nabla\rho^2\cdot F_2 + \sigma(\rho^2)^2F_3 \right)\bar{k}^{\epsilon,\delta}_{s,2}\,dx'\right)\,ds\nonumber\\
    &\left. - \int_0^t\int_{U}\bar{k}^{\epsilon,\delta}_{s,2}\nabla\cdot (\sigma(\rho^2) \,d{\xi}^F)\,dx' -         \int_0^t\int_{U}\bar{k}^{\epsilon,\delta}_{s,2}\nabla\cdot\nu(\rho^2)\,dx'\,ds\right]\,dy\,d\eta.
\end{align*} 
\normalsize
We integrate by parts and move derivatives onto (smooth approximations of) the product\\ $\zeta_M(\eta)\phi_\beta(\eta)\iota_\gamma(y)\chi^{\epsilon,\delta}_{s,1}(y,\eta)$, which are smooth, compactly supported so can be done using classical integration by parts. 
We use the product rule when integrating in $y$ then the integration by parts lemma, Lemma \ref{lemma integration by parts against kinetic function} noting the convolution kernel is compactly supported since $y,\eta\in U_{2\epsilon}\times(2\delta,\infty)$:
\begin{align*}
    \nabla_y \chi^{\epsilon,\delta}_{s,1}(y,\eta)&:=\int_{\mathbb{R}}\int_U \chi^i(x,\xi,s)\nabla_y\kappa^{\epsilon,\delta}(y,x,\eta,\xi)\,dx\,d\xi\\
    &=-\int_{\mathbb{R}}\int_U \chi^i(x,\xi,s)\nabla_x\kappa^{\epsilon,\delta}(y,x,\eta,\xi)\,dx\,d\xi\\
    &=-\int_{U}\bar{k}^{\epsilon,\delta}_{s,i}\nabla\rho^i(x,r)\,dx,
\end{align*}
to obtain the decomposition 
\begin{align*}
   \int_0^t\int_{\mathbb{R}}\int_U  & \chi^{\epsilon,\delta}_{s,1}(y,\eta) d\chi^{\epsilon,\delta}_{s,2}(y,\eta)\zeta_M(\eta)\phi_\beta(\eta) \iota_\gamma(y)\,dy\,d\eta\\
&=I_t^{1,2,err}+I_t^{1,2,meas}+I_t^{1,2,cut}+I_t^{1,2,mart}+I_t^{1,2,cons}.
\end{align*}
Adding the term \[\int_0^t\int_{\mathbb{R}}\int_{U^3}[\Phi'(\rho^1)]^{1/2}[\Phi'(\rho^2)]^{1/2}\nabla\rho^1\cdot\nabla\rho^2\bar{k}^{\epsilon,\delta}_{s,1}\bar{k}^{\epsilon,\delta}_{s,2}\phi_\beta(\eta)\zeta_M(\eta)\iota_\gamma(y)\,dx\,dx'\,dy\,d\eta\,ds\] to the error term and taking it away from the measure term gives for each term separately (note below that we get an extra spatial and an extra real integral due to the definition of convolution in $\chi^{\epsilon,\delta}$):
\small
\begin{align*}
    &I_t^{1,2,err}=\\
    &-\int_0^t\int_{\mathbb{R}}\int_{U^3}  \zeta_M(\eta)\phi_\beta(\eta) \iota_\gamma(y)\Phi'(\rho^2)\nabla\rho^2\cdot\nabla\rho^1\bar{k}^{\epsilon,\delta}_{s,1}\bar{k}^{\epsilon,\delta}_{s,2}\,dx\,dx'\,dy\,d\eta\,ds\\
     &-\int_0^t\int_{\mathbb{R}}\int_{U^3}  \zeta_M(\eta)\phi_\beta(\eta) \iota_\gamma(y)\left(\frac{1}{2}F_1[\sigma'(\rho^2)]^2\nabla\rho^2\cdot\nabla\rho^1 +\frac{1}{2} \sigma'(\rho^2)\sigma(\rho^2)F_2\cdot\nabla\rho^1\right)\bar{k}^{\epsilon,\delta}_{s,1}\bar{k}^{\epsilon,\delta}_{s,2}\,dx\,dx'\,dy\,d\eta\,ds\\
    &-\frac{1}{2}\int_0^t\int_{\mathbb{R}}\int_{U^3}   \zeta_M(\eta)\phi_\beta(\eta) \iota_\gamma(y)\left(
    \sigma'(\rho^2)\sigma(\rho^2)\nabla\rho^2\cdot F_2 + \sigma(\rho^2)^2F_3 \right)\bar{k}^{\epsilon,\delta}_{s,2}\bar{k}^{\epsilon,\delta}_{s,1}\,dx\,dx'\,dy\,d\eta\,ds\\
    &+\int_0^t\int_{\mathbb{R}}\int_{U^3}[\Phi'(\rho^1)]^{1/2}[\Phi'(\rho^2)]^{1/2}\nabla\rho^1\cdot\nabla\rho^2\bar{k}^{\epsilon,\delta}_{s,1}\bar{k}^{\epsilon,\delta}_{s,2}\phi_\beta(\eta)\zeta_M(\eta)\iota_\gamma(y)\,dx\,dx'\,dy\,d\eta\,ds,
\end{align*}
\normalsize
measure term
\begin{align*}
    I_t^{1,2,meas}&= \int_0^t\int_{\mathbb{R}^2}\int_{U^3}   \zeta_M(\eta)\phi_\beta(\eta)\iota_\gamma(y)\kappa^{\epsilon,\delta}(x',y,\xi,\eta) \bar{k}^{\epsilon,\delta}_{s,1}\,dq^2(x',\xi,s)\,dx\,dy\,d\eta\\
    &-\int_0^t\int_{\mathbb{R}}\int_{U^3}[\Phi'(\rho^1)]^{1/2}[\Phi'(\rho^2)]^{1/2}\nabla\rho^1\cdot\nabla\rho^2\bar{k}^{\epsilon,\delta}_{s,1}\bar{k}^{\epsilon,\delta}_{s,2}\phi_\beta(\eta)\zeta_M(\eta)\iota_\gamma(y)\,dx\,dx'\,dy\,d\eta\,ds,
\end{align*}
cutoff term defined by
\begin{align*}
    &I^{1,2,cut}_t = -\int_0^t\int_{\mathbb{R}^2}\int_{U^2}   \partial_\eta(\zeta_M(\eta)\phi_\beta(\eta))\chi^{\epsilon,\delta}_{s,1}(y,\eta) \iota_\gamma(y)\kappa^{\epsilon,\delta}(x',y,\xi,\eta)\,dq^2(x',\xi,s)\,dy\,d\eta\\
    &+\frac{1}{2}\int_0^t\int_{\mathbb{R}}\int_{U^2}   \partial_\eta(\zeta_M(\eta)\phi_\beta(\eta)) \chi^{\epsilon,\delta}_{s,1}(y,\eta) \iota_\gamma(y)\left(
    \sigma'(\rho^2)\sigma(\rho^2)\nabla\rho^2\cdot F_2 + \sigma(\rho^2)^2F_3 \right)\bar{k}^{\epsilon,\delta}_{s,2}\,dx'\,dy\,d\eta\,ds\\
    &-\int_0^t\int_{\mathbb{R}}\int_{U^2}  \zeta_M(\eta)\phi_\beta(\eta) \chi^{\epsilon,\delta}_{s,1}(y,\eta)\nabla_y\iota_\gamma(y)\cdot\left(\Phi'(\rho^2)\nabla\rho^2+\frac{1}{2}F_1[\sigma'(\rho^2)]^2\nabla\rho^2 +\frac{1}{2} \sigma'(\rho^2)\sigma(\rho^2)F_2\right)\\
    &\hspace{250pt}\times \bar{k}^{\epsilon,\delta}_{s,2}\,dx'\,dy\,d\eta\,ds,
\end{align*}
martingale term
\begin{align*}
     I^{1,2,mart}_t&=-\int_0^t\int_{\mathbb{R}}\int_{U^2}   \bar{k}^{\epsilon,\delta}_{s,2}\zeta_M(\eta)\phi_\beta(\eta) \iota_\gamma(y)\chi^{\epsilon,\delta}_{s,1}(y,\eta)\nabla\cdot (\sigma(\rho^2) \,d{\xi}^F)\,dx'\,dy\,d\eta,
\end{align*}
and conservative term
\begin{align*}
    I^{1,2,cons}_t&=-\int_0^t\int_{\mathbb{R}}\int_{U^2}   \bar{k}^{\epsilon,\delta}_{s,2}\zeta_M(\eta)\phi_\beta(\eta) \iota_\gamma(y)\chi^{\epsilon,\delta}_{s,1}(y,\eta)\nabla\cdot\nu(\rho^2)\,dx'\,dy\,d\eta\,ds.
\end{align*}
Note again that the challenge arises from the final line of the cutoff term involving the derivative $\nabla\iota_\gamma$.
An analogous decomposition holds for the second term on the right hand side of \eqref{equation for stochastic product rule for uniqueness proof}.
We denote error, measure, cutoff, martingale and conservative terms of the second term up to time $t\in[0,T]$ by $I^{2,1,\cdot}_t$, where we again artificially add an error term and subtract it from the measure term.
Finally we deal with the quadratic variation term in equation \eqref{equation for stochastic product rule for uniqueness proof}.
Let us begin by noticing, using Definition \ref{definition noise} 
 of the noise $\xi^F$, that formally
\begin{align*}
&d\langle\chi^{\epsilon,\delta}_{\cdot,1},\chi^{\epsilon,\delta}_{\cdot,2}\rangle_s(y,\eta):=d\langle(\chi^1\ast \kappa^{\epsilon,\delta}),(\chi^2\ast \kappa^{\epsilon,\delta})\rangle_s(y,\eta)\\
&=\int_{U^2}\int_{\mathbb{R}^2}d\langle \chi^1,\chi^2\rangle_s \kappa^{\epsilon,\delta}_{s,1}(y,x,\eta,\xi)\kappa^{\epsilon,\delta}_{s,2}(y,x',\eta,\xi')\,d\xi\,d\xi'\,dx\,dx'\\
&=\int_{U^2}\int_{\mathbb{R}^2}\delta_0(\xi-\rho^1)\delta_0(\xi'-\rho^2)\nabla\cdot\left(\sigma(\rho^1)\sum_{k=1}^\infty f_k(x)dB^k_s\right)\nabla\cdot\left(\sigma(\rho^2)\sum_{j=1}^\infty f_j(x')dB^j_s\right)\\
&\hspace{220pt}\times
\kappa^{\epsilon,\delta}_{s,1}(x,y,\xi,\eta)\kappa^{\epsilon,\delta}_{s,2}(x',y,\xi',\eta)\,d\xi\,d\xi'\,dx\,dx'\\
&=\sum_{j,k=1}^\infty\int_{U^2}\left(f_k\sigma'(\rho^1)\nabla\rho^1+\sigma(\rho^1)\nabla f_k\right)\left(f_j\sigma'(\rho^2)\nabla\rho^2+\sigma(\rho^2)\nabla f_j\right)d\langle B^k_\cdot,B^j_\cdot\rangle_s\bar{k}^{\epsilon,\delta}_{s,1}\bar{k}^{\epsilon,\delta}_{s,2}dxdx'\\
&=\sum_{k=1}^\infty\int_{U^2}\left(f_k\sigma'(\rho^1)\nabla\rho^1+\sigma(\rho^1)\nabla f_k\right)\left(f_j\sigma'(\rho^2)\nabla\rho^2+\sigma(\rho^2)\nabla f_j\right)\bar{k}^{\epsilon,\delta}_{s,1}\bar{k}^{\epsilon,\delta}_{s,2}\,dx\,dx' \,ds.
\end{align*}
The above can be made rigorous by integrating against smooth approximations of the product $\phi_\beta \zeta_M \iota_\gamma$ and rather than the multiplication of delta functions, and using the integration by parts lemma, Lemma \ref{lemma integration by parts against kinetic function}. One obtains
    \begin{align*}
        &\int_0^t\int_\mathbb{R}\int_U d\langle\chi^{\epsilon,\delta}_{\cdot,1},\chi^{\epsilon,\delta}_{\cdot,2}\rangle_s(y,\eta)\, \phi_\beta(\eta) \zeta_M(\eta) \iota_\gamma(y)\,dy\,d\eta\\
        &=\sum_{k=1}^{\infty}\int_0^t\int_\mathbb{R}\int_{U^3} \left(f_k\sigma'(\rho^1)\nabla\rho^1+\sigma(\rho^1)\nabla f_k\right)\cdot\left(f_j\sigma'(\rho^2)\nabla\rho^2+\sigma(\rho^2)\nabla f_j\right)\\
        &\hspace{180pt}\times\bar{k}^{\epsilon,\delta}_{s,1}\bar{k}^{\epsilon,\delta}_{s,2} \phi_\beta(\eta) \zeta_M(\eta) \iota_\gamma(y)\,dx\,dx'\,dy\,d\eta\,ds.
    \end{align*} 
Putting this together, it follows from equation \eqref{equation for stochastic product rule for uniqueness proof} and the subsequent above computations that we have the decomposition:
\begin{align}\label{mixed term in uniqueness}
    &\left.\int_{\mathbb{R}}\int_{U}\chi^{\epsilon,\delta}_{s,1}(y,\eta) \chi^{\epsilon,\delta}_{s,2}(y,\eta) \zeta_M(\eta)\phi_\beta(\eta) \iota_\gamma(y)\,dy\,d\eta\right|_{s=0}^t\nonumber\\
    &\hspace{180pt}
    =I_t^{err}+I_t^{meas}+I_t^{mix,cut}+I_t^{mix,mart}+I_t^{mix,cons}.
\end{align}
We put all four terms from the quadratic variation term into the error term and regroup the terms.
Note that the addition of the artificial term in terms $I^{1,2,err}$ and $I^{2,1,err}$ factorises with a term from the quadratic variation and allows it to be controlled, for more detail see \textcolor{red}{\cite[Eqn.~(4.24)]{fehrman2024well}} and subsequent computation there.
Similarly, the measure term just arises from the first two components of \eqref{equation for stochastic product rule for uniqueness proof},
\[I_t^{meas}=I_t^{1,2,meas} + I_t^{2,1,meas}.\]
The contribution from the mixed term (third term of \eqref{rigorous equation for difference of solutions in uniqueness proof}) in the cutoff, martingale and conservative terms just comes from the sum of the first two terms of equation \eqref{equation for stochastic product rule for uniqueness proof} and are denoted as mixed terms above in \eqref{mixed term in uniqueness}.
Finally, we return back to the equation of interest that governs the $L^1$ difference of two solutions, equation \eqref{rigorous equation for difference of solutions in uniqueness proof}. One has the decomposition
\begin{equation}\label{equation for difference of two solutions in uniqueness proof}
    \left.\int_{\mathbb{R}}\int_U \left(\chi^{\epsilon,\delta}_{s,1}+ \chi^{\epsilon,\delta}_{s,2} -2\chi^{\epsilon,\delta}_{s,1}\chi^{\epsilon,\delta}_{s,2}\right) \phi_\beta \zeta_M\iota_\gamma\right|_{s=0}^t= -2 I_t^{err} -2 I_t^{meas} + I_t^{mart} + I_t^{cut}+ I_t^{cons}.
\end{equation}
The error and measure term were defined above and arise solely from the mixed term \eqref{mixed term in uniqueness}, the final term on the left hand side of \eqref{equation for difference of two solutions in uniqueness proof}.
The martingale, cutoff and conservative terms arise from all three terms in the left hand side of equation \eqref{equation for difference of two solutions in uniqueness proof},
\[I_t^{mart,cut,cons}= I_t^{1,mart,cut,cons}+ I_t^{2,mart,cut,cons}-2 I_t^{mix,mart,cut,cons}.\]
Let us deal with each term on the right hand side of \eqref{equation for difference of two solutions in uniqueness proof} separately.\\
\textbf{Measure term.}\\
Firstly, by H\"older's inequality and the regularity property (point three) in Definition \ref{definition of stochastic kinetic solution of generalised Dean--Kawasaki equation} of stochastic kinetic solution, we have
\[I_{t}^{meas}\geq 0. \]
\textbf{Error term.}\\
Following the computations from equation \textcolor{red}{\cite[Eqn.~(4.24)]{fehrman2024well}} to \textcolor{red}{\cite[Eqn.~(4.26)]{fehrman2024well}}, we have
\[\limsup_{\delta\to0}\left(\limsup_{\epsilon\to0}|I_t^{err}|\right)=0.\]
\textbf{Cutoff term.}\\
We have for every $\beta\in(0,1), M\in\mathbb{N}, \gamma>0$ sufficiently small,  $\delta\in(0,\beta/4), \epsilon\in(0,\gamma/4)$ that for every $t\in[0,T]$,
\begin{align}\label{equation cutoff term in uniqueness proof}
    &I_t^{cut}:= I_t^{1,cut} + I_t^{2,cut} -2 I_t^{mix,cut}\nonumber\\
    &=\int_0^t \int_{\mathbb{R}^2}\int_{U^2}\kappa^{\epsilon,\delta}(y,x,\eta,\xi)\partial_\eta(\zeta_M(\eta)\phi_\beta(\eta)) \iota_\gamma(y) (-1+2\chi^{\epsilon,\delta}_{s,2})\,dq^1(x, \xi, s)\,dy\,d\eta\nonumber\\
    &+\frac{1}{2}\int_{\mathbb{R}}\int_0^t \int_{U^2}\left(
    \sigma'(\rho^1)\sigma(\rho^1)\nabla\rho^1\cdot F_2 + \sigma(\rho^1)^2F_3 \right)
    \bar{k}^{\epsilon,\delta}_{s,1}(1-2\chi^{\epsilon,\delta}_{s,2})\partial_\eta(\zeta_M(\eta)\phi_\beta(\eta)) \iota_\gamma(y)\,dy\,dx\,ds\,d\eta\nonumber\\
    &+\int_0^t \int_{\mathbb{R}^2}\int_{U^2}\kappa^{\epsilon,\delta}(y,x',\eta,\xi)\partial_\eta(\zeta_M(\eta)\phi_\beta(\eta)) \iota_\gamma(y)(-1+2\chi^{\epsilon,\delta}_{s,1})\,dq^2(x', \xi, s)\,dy\,d\eta\nonumber\\
    &+\frac{1}{2}\int_{\mathbb{R}}\int_0^t \int_{U^2}\left( \sigma'(\rho^2)\sigma(\rho^2)\nabla\rho^2\cdot F_2 + \sigma(\rho^2)^2F_3 \right)\bar{k}^{\epsilon,\delta}_{s,2}(1-2\chi^{\epsilon,\delta}_{s,1})\partial_\eta(\zeta_M(\eta)\phi_\beta(\eta)) \iota_\gamma(y)\,dy\,dx\,ds\,d\eta\nonumber\\
    &+\int_{\mathbb{R}}\int_0^t \int_{U^2}\left(\Phi'(\rho^1)\nabla\rho^1+\frac{1}{2}F_1[\sigma'(\rho^1)]^2\nabla\rho^1 +\frac{1}{2} \sigma'(\rho^1)\sigma(\rho^1)F_2\right)\nonumber\\
    &\hspace{200pt}\times \bar{k}^{\epsilon,\delta}_{s,1}\zeta_M(\eta)\phi_\beta(\eta) \cdot\nabla\iota_\gamma(y)(-1+2\chi^{\epsilon,\delta}_{s,2})\,dy\,dx\,ds\,d\eta\nonumber\\
    &+\int_{\mathbb{R}}\int_0^t \int_{U^2}\left(\Phi'(\rho^2)\nabla\rho^2 +\frac{1}{2}F_1[\sigma'(\rho^2)]^2\nabla\rho^2+\frac{1}{2} \sigma'(\rho^2)\sigma(\rho^2)F_2\right)\nonumber\\
    & \hspace{200pt}\times\bar{k}^{\epsilon,\delta}_{s,2}\zeta_M(\eta)\phi_\beta(\eta) \cdot\nabla\iota_\gamma(y)(-1+2\chi^{\epsilon,\delta}_{s,1})\,dy\,dx'\,ds\,d\eta.    
\end{align}
Let us begin by bounding the final two lines of $I_t^{cut}$ above, comprising of the new terms involving gradients of the spatial cutoff. 
We take the $\epsilon,\delta\to 0$ limits first and use the distributional equality for $i,j\in\{1,2\}$,
\begin{align}\label{distributional equality for k bar times 1- chi}
\lim_{\epsilon,\delta\to0}\bar{k}^{\epsilon,\delta}_{s,i}(x,y,\eta,\rho)(-1+2\chi^{\epsilon,\delta}_{s,j}(y,\eta))&\to \delta_0(x-y)\delta_0(\eta-\rho^i)sgn(\rho^j-\eta)\nonumber\\
&=\delta_0(x-y)\delta_0(\eta-\rho^i)sgn(\rho^j-\rho^i).
\end{align}
This means that the final two lines of the cutoff \eqref{equation cutoff term in uniqueness proof} can be realised in the $\epsilon,\delta\to0$ as
\begin{align}\label{new terms in uniqueness}
    &\int_0^t \int_{U}\left(\Phi'(\rho^1)\nabla\rho^1+\frac{1}{2}F_1[\sigma'(\rho^1)]^2\nabla\rho^1 +\frac{1}{2} \sigma'(\rho^1)\sigma(\rho^1)F_2\right)\nonumber\\
    &\hspace{140pt}\times\zeta_M(\rho^1)\phi_\beta(\rho^1) \cdot\nabla\iota_\gamma(y)sgn(\rho^2-\rho^1)\,dy\,ds\nonumber\\
    &+\int_0^t \int_{U^2}\left(\Phi'(\rho^2)\nabla\rho^2 +\frac{1}{2}F_1[\sigma'(\rho^2)]^2\nabla\rho^2+\frac{1}{2} \sigma'(\rho^2)\sigma(\rho^2)F_2\right)\nonumber\\
    & \hspace{140pt}\times\zeta_M(\rho^2)\phi_\beta(\rho^2) \cdot\nabla\iota_\gamma(y)sgn(\rho^1-\rho^2)\,dy\,ds.
\end{align}
The terms of the two lines are combined using the fact that $sgn(\rho^1-\rho^2)=-sgn(\rho^2-\rho^1)$.\\
We will deal with terms that have a factor of $\nabla\rho$ and terms that do not separately.
Let us consider the first terms in both the lines of \eqref{new terms in uniqueness}.
Start by defining the function 
$\Phi_{M,\beta}$ to be the unique function such that $\Phi_{M,\beta}(0)=0$ and
\[ \Phi'_{M,\beta}(\xi)=\zeta_M(\xi)\phi_\beta(\xi)\Phi'(\xi)\geq 0.
\]
This says that the function $\Phi_{M,\beta}$ is non-decreasing. 
Hence, with the convention that $sgn(0)=1$, using Lemma \ref{remark derivative of spatial cutoff} to define the spatial derivative of the cutoff, the notation $v_y:=\frac{y-y*}{|y-y*|}$ for the inward pointing unit normal, and the fundamental theorem of calculus, we can show that the difference of the first terms of \eqref{new terms in uniqueness} is non-negative
\begin{align}\label{equation bringing sgn into derivative for increasing function}
     &-\int_0^t \int_{U}\left(\nabla\Phi_{M,\beta}(\rho^2)-\nabla\Phi_{M,\beta}(\rho^1)\right) \cdot\nabla\iota_\gamma(y)\,sgn(\rho^2-\rho^1)\,dy\,ds\nonumber\\
     &=-\int_0^t \int_{U}\left(\nabla\Phi_{M,\beta}(\rho^2)-\nabla\Phi_{M,\beta}(\rho^1)\right)\cdot\nabla\iota_\gamma(y)\,sgn(\Phi_{M,\beta}(\rho^2)-\Phi_{M,\beta}(\rho^1))\,dy\,ds\nonumber\\
     &=-\int_0^t \int_{U}\nabla|\Phi_{M,\beta}(\rho^2)-\Phi_{M,\beta}(\rho^1)| \cdot\nabla\iota_\gamma(y)\,dy\,ds\nonumber\\
     &=-\gamma^{-1}\int_0^t \int_{U\setminus U_\gamma}\nabla|\Phi_{M,\beta}(\rho^2)-\Phi_{M,\beta}(\rho^1)| \cdot v_y\,dy\,ds\nonumber\\
     &=-\gamma^{-1}\int_0^t \int_0^\gamma \int_{\partial U_{z}}\nabla|\Phi_{M,\beta}(\rho^2(y^*+zv_y,s))-\Phi_{M,\beta}(\rho^1(y^*+zv_y,s))| \cdot v_y\,dy\,dz\,ds\nonumber\\
     &=-\gamma^{-1}\int_0^t \int_0^\gamma \int_{\partial U_{z}}\frac{\partial}{\partial z}|\Phi_{M,\beta}(\rho^2(y^*+zv_y,s))-\Phi_{M,\beta}(\rho^1(y^*+zv_y,s))|\,dy\,dz\,ds\nonumber\\
     &=\gamma^{-1}\int_0^t \int_{\partial U}|\Phi_{M,\beta}(\rho^2)-\Phi_{M,\beta}(\rho^1)|- \gamma^{-1}\int_0^t \int_{\partial U_\gamma}|\Phi_{M,\beta}(\rho^2)-\Phi_{M,\beta}(\rho^1)|\,\leq 0.
     \end{align}
     \textcolor{red}{The final line in the above is non-zero, since the first term is an integral over $\partial U$ and the second is an integral over $\partial U_\gamma$. However,}
the final line is signed, since the first term vanishes due to the solutions coinciding on the boundary, and the second term is non-positive because the integrand is clearly non-negative for every fixed $\gamma>0$.
By repeating the same arguments, noting that $\frac{1}{2}F_1(\sigma'(\rho^2))^2\geq 0$, one can conclude that the combination of second terms of \eqref{new terms in uniqueness} are non-positive for every $\gamma>0$. 
For the final term of \eqref{new terms in uniqueness}, we have by Lemma \ref{remark derivative of spatial cutoff} as well as by the boundedness of $F_2$ and $sgn$
\begin{align*}
    &\frac{1}{2} \int_0^t \int_{U} \left(\sigma'(\rho^2)\sigma(\rho^2)\zeta_M(\rho^2)\phi_\beta(\rho^2) - \sigma'(\rho^1)\sigma(\rho^1)\zeta_M(\rho^1)\phi_\beta(\rho^1) \right)sgn(\rho^2-\rho^1)F_2\cdot\nabla\iota_\gamma(y)\,dy\,ds\nonumber\\
    &\leq c \gamma^{-1} \int_0^t \int_{U} \left|\sigma'(\rho^2)\sigma(\rho^2)\zeta_M(\rho^2)\phi_\beta(\rho^2) - \sigma'(\rho^1)\sigma(\rho^1)\zeta_M(\rho^1)\phi_\beta(\rho^1) \right|\mathbbm{1}_{U \setminus U_{\gamma}}(y)\,dy\,ds.
\end{align*}
For ease of notation, for every $M,\beta$ define the function $G_{M,\beta}$ by $G_{M,\beta}(\xi)=\sigma'(\xi)\sigma(\xi)\zeta_M(\xi)\phi_\beta(\xi)$, which is bounded and Lipschitz due to the fact that $\sigma$ and $\sigma'$ are locally Lipschitz, and also it is clearly zero outside $[\beta/2,M+1]$.
For every $y\in U$ sufficiently close to the boundary (i.e. $\gamma$ sufficiently small above) let $y^*=y^*(y)$ denote the unique closest point on the boundary to $y$.
For $i=1,2$ we denote $\rho^i(y^*):=\Phi^{-1}\left(\bar{f}(y^*)\right)$ and $\rho^i_{M,\beta}(y^*):=(\rho^i(y^*)\vee\beta/2)\wedge (M+1)$ for $y^*\in \partial U$.
By adding and subtracting this boundary data, the triangle inequality and Lipschitz property of $G_{M,\beta}$, we have
\begin{align}\label{final cutoff term in uniqueness}
    &c \gamma^{-1} \int_0^t \int_{U\setminus U_{\gamma}} \left|\sigma'(\rho^2)\sigma(\rho^2)\zeta_M(\rho^2)\phi_\beta(\rho^2) - \sigma'(\rho^1)\sigma(\rho^1)\zeta_M(\rho^1)\phi_\beta(\rho^1) \right|\,dy\,ds \nonumber\\
    &\leq c \gamma^{-1} \int_0^t \int_{U\setminus U_{\gamma}} \left(|G_{M,\beta}(\rho^2) - G_{M,\beta}(\rho^2(y^*))|+ |G_{M,\beta}(\rho^1(y^*))-G_{M,\beta}(\rho^1)|\right)\,dy\,ds \nonumber\\
    &\leq c \gamma^{-1} 
    \int_0^t \int_{U\setminus U_{\gamma}} \left(|\rho_{M,\beta}^2(y,s) - \rho^2_{M,\beta}(y^*,s)|+|\rho^1_{M,\beta}(y^*,s)-\rho_{M,\beta}^1(y,s)|\right)\,dy\,ds.
\end{align}
Both of the above terms can be handled in the same way.
Write for $i=1,2$, fixed distance from the boundary $\gamma'\in(0,\gamma)$, fixed time $s\in[0,t]$, and for running constant $c\in(0,\infty)$,
\begin{align*}
    \int_{\partial U_{\gamma'}}|\rho_{M,\beta}^i(y,s)-\rho^i_{M,\beta}(y^*,s)|&=\int_{\partial U_{\gamma'}}\left|\int_0^{\gamma'}\nabla\rho_{M,\beta}^i\left(y^*+zv_y,s\right)\,dz \right| \nonumber\\
    &\leq\int_{U\setminus U_{\gamma'}}|\nabla\rho^i_{M,\beta}(y,s)|\,dy\\
    &\leq |U\setminus U_{\gamma'}|^{1/2}\,\|\nabla\rho^i_{M,\beta}\|_{L^2(U\setminus U_{\gamma'})}\\
    &\leq c(\gamma')^{1/2}\,\|\nabla\rho^i_{M,\beta}\|_{L^2(U\setminus U_{\gamma'})}\\
    &\leq c(\gamma')^{1/2} \,\|\nabla\rho^i_{M,\beta}\|_{L^2(U)},
\end{align*}
where in the final line we made the norm independent of $\gamma'$.
We therefore bound the terms of \eqref{final cutoff term in uniqueness} by
\begin{align*}
         c\, \gamma^{-1}  \int_0^t \int_{U\setminus U_{\gamma}} |\rho^i_{M,\beta}(y,s) - \rho^i_{M,\beta}(y^*)|\,dy\,ds&=c\, \gamma^{-1}  \int_0^t \int_0^\gamma \int_{\partial U_{\gamma'}} |\rho^i_{M,\beta}(y,s) - \rho^i_{M,\beta}(y^*)|\,dy\,d\gamma'\,ds\\
        &\leq c\, \gamma^{-1}  \int_0^t \|\nabla\rho^i_{M,\beta}\|_{L^2(U)}\left(\int_0^\gamma  \sqrt{\gamma'}\,d\gamma'\right)\,ds\\
     &\leq c\, \gamma^{1/2}\|\nabla\rho^i_{M,\beta}\|_{L^1([0,t];L^2(U))},
\end{align*}
which converges to $0$ as $\gamma\to0$ for fixed $M,\beta$.
Therefore we managed to show that the final two lines of the cutoff \eqref{equation cutoff term in uniqueness proof} consisting of the new terms are non-positive in the $\gamma\to0$ limit.
The below is an important point about the remaining cutoff terms, as well as martingale and conservative terms that are yet to be analysed.\\
Following the computations of the uniqueness proof \textcolor{red}{\cite[Thm.~4.7]{fehrman2024well}}, the terms involving $F_2$ in the second and fourth lines of the cutoff term \eqref{equation cutoff term in uniqueness proof},
    as well as martingale and conservative terms are handled by integration by parts.
    More precisely the $\epsilon,\delta$ limits are taken first, then integration by parts is performed before taking $M,\beta$ limits.
On the bounded domain, the presence of a spatial cutoff $\iota_\gamma$ leads to additional terms with factors $\nabla\iota_\gamma$ when integrating by parts.
We emphasise that these new terms can be bounded in a similar way to the third term in \eqref{new terms in uniqueness}.
We will explain how to bound these terms, which will allow us to conclude.\\
Firstly, to bound the terms involving $F_2$ in the second and fourth line of \eqref{equation cutoff term in uniqueness proof}, analogously to above, in the $\epsilon,\delta\to0$ limit we use the distributional equality \eqref{distributional equality for k bar times 1- chi}, followed by the equality $sgn(\rho^1-\rho^2)=-sgn(\rho^2-\rho^1)$ and finally the 
product rule to evaluate the derivative of the cutoffs to get
\begin{align*}
      &\lim_{\epsilon,\delta\to0}\left(\frac{1}{2}\int_{\mathbb{R}}\int_0^t \int_{U^2}
    \sigma'(\rho^1)\sigma(\rho^1)\nabla\rho^1\cdot F_2
    \bar{k}^{\epsilon,\delta}_{s,1}(1-2\chi^{\epsilon,\delta}_{s,2})\partial_\eta(\zeta_M(\eta)\phi_\beta(\eta)) \,\iota_\gamma(y)\,dy\,dx\,ds\,d\eta\right.\nonumber\\
    &\hspace{30pt}\left.+\frac{1}{2}\int_{\mathbb{R}}\int_0^t \int_{U^2} \sigma'(\rho^2)\sigma(\rho^2)\nabla\rho^2\cdot F_2 \bar{k}^{\epsilon,\delta}_{s,2}(1-2\chi^{\epsilon,\delta}_{s,1})\partial_\eta(\zeta_M(\eta)\phi_\beta(\eta)) \,\iota_\gamma(y)\,dy\,dx\,ds\,d\eta\right)\nonumber\\
    &=\frac{1}{2}\int_0^t\int_U\left(\sigma'(\rho^2)\sigma(\rho^2)\nabla\rho^2-\sigma'(\rho^1)\sigma(\rho^1)\nabla\rho^1\right)\cdot F_2 \partial_\eta(\zeta_M(\eta)\phi_\beta(\eta)) sgn(\rho^2-\rho^1)\,\iota_\gamma\,dx\,ds\nonumber\\
    &=\frac{1}{4}\int_0^t\int_U\left(\nabla\sigma^2(\rho^2)-\nabla\sigma^2(\rho^1)\right)\cdot F_2 \left(\mathbbm{1}_{M<\rho<M+1}+\beta^{-1}\mathbbm{1}_{\beta/2<\rho<\beta}\right) sgn(\rho^2-\rho^1)\,\iota_\gamma\,dx\,ds\nonumber\\
    &=\frac{1}{4\beta}\int_0^t\int_U\nabla\left(\sigma^2\left(\left(\rho^2\vee\beta\right)\wedge\beta/2\right)-\sigma^2\left(\left(\rho^1\vee\beta\right)\wedge\beta/2\right)\right)\cdot F_2  sgn(\rho^2-\rho^1)\,\iota_\gamma\,dx\,ds\nonumber\\
    &\hspace{20pt}+\frac{1}{4}\int_0^t\int_U\nabla\left(\sigma^2\left(\left(\rho^2\vee(M+1)\right)\wedge M\right)-\sigma^2\left(\left(\rho^1\vee(M+1)\right)\wedge M\right)\right)\cdot F_2  sgn(\rho^2-\rho^1)\,\iota_\gamma\,dx\,ds.
\end{align*}
Note that $\sigma^2$ is not necessarily increasing, so we \textcolor{red}{cannot} use that $sgn(\rho^2-\rho^1)=sgn(\sigma^2(\rho^2)-\sigma^2(\rho^1))$ as in \eqref{equation bringing sgn into derivative for increasing function}.
Instead we smooth out the sign function by writing $sgn^\delta:=sgn \ast \kappa^\delta_1$ for $\delta\in(0,1)$ before integrating by parts.
The terms involving $M$ are handled in the same way as the terms involving $\beta$, so we illustrate how to handle the $\beta$ term.
For convenience introduce the shorthand notation $\rho^i_\beta(x,t):=\left(\rho^i(x,t)\vee\beta\right)\wedge\beta/2$ for $i=1,2$, $(x,t)\in U\times[0,T]$. The terms involving $\beta$ can be written as
\begin{align*}
    &\frac{1}{4\beta}\int_0^t\int_U\nabla\left(\sigma^2(\rho^2_\beta)-\sigma^2(\rho^1_\beta)\right)\cdot F_2  sgn(\rho^2-\rho^1)\iota_\gamma\,dx\,ds\\
    &=\lim_{\delta\to0}\frac{1}{4\beta}\int_0^t\int_U\nabla\left(\sigma^2(\rho^2_\beta)-\sigma^2(\rho^1_\beta)\right)\cdot F_2  sgn^\delta(\rho^2-\rho^1)\iota_\gamma\,dx\,ds\\
    &=-\lim_{\delta\to0}\frac{1}{4\beta}\int_0^t\int_U\left(\sigma^2(\rho^2_\beta)-\sigma^2(\rho^1_\beta)\right)\nabla\cdot F_2  sgn^\delta(\rho^2-\rho^1)\iota_\gamma\,dx\,ds\\
    &\hspace{20pt}-\lim_{\delta\to0}\frac{1}{4\beta}\int_0^t\int_U\left(\sigma^2(\rho^2_\beta)-\sigma^2(\rho^1_\beta)\right) \nabla(\rho^2-\rho^1)\cdot F_2  (sgn^\delta)'(\rho^2-\rho^1)\iota_\gamma\,dx\,ds\\
    &\hspace{20pt}-\lim_{\delta\to0}\frac{1}{4\beta}\int_0^t\int_U\left(\sigma^2(\rho^2_\beta)-\sigma^2(\rho^1_\beta)\right)\nabla\iota_\gamma\cdot F_2\, sgn^\delta(\rho^2-\rho^1)\,dx\,ds.
\end{align*}
For the first two terms we can directly take the $\gamma\to0$ limit since the cutoff converges point-wise to $1$  and so they can be handled analogously as on the torus, see the computation leading from equation \textcolor{red}{\cite[Eqn.~(4.28)]{fehrman2024well}} to \textcolor{red}{\cite[Eqn.~(4.31)]{fehrman2024well}}. 
We only need to consider the final term involving the gradient of spatial cutoff, which we can bound in an analogous way to the final term in equation \eqref{new terms in uniqueness} after realising that $\sigma^2(\cdot\vee\beta\wedge\beta/2)$ is Lipschitz for every fixed $\beta>0$, $F_2$ and $sgn^\delta$ are bounded, and noting we take $\gamma\to0$ limit before $M$ and $\beta$ limits.\\
To show that the first and third lines of the cutoff term \eqref{equation cutoff term in uniqueness proof} vanish we use precisely the decays of the kinetic measure at zero and infinity.
Putting \eqref{equation cutoff term in uniqueness proof} and subsequent computations together, we conclude
\[\lim_{M\to\infty,\beta\to0}\lim_{\gamma\to0}\lim_{\epsilon,\delta\to0}I_t^{cut}\leq0.\]
\textbf{Martingale term.}\\
Following the analysis from \textcolor{red}{\cite[Eqn.~(4.27)]{fehrman2024well}} to \textcolor{red}{\cite[Eqn.~(4.36)]{fehrman2024well}}, we have that, for the unique function $\Theta_{M,\beta}:[0,\infty)\to[0,\infty)$ defined by $\Theta_{M,\beta}(0)=0$, $\Theta_{M,\beta}'(\xi)=\phi_\beta(\xi)\zeta_M(\xi)\sigma'(\xi)$
\begin{align*}
    \lim_{\epsilon,\delta\to0}I_t^{mart}&=\int_0^t\int_U sgn(\rho^2-\rho^1)\iota_\gamma\nabla\cdot\left(\left(\Theta_{M,\beta}(\rho^1)-\Theta_{M,\beta}(\rho^2)\right)d\xi^F\right)\\
    &+\int_0^t\int_U sgn(\rho^2-\rho^1)\iota_\gamma\left(\phi_\beta(\rho^1)\zeta_M(\rho^1)\sigma(\rho^1)-\Theta_{M,\beta}(\rho^1)\right)\nabla\cdot d\xi^F\\
    &-\int_0^t\int_U sgn(\rho^2-\rho^1)\iota_\gamma\left(\phi_\beta(\rho^2)\zeta_M(\rho^2)\sigma(\rho^2)-\Theta_{M,\beta}(\rho^2)\right)\nabla\cdot d\xi^F.
\end{align*}
The final two terms can be handled directly as in \cite{fehrman2024well} by first directly taking the $\gamma\to0$ limit.
For the first term, using again the regularisation of the sign function and subsequently integrating by parts gives two terms. 
When the derivative hits the regularised sign, the term can be handled in the same way as \cite{fehrman2024well} after immediately taking the $\gamma\to0$ limit, and we get a new term when the derivative hits the spatial cutoff,
\[\int_0^t\int_U sgn^\delta(\rho^2-\rho^1)\left(\Theta_{M,\beta}(\rho^1)-\Theta_{M,\beta}(\rho^2)\right)\nabla\iota_\gamma\cdot d\xi^F.\]
By using the Burkholder-Davis-Gundy inequality, see \textcolor{red}{\cite[Thm.~4.1]{revuz2013continuous}}, the boundedness of the $sgn^\delta$ as well as the bound on the derivative of the spatial cutoff given in Lemma \ref{remark derivative of spatial cutoff}, we have
\begin{align*}
    &\mathbb{E}\left(\sup_{t\in{[0,T]}}\left|\int_0^t\int_U sgn^\delta(\rho^2-\rho^1)\left(\Theta_{M,\beta}(\rho^1)-\Theta_{M,\beta}(\rho^2)\right)\nabla\iota_\gamma\cdot d\xi^F\right|\right)\\
    &\leq c\mathbb{E}\left(\int_0^T\left(\int_U sgn^\delta(\rho^2-\rho^1)\left(\Theta_{M,\beta}(\rho^1)-\Theta_{M,\beta}(\rho^2)\right)|\nabla\iota_\gamma|\right)^2\,ds\right)^{1/2}\\
    &\leq c\gamma^{-1}\,\mathbb{E}\left(\int_0^T\left(\int_U \left|(\Theta_{M,\beta}(\rho^1)-\Theta_{M,\beta}(\rho^2)\right|\mathbbm{1}_{U \setminus U_\gamma}\right)^2\,ds\right)^{1/2}.
\end{align*}
\textcolor{red}{For every fixed $M,\beta,$ we have that $\Theta_{M,\beta}$ is Lipschitz, so this term can now just be handled in the same way as the final term of \eqref{new terms in uniqueness}. 
The key point is again that we obtain a factor of $\gamma^{3/2}$, which more than compensates the blow up of $\gamma^{-1}$.}
\begin{align*}
    &\mathbb{E}\left(\sup_{t\in{[0,T]}}\left|\int_0^t\int_U sgn(\rho^2-\rho^1)\left(\Theta_{M,\beta}(\rho^1)-\Theta_{M,\beta}(\rho^2)\right)\nabla\iota_\gamma\cdot d\xi^F\right|\right)\\
    &\leq c\gamma^{-1}\,\mathbb{E}\left(\int_0^T\left(\sum_{i=1}^2\|\nabla\rho^i_{M,\beta}\|_{L^2(U)} \gamma^{3/2}\right)^2\,ds\right)^{1/2}\\
    &\leq c\gamma^{1/2}\sum_{i=1}^2\mathbb{E}\|\nabla\rho^i_{M,\beta}\|_{L^2([0,T];L^2(U))}.
\end{align*}
The final term converges to zero in the $\gamma\to0$ limit for fixed $M,\beta$. This implies almost sure convergence of this new term along a sub-sequence, that is
\[\lim_{M\to\infty,\beta\to0}\left(\lim_{\gamma\to0}\left(\lim_{\epsilon,\delta\to0}I_t^{mart}\right)\right)=0.\]
\textbf{Conservative term.}\\
The same arguments as the martingale term, in particular equation \textcolor{red}{\cite[Eqn.~(4.31)]{fehrman2024well}} to \textcolor{red}{\cite[Eqn.~(4.35)]{fehrman2024well}}, also apply here.
After again taking the $\epsilon,\delta$ limit we have
\begin{align*}
    \lim_{\epsilon,\delta\to0}I_t^{cons}=\int_0^t\int_U sgn(\rho^2-\rho^1)\iota_\gamma\left(\nabla\cdot\nu(\rho^1)\zeta_M\phi_\beta(\rho^1)-\nabla\cdot\nu(\rho^2)\zeta_M\phi_\beta(\rho^2)\right).
\end{align*}
First define the Lipschitz vector valued function $\Psi_{M,\beta,\nu}=(\Psi_{M,\beta,\nu,i})_{i=1}^d$ such that for $\nu=(\nu_i)_{i=1}^d$
\[\Psi_{M,\beta,\nu,i}(0)=0, \hspace{10pt} \frac{\partial\Psi_{M,\beta,\nu,i}}{\partial x_i}(\xi)=\frac{\partial^2\nu_i}{\partial x_i^2}(\xi)\phi_\beta(\xi)\zeta_M(\xi).\]

Then with a similar re-writing as the martingale term, we have
\begin{align*}
     \lim_{\epsilon,\delta\to0}I_t^{cons}&=\int_0^t\int_U  sgn(\rho^2-\rho^1)\iota_\gamma\nabla\cdot\left(\Psi_{M,\beta,\nu}(\rho^1)-\Psi_{M,\beta,\nu}(\rho^2)\right)\\
     &+\int_0^t\int_U sgn(\rho^2-\rho^1)\iota_\gamma\left(\nabla\cdot\nu(\rho^1)\zeta_M\phi_\beta(\rho^1)-\nabla\cdot\Psi_{M,\beta,\nu}(\rho^1)\right)\\
     &-\int_0^t\int_U sgn(\rho^2-\rho^1)\iota_\gamma\left(\nabla\cdot\nu(\rho^2)\zeta_M\phi_\beta(\rho^2)-\nabla\cdot\Psi_{M,\beta,\nu}(\rho^2)\right).
\end{align*}
The final two terms can be handled analogously to the final two terms appearing in the martingale term after taking a $\gamma\to0$ limit.
The first term can be handled using integration by parts, analogous to the martingale term as shown above.
Here we use the $L^1(U)$ integrability of $\nu(\rho)$ and the final assumption of Assumption \ref{assumptions on coefficients for uniqueness} to apply the dominated convergence theorem, and conclude that along subsequences
\[\lim_{M\to\infty,\beta\to0}\left(\lim_{\gamma\to0}\left(\lim_{\epsilon,\delta\to0}I_t^{cons}\right)\right)=0.\]
\
\textbf{Conclusion.}\\
Putting everything together we get from \eqref{equation for difference of two solutions in uniqueness proof} and subsequent handling of each term, that there are random sub-sequences $\epsilon,\delta,\beta,\gamma \to0$, $M\to\infty$ along which
\begin{align*}
    \left.\int_\mathbb{R}\int_{U}|\chi^1_s-\chi^2_s|^2\right|_{s=0}^t&=\lim_{\beta\to0,M\to\infty}\lim_{\gamma\to0}\lim_{\epsilon,\delta\to0}\left.\int_\mathbb{R}\int_{U}|\chi^{\epsilon,\delta}_{s,1}-\chi^{\epsilon,\delta}_{s,2}|^2\phi_\beta\zeta_M\iota_\gamma\right|_{s=0}^t\\
    &=\lim_{\beta\to0,M\to\infty}\lim_{\gamma\to0}\lim_{\epsilon,\delta\to0}\left(-2 I_t^{err} -2 I_t^{meas} + I_t^{mart} + I_t^{cut}+ I_t^{cons}+I_t^{bound}\right)\leq 0.
\end{align*}
This implies that
\begin{align*}
    &\int_U|\rho^1(\cdot,t)-\rho^2(\cdot,t)|=\int_\mathbb{R}\int_{U}|\chi^1_t-\chi^2_t|^2\leq \int_\mathbb{R}\int_{U}|\chi^1_0-\chi^2_0|^2 =\int_U|\rho^1_0-\rho^2_0|.
\end{align*}
\end{proof}

\section{Existence}\label{section 4 Existence}
In this section, we construct a stochastic kinetic solution of the generalised Dean--Kawasaki equation \eqref{generalised Dean--Kawasaki equation Ito} in the sense of Definition \ref{definition of stochastic kinetic solution of generalised Dean--Kawasaki equation}.
The existence consists of three steps.
Firstly, in Section \ref{section 4.1 a priori estimates for regularised equation} we will prove $L^2(U)$ energy estimates for a suitable regularised version of \eqref{generalised Dean--Kawasaki equation Ito}.
To do this we will use the regularised equation \eqref{Regularised equation} and smooth the non-linearity $\sigma$.
We then proceed to prove further space-time regularity results for weak solutions of the regularised equation.\\
As an aside, in \ref{section4.2 entropy estimate} is dedicated to proving an entropy estimate for the equation and a localised version of this argument helps us to prove a statement about the Kinetic measure at zero in Section \ref{section 4.3 decay of kinetic measure at zero}.
For all of the energy estimates we will need to introduce harmonic PDE's (for example in Definitions \ref{definition of functions g and h_M} and \ref{definition of v delta}) that allow certain functions to vanish along the boundary when applying It\^o's formula.\\
The second and third steps are analogous to \textcolor{red}{\cite[Sec.~5]{fehrman2024well}}, and so we are brief in their presentation.
For the second step, in the first part of Section \ref{section 4.2 existence of solution} we show that there exists a stochastic kinetic solution to the regularised equation.
Since all the coefficients are regular, the proof follows by a projection argument, where the projected system is just a finite dimensional system of stochastic differential equations and so has a unique strong solution.
The final step, illustrated in the latter half of Section \ref{section 4.2 existence of solution}, requires us to pass to the limit in the regularisation.

\subsection{A priori estimates for the regularised equation}\label{section 4.1 a priori estimates for regularised equation}

 In this section we start with some definitions as well as stating the relevant assumptions needed for uniqueness. 
 The main result of the section is the subsequent $L^2(U)$ energy estimate of the regularised equation in Proposition \ref{energy estimate of regularised equation proposition}.
 We conclude by proving some higher order spatial regularity of the solution in Lemma \ref{lemma on higher order spatial regularity of regularised solution} and higher order space-time regularity of the solution cutoff away from zero in Proposition \ref{proposition higher order time regularity of solution cut away from zero set} that will be essential in the tightness arguments.\\
The estimates will be proven with respect to the  regularised equation \eqref{Regularised equation}, which we recall is given by
\[
 \partial_t\rho^\alpha=\Delta\Phi (\rho^\alpha)+\alpha\Delta \rho^\alpha -\nabla\cdot (\sigma(\rho^\alpha) \dot{\xi}^F + \nu(\rho^\alpha)) + \frac{1}{2}\nabla\cdot(F_1[\sigma'(\rho^\alpha)]^2\nabla\rho^\alpha + \sigma'(\rho^\alpha)\sigma(\rho^\alpha)F_2),\]
 defined for $\alpha\in(0,1)$ with boundary condition $\Phi(\rho^\alpha)=\bar{f}$.
 For ease of notation when proving estimates about the regularised equation we denote the regularised equation by $\rho$ instead of $\rho^\alpha$.
Motivated with trying to write the factor $\Phi'(\rho)|\nabla\rho|^2$ in the regularised kinetic measure as a single gradient term, we introduce the below auxiliary function.

\begin{definition}[Auxiliary function corresponding to $\Phi$]\label{definition of Auxiliary function corresponding to phi}
    Let $\Phi$ be any $C([0,\infty))\cap C^1_{loc}(0,\infty)$ function that is strictly increasing with $\Phi(0)=0$.
    Define $\Theta_{\Phi}\in C([0,\infty))\cap C^1_{loc}(0,\infty)$ to be the unique function satisfying $\Theta_{\Phi}(0)=0$ and \[\Theta_{\Phi}'(\xi)=(\Phi'(\xi))^{1/2}.\]
\end{definition}

We now state the assumptions needed for existence.
Some of the assumptions overlap with the uniqueness assumptions, Assumption \ref{assumptions on coefficients for uniqueness}.

\begin{assumption}[Existence assumptions]\label{existence assumptions}
Suppose $\Phi,\sigma\in C([0,\infty))\cap C^1_{loc}((0,\infty))$ and $\nu\in C([0,\infty);\mathbb{R}^d)\cap C^1_{loc}((0,\infty);\mathbb{R}^d)$ satisfy the six assumptions in the same spirit as \textcolor{red}{\cite[Asm.~5.2]{fehrman2024well}}:
\begin{enumerate}
\item We have $\Phi(0)=\sigma(0)=0$ and $\Phi'> 0$ on $(0,\infty)$.
    \item There exists constants $m\in(0,\infty), c\in(0,\infty)$ such that for every $\xi\in[0,\infty)$
    \[\Phi(\xi)\leq c(1+\xi^m).\]
    \item There is a constant $c\in(0,\infty)$ such that,  we have for every $\xi\in[0,\infty)$
    \[\Phi'(\xi)\leq c(1+\xi+\Phi(\xi)).\]
    \item For $\Theta_{\Phi}$ defined in Definition \ref{definition of Auxiliary function corresponding to phi}, we have that either for constants $c\in(0,\infty)$ and $\theta\in[0,1/2]$ that for every $\xi\in(0,\infty)$,
    \begin{equation}\label{1 of 2 point 4 existence assumption}
        (\Theta'_{\Phi}(\xi))^{-1}:=\Phi'(\xi)^{-1/2}\leq c\xi^\theta,
    \end{equation}
    or we have constants
    $c\in(0,\infty)$, $q\in[1,\infty)$ such that for every $\xi,\eta\in[0,\infty)$
    \begin{equation} \label{2 of 2 point 4 existence assumption}
        |\xi-\eta|^q\leq c|\Theta_{\Phi}(\xi)-\Theta_{\Phi}(\eta)|^2.
    \end{equation}
    \item For a constant $c\in(0,\infty)$ and every $\xi\in[0,\infty)$ we have
    \begin{equation*}
        \sigma^2(\xi)\leq  c(1+\xi+\Phi(\xi)).
    \end{equation*}
    \item For each $\delta\in(0,1)$ there is a constant $c_\delta\in(0,\infty)$ such that for every $\xi\in(\delta,\infty)$,
    \[\frac{(\sigma'(\xi))^4}{\Phi'(\xi)}\leq c_\delta (1+\xi+\Phi(\xi)).\]

          \vspace{10pt}
        Furthermore we have the additional new assumptions, used to bound new boundary terms that arise in the estimates: 
        \item For a constant $c\in(0,\infty)$, \textcolor{red}{$m\in(0,\infty)$ as in point 2} and every $\xi\in(0,\infty)$ we have
        \[\Theta_\Phi(\xi)\geq c(\xi^{\frac{m+1}{2}}-1).\]
         \item There is a constant $c\in(0,\infty)$ such that for every $\xi\in[0,\infty)$
    \[|\nu(\xi)|^2+ |\sigma(\xi)\sigma'(\xi)|^2\leq c(1+\xi+\Phi(\xi)).\]
    \item The anti-derivative of $\nu$, defined element-wise for $i=1,\hdots,d$, by $\Theta_{\nu,i}(0)=0$, $\Theta'_{\nu,i}(\xi)=\nu_i(\xi)$ satisfies for $i=1\hdots,d$ that $\Theta_{\nu,i}(\Phi^{-1}(\bar{f}))\cdot\hat{\eta}_i\in L^1(\partial U)$, where $\bar{f}$ is the boundary condition for the regularised equation.
    \item We have $\sigma^2(\Phi^{-1}(\bar{f}))\in L^1(\partial U)$.
        \item Either $\bar{f}$ is constant, or for the unique function $\Psi_\sigma$ defined by $\Psi_\sigma(1)=0$, $\Psi'_\sigma(\xi)=F_1[\sigma'(\xi)]^2$, we have
        \[\bar{f}\in L^2(\partial U), \hspace{10pt}\Phi^{-1}(\bar{f})\in L^2(\partial U), \hspace{10pt} \Psi_\sigma(\Phi^{-1}(\bar{f}))\in L^2(\partial U).\]
\end{enumerate}
\end{assumption}

\begin{remark}
    The fourth point in the above assumption enables us to consider $\Phi(\xi)=\xi^m$ for every $m\in(0,\infty)$.
    If $m<1$ then $\Phi'(\xi)^{-1/2}=m^{-1/2}\xi^{\frac{1-m}{2}}$ so satisfies \eqref{1 of 2 point 4 existence assumption}.
    On the other hand, if $m\geq 1$ then by Remark \ref{remark on precise identification of phi} we have $c|\Theta_\Phi(\xi)-\Theta_\Phi(\eta)|^2=cm|\xi^{\frac{m+1}{2}}-\eta^{\frac{m+1}{2}}|^2$ so satisfies \eqref{2 of 2 point 4 existence assumption} with $q=m+1$.
\end{remark}

\begin{remark}[Growth assumption on $\Theta_\Phi$]\label{remark on precise identification of phi}
The lower bound on the growth of $\Theta_\Phi$ in point seven of the above assumption is essential for obtaining $L^k(U)$ estimates of the solution in Proposition \ref{proposition estimating powers of l1 norm of solution} below.
Formally, one should think that in order to obtain $L^k(U)$ estimates for the solution, one needs to harness the additional regularity in equation \eqref{Regularised equation} coming from when the Laplacian acts on $\Phi(\rho)=\rho^m$.
In the model case $\Phi'(\xi)=m\xi^{m-1}$ and so the assumption is satisfied since
 \[\Theta_{\Phi}(\xi)=m^{1/2}\int_0^\xi \eta^{(m-1)/2}\,d\eta=\frac{2m^{1/2}}{m+1}\xi^{(m+1)/2}.\]
\end{remark}

\begin{remark}[Constraint on boundary conditions]\label{remark on constraints of BC}
     For the final assumption in the final point, we have in the model case that $\Psi_\sigma(\xi)\sim \log(\xi)$.
     Hence the assumption still enables us to handle boundary data $\bar{f}$ that is constant or uniformly bounded away from zero.
\end{remark}
We introduce two PDEs that will allow us to avoid boundary terms in our energy estimate, Proposition \ref{energy estimate of regularised equation proposition}, the specific form of the PDEs is discussed in Remark \ref{remark choice of pdes}.

 \begin{definition}[The PDEs $g$ and $h_M$]\label{definition of functions g and h_M}
    Let $\bar{f}$ be the boundary condition in the regularised equation \eqref{Regularised equation}.
    Define $g:U\to\mathbb{R}$ to be the below harmonic function that captures the regularity of the solution on the boundary
     \begin{equation*}
    \begin{cases}
            -\Delta g = 0 & \text{on} \hspace{5pt}U,\\
    g=\Phi^{-1}(\bar{f})& \text{on} \hspace{5pt}\partial U.
    \end{cases}
\end{equation*}
For the function $S_M:[0,\infty)\to[0,\infty)$ satisfying $S_M''(\xi)=\mathbbm{1}_{M_1<\xi<M_2}(\xi)$ with $0<M_1<M_2$, define $h_M:U\to\mathbb{R}$ to be the below harmonic function satisfying 
\begin{equation*}
    \begin{cases}
            -\Delta h_M = 0 & \text{on} \hspace{5pt}U,\\
    h_M=S_M'(\Phi^{-1}(\bar{f}))& \text{on} \hspace{5pt}\partial U.
    \end{cases}
\end{equation*}
 \end{definition}

Note that $S'_M$ is $[0,M_2-M_1]-$valued, and therefore by the maximum principle $h_M$ is bounded by $M_2-M_1$. 
 Before stating the assumption we will need on $g$, we first state an important remark motivating the harmonic PDEs introduced in this paper.
 
\begin{remark}[Choice of harmonic PDEs.]\label{remark choice of pdes}
    The boundary data of the PDEs are chosen to ensure certain functions vanish along the boundary when applying It\^o's formula, and consequently allows integration by parts without picking up boundary terms. 
    They are all also chosen to be harmonic in $U$, and this is for several reasons, illustrated using the example of PDE $g$ in equation \eqref{equation for L2 estimate} in the energy estimate of Proposition \ref{energy estimate of regularised equation proposition} below.
    \begin{enumerate}
        \item Looking at the final line of \eqref{equation for L2 estimate}, several terms include integrands with a factor of $\nabla g$ multiplied by another gradient term.
        To bound these terms, integrating by parts and moving the derivative onto the $\nabla g$, since $g$ is harmonic we have that they turn into boundary terms. That is, for $f:U\times[0,T]\to\mathbb{R}$ such that the below integrals are well defined, we have
        \[\int_U \nabla f(x,t)\cdot\nabla g(x)=-\int_U f(x,t)\Delta g(x)+\int_{\partial U}f(x^*,t) \frac{\partial g}{\partial \hat{\eta}}(x^*)=\int_{\partial U}f(x^*,t) \frac{\partial g}{\partial \hat{\eta}}(x^*).\]
        \item For the remaining terms in the final line of equation \eqref{equation for L2 estimate}, to handle integrands where $\nabla g$ is multiplied by terms without another gradient we use either H\"older's or Young's inequality. 
        In both of these cases we have to bound the $L^2(U)$ norm of $\nabla g$. 
        Since $g$ is harmonic, testing the PDE against the solution and integrating by parts allows us to estimate this quantity,
        \[\int_U |\nabla g|^2=\int_{\partial U}g \frac{\partial g}{\partial \hat{\eta}}. \]
        \end{enumerate}
\end{remark}
\textcolor{red}{The terms arising in the above two points can then be handled using the below Theorem, see \textcolor{red}{\cite[Thm.~2.4(iii)]{fabes1978potential}} for a more general statement and proof.
\begin{theorem} \label{thm: definiton of H^1 boundary norm} Suppose $U\in C^1(\mathbb{R}^d)$ is a bounded domain with $\mathbb{R}^d\backslash \bar U$ is connected.
Let $g$ be as in Definition \ref{definition of functions g and h_M}.
Then there exists a constant $c\in(0,\infty)$ independent of $\bar{f}$ so that
        \[\left\|\frac{\partial g}{\partial \hat{\eta}}\right\|_{L^p(\partial U)}\leq c\|\Phi^{-1}(\bar{f})\|_{H^p(\partial U)},\]
where we define the $H^p(\partial U)$ norm as the definition of the $L^p_1(\partial D)$ norm in \textcolor{red}{\cite[pp.~176]{fabes1978potential}}.
Briefly, the norm is defined by
\[\|f\|_{H^p(\partial U)}:=\|f\|_{L^p(\partial U)}+\sum_j\|\nabla \widetilde{\psi_jf}\|_{L^p(\mathbb{R}^{d-1})},\]
where $\{B_j\}_{j=1}^l$ is a fixed covering of $\partial U$, $\{\psi_j\}$ is a partition of unity subordinate to this cover, and we define $\widetilde{\psi f}$ is a version of the product $\psi f$ but with respect to a different co-ordinate system.
Note that if we use a different covering and partition of unity, this will give rise to a norm equivalent to the one defined above.
\end{theorem}
See \textcolor{red}{\cite[Lem.~ 4]{bella2018liouville}} for a similar result on the cube $U=[0,1]^d$.\\
Importantly, the above theorem allows us to express our below bounds in terms of norms of the boundary data $\bar{f}$ directly, rather than in terms of norms the corresponding PDEs.
}
 
With the above remark in mind, we state below assumptions.

\begin{assumption}[Assumption on $g$] \label{assumption on g and h_M}
   Either the boundary data $\bar{f}$ is constant, or using the PDE for $g$ in Definition \ref{definition of functions g and h_M}, we have $\Phi^{-1}(\bar{f})\in H^1(\partial U)$.
   Further assume that either the boundary data $\bar{f}$ is bounded, or we have $\Theta_\Phi(g)\in H^1(U)$.   
\end{assumption}

We once again need to deal with singularities from the It\^o-to-Stratonovich conversion and ensure the integrals below are well defined, so need the below assumption smoothing $\sigma$.
We state it as a separate assumption since it is not necessary and will subsequently be dispensed of via an approximation argument in Lemma \ref{lemma on approximating sigma by smooth function}.

\begin{assumption}\label{assumption about smoothing sigma in existence}
Let $\sigma\in C([0,\infty))\cap C^\infty((0,\infty))$ with $\sigma(0)=0$ and $\sigma'\in C_c^\infty([0,\infty))$.
\end{assumption}

For the regularised equation with smoothed $\sigma$ we can make sense of a weak solution:

\begin{definition}[Weak solution of regularised equation \eqref{Regularised equation}] \label{definition of weak Solution of regularised equation}
    Let $\xi^F, \Phi,\sigma$ and $\nu$ satisfy Assumptions \ref{Assumption on noise Fi}, \ref{existence assumptions} and \ref{assumption about smoothing sigma in existence}. 
    Let further $\rho_0\in L^2(\Omega;L^2(U))$ be non-negative and $\mathcal{F}_0$ measurable, \textcolor{red}{and let $g$ be as in Definition \ref{definition of functions g and h_M}}. 
   A weak solution $\rho$  of \eqref{Regularised equation} with initial condition $\rho_0$ is a continuous $L^2(U)$ valued, non-negative $\mathcal{F}_t$-predictable process such that $\rho-g \in L^2([0,T];H^1_0(U))$ and $\Theta_{\Phi}(\rho)\in L^2([0,T];H^1(U))$, and such that for every  $\psi\in C_c^\infty(U)$, 
     almost surely for every $t\in[0,T]$,
    \begin{align*}
        \int_U \rho(x,t)\psi(x)\,dx&= \int_U \rho_0\psi\,dx- \int_0^t\int_U \Phi'(\rho)\nabla\rho\cdot\nabla\psi\,dx\,dt- \alpha\int_0^t\int_U \nabla\rho\cdot\nabla\psi\,dx\,dt\\
        &+\int_0^t\int_U \nu(\rho)\cdot\nabla\psi\,dx\,dt
        +\int_0^t\int_U \sigma(\rho)\nabla\psi\cdot\,d\xi^F\,dt\\
        &-\frac{1}{2}\int_0^t\int_U F_1(\sigma'(\rho))^2\nabla\rho\cdot\nabla\psi\,dx\,dt-\frac{1}{2}\int_0^t\int_U \sigma(\rho)\sigma'(\rho)F_2\cdot\nabla\psi\,dx\,dt.
    \end{align*}
\end{definition}

We arrive to the first result of the section, and obtain a bound for powers of the solution by comparing it with the function $\Theta_\Phi$ defined in Definition \ref{definition of Auxiliary function corresponding to phi}.
Such an estimate is required because we are not on the torus so do not have preservation of mass, and nor do we have enough regularity to quantify the flux along the boundary, see Remark \ref{remark outlining why we can't follow the proof of FG on bounding kinetic equation at zero} below.

\begin{proposition}[Estimate for $L^1_tL^k_x$ norm of the solution]\label{proposition estimating powers of l1 norm of solution}
    Suppose that $\Phi$ satisfies the polynomial growth condition \textcolor{red}{with constant $m\in(0,\infty)$} as in point two of Assumption \ref{existence assumptions}, and the PDE $g$ as in Definition \ref{definition of functions g and h_M} satisfies Assumption \ref{assumption on g and h_M}.
    If $\rho$ is a weak solution of the regularised equation \eqref{Regularised equation} in the sense of Definition \ref{definition of weak Solution of regularised equation}, then one has, for every $\epsilon>0$ and $k\in(0,m+1)$, a constant $c\in(0,\infty)$ depending only on $k,m$ and $U$,
    \[\int_0^T\int_U|\rho|^k\leq c\left(\frac{T}{\epsilon}+\epsilon T\|\Theta_\Phi(g)\|_{H^1(U)}+\epsilon\int_0^T\int_U\left|\nabla\Theta_\Phi(\rho)\right|^2\right).\]
    If the boundary data $\bar{f}$ is bounded by constant $K$, we obtain the simplified bound with constant again depending only on $\textcolor{red}{K},k,m$ and $U$, 
    \[\int_0^T\int_U|\rho|^k\leq c\left(\frac{T}{\epsilon}+\epsilon T+\epsilon\int_0^T\int_U\left|\nabla\Theta_\Phi(\rho) \right|^2\right).\]
\end{proposition}

\begin{remark}\label{remark outlining estimate for norm of solution}
    Using Jensen's inequality, the above proposition also provides a bound for powers of the $L^1(U)$ norm of the solution, since
    \[\int_0^T\left(\int_U|\rho|\right)^k=|U|^k\int_0^T\left(\frac{1}{|U|}\int_U|\rho|\right)^k\leq |U|^{k-1}\int_0^T\int_U|\rho|^k.\]
    The proposition implies that for the full range $m\in(0,\infty)$ we at least have an $L^1_tL^1_x$ estimate for the solution, since $k\in(0,m+1)$.
    Apart from the bounding the $L^1_x$ norm, the result is useful when $m,k>1$, since if $k<1$ we can just use interpolation to obtain 
    \[\int_0^T\int_U|\rho|^k\leq \int_0^T\int_U \left(1+|\rho|\right).\]
\end{remark}

\begin{proof}
    The bound on $\Theta_\Phi$ given in point seven of Assumption \ref{existence assumptions} gives, for a running constant $c$ depending on $k,U$ and $m$,
  \begin{align*}
      \int_0^T\int_U|\rho|^k\leq cT+ c\int_0^T\int_U|\Theta_\Phi(\rho)|^{\frac{2k}{m+1}}.
  \end{align*}
 Note that here the exponent in  the integrand is strictly less than two.
 Applying Young's inequality with $\epsilon$ and exponent $\frac{m+1}{k}>1$ and Jensen's inequality then gives
\begin{align*}
      \int_0^T\int_U|\Theta_\Phi(\rho)|^{\frac{2k}{m+1}}&\leq \int_0^T  \frac{(m+1-k)1^{\frac{m+1}{m+1-k}}}{\epsilon(m+1)} +\frac{\epsilon k}{m+1}\left(\int_U\Theta_\Phi(\rho)^\frac{2k}{m+1}\right)^{\frac{m+1}{k}}\\
      &\leq \frac{cT}{\epsilon}+ c\epsilon\int_0^T\int_U\Theta_\Phi(\rho)^2.
  \end{align*}
Using the trivial inequality $a^2\leq 2(a-b)^2+b^2$ with $b=\Theta_\Phi(g)$, where the PDE $g$ is as in Definition \ref{definition of functions g and h_M}, and subsequently applying Poincar\'e inequality gives the first claim,
  \begin{align}\label{inequality in new bound of rho to power}
    \int_0^T\left(\int_U|\rho|\right)^k&\leq \frac{c T}{\epsilon}+c\epsilon\int_0^T\int_U\left(\Theta_\Phi(\rho) -\Theta_\Phi(g)\right)^2+c\epsilon\int_0^T\int_U\Theta_\Phi(g)^2\\
      &\leq \frac{cT}{\epsilon}+c\epsilon\int_0^T\int_U\left|\nabla\left(\Theta_\Phi(\rho) -\Theta_\Phi(g)\right)\right|^2+c\epsilon T\int_U\Theta_\Phi(g)^2\nonumber\\
      &\leq \frac{cT}{\epsilon}+c\epsilon\int_0^T\int_U\left|\nabla\Theta_\Phi(\rho)\right|^2+c\epsilon T\|\Theta_\Phi(g)\|_{H^1(U)}.\nonumber
  \end{align}
  For the second claim, if the boundary data is bounded by constant $K$, then we can use the comparison principle which tells us that
  \[ \int_0^T\left(\int_U|\rho|\right)^k\leq  \int_0^T\left(\int_U|\tilde{\rho}|\right)^k,\]
  where $\tilde{\rho}$ solves the same equation as $\rho$ but with boundary condition $K$.
  Repeating the steps above to bound the norm on the right hand side, we see that when we arrive to \eqref{inequality in new bound of rho to power} we add and subtract the constant \[\Theta_\Phi(K)=\frac{2m^{1/2}}{m+1}K^{\frac{m+1}{2}},\] 
  which can subsequently be absorbed into the running constant.
  This gives the second claim.
  \end{proof}

\begin{proposition}[Energy estimate for solution of regularised equation with smooth, bounded $\sigma$]\label{energy estimate of regularised equation proposition}
    Let $\xi^F, \Phi,\sigma$ and $\nu$ satisfy Assumptions \ref{Assumption on noise Fi}, \ref{existence assumptions} and \ref{assumption about smoothing sigma in existence},  and let $\alpha\in(0,1)$, $T\in[1,\infty)$.
    Let further $\rho_0\in L^2(\Omega;L^2(U))$ be non-negative and $\mathcal{F}_0$ measurable, $\Theta_{\nu,i}$ and $\Psi_\sigma$ be defined as in Assumption \ref{existence assumptions} and $g,h_M$ be the PDEs defined in Definition \ref{definition of functions g and h_M} satisfying Assumption \ref{assumption on g and h_M}. 
    If $\rho$ is a weak solution of the regularised equation \eqref{Regularised equation} in the sense of Definition \ref{definition of weak Solution of regularised equation} then one has for $c\in(0,\infty)$ independent of $\alpha$,  the estimate
    \begin{align}\label{first energy estimate of regularised solution}
    &\frac{1}{2}\sup_{t\in[0,T]}\mathbb{E}\left(\int_U(\rho(x,t)-g(x))^2\right) + \mathbb{E}\left(\int_U\int_0^T|\nabla\Theta_{\Phi}(\rho)|^2\right) +\mathbb{E}\left(\alpha\int_U\int_0^T|\nabla \rho|^2\right) \nonumber\\
    &\leq \frac{1}{2}\|\rho_0-g\|^2_{L^2(U)}+ T\left(c+c\|\Theta_\Phi(g)\|_{H^1(U)}+\sum_{i=1}^d\int_
    {\partial U}\Theta_{\nu,i}(\Phi^{-1}(\bar{f}))\cdot\hat{\eta}_i+ c\|\sigma^2(\Phi^{-1}(\bar{f}))\|_{L^1(\partial U)}\right)\nonumber\\
    &+T\left(\|\bar{f}\|^2_{L^2(\partial U)}+\|\Psi_\sigma(\Phi^{-1}(\bar{f}))\|^2_{L^2(\partial U)}+c(1+\alpha)\|\Phi^{-1}(\bar{f})\|^2_{H^1(\partial U)}\right).
\end{align}
    Let the function $S_M$  and the PDE $h_M$ be defined as in Definition \ref{definition of functions g and h_M}, and for $i=1,\hdots,d$ define the functions $\Theta_{M,\nu,i}:\mathbb{R}\to\mathbb{R}$ by $\Theta_{M,\nu,i}(0)=0$ and $\Theta'_{M,\nu,i}(\xi)=\mathbbm{1}_{M_1<\xi<M_2}\nu_i(\xi)$.
    Then we have for every $M_1<M_2\in(0,\infty)$ the existence of constant $c\in(0,\infty)$ independent of both $M_1,M_2$ such that
    \small
         \begin{align}\label{second energy estimate of regularised solution}
     &\mathbb{E}\int_U\int_0^T \mathbbm{1}_{M_1<\rho<M_2} \left(\Phi'(\rho)|\nabla\rho|^2+\alpha|\nabla\rho|^2\right)\leq \mathbb{E}\int_U (\rho_0(x)-M_1)_++\|h_M\|_{L^2(U)}\mathbb{E}\|\rho_T\|_{L^2(U)}\nonumber\\
    &+c\left\|S'_M(\Phi^{-1}(\bar{f}))\right\|_{H^1(\partial U)}\mathbb{E}\int_0^T\int_U\left|\nabla\Theta_\Phi(\rho)\right|^2+ c\mathbb{E}\int_U\int_0^T\mathbbm{1}_{\rho\geq M_1}\sigma^2(\rho\wedge M_2)\nonumber\\
    &+T\sum_{i=1}^d\int_
    {\partial U}\Theta_{M,\nu,i}(\Phi^{-1}(\bar{f}))\cdot\hat{\eta}
    +cT\int_{\partial U}\left(\sigma^2((\Phi^{-1}(\bar{f})\wedge M_2)\vee M_1)-\sigma^2(M_1)\right) \nonumber\\
    &+cT \left\|S'_M(\Phi^{-1}(\bar{f}))\right\|_{H^1(\partial U)}\left(T+\|\Theta_\Phi(g)\|_{H^1(U)}+\left\|\bar{f}\right\|_{L^2(\partial U)}+\alpha\left\|\Phi^{-1}(\bar{f})\right\|_{L^2(\partial U)}+\left\|\Psi_\sigma(\Phi^{-1}(\bar{f}))\right\|_{L^2(\partial U)}\right).
\end{align}
\normalsize
\end{proposition}

Note that the final two terms on the left hand side of \eqref{first energy estimate of regularised solution} are the approximate kinetic measures $q^\alpha$, and the bound in \eqref{second energy estimate of regularised solution} is a bound on the approximate kinetic measures when $\rho$ is bounded away from zero and infinity.
The right hand side of \eqref{second energy estimate of regularised solution} is written deliberately to emphasise that it converges to zero as $M_1,M_2\to\infty$ due to the fact that $h_M\equiv0$ on $\bar{U}$ as $M_1\to\infty$.

\begin{proof}
    To prove the first claim, applying It\^o's formula to the solution minus the PDE with correct boundary data so that the difference vanishes along the boundary, gives
\begin{align}\label{ito formula in energy estimate}
\left.\int_U(\rho(x,s)-g(x))^2\,dx\right|_{s=0}^t&=\int_U\int_0^td(\rho-g)^2=\int_U\int_0^t 2(\rho-g)\,d(\rho-g)+ \frac{1}{2}2(\rho-g)^{0}d\langle\rho-g\rangle_s\,dx.
\end{align}
For the first term on the right hand side, noting $g$ does not depend on time so $d(\rho-g)=d\rho$ and integrating by parts gives
\begin{align*}
&\int_U\int_0^t 2(\rho-g)\,d(\rho-g)\,dx\\
&=-2\int_U\int_0^t \nabla(\rho-g)\cdot\left(\nabla\Phi(\rho)+\alpha\nabla\rho-\sigma(\rho)\dot{\xi}^F -\nu(\rho)+\frac{1}{2}F_1[\sigma'(\rho)]^2\nabla\rho + \frac{1}{2}\sigma'(\rho)\sigma(\rho)F_2\right)\,ds\,dx.
\end{align*}
The It\^o correction can be easily evaluated considering $d\langle\rho-g\rangle=d\langle\rho\rangle$ and using the definition of the noise,
\begin{align*}
    &\int_U\int_0^t d\langle\rho-g\rangle_s\,dx= \int_U\int_0^t F_1(\sigma'(\rho))^2|\nabla\rho|^2+2\sigma\sigma'(\rho)F_2\cdot\nabla\rho+F_3\sigma^2(\rho)\,ds\,dx.
\end{align*}
Putting these two together, we get from \eqref{ito formula in energy estimate} that
\begin{align}\label{equation for L2 estimate}
&\frac{1}{2}\left.\int_U(\rho(x,s)-g(x))^2\,dx\right|_{s=0}^t\nonumber\\
&=-\int_U\int_0^t \left(\Phi'(\rho)|\nabla\rho|^2+\alpha|\nabla\rho|^2-\sigma(\rho)\nabla\rho\cdot\dot{\xi}^F -\nabla\rho\cdot\nu(\rho) - \frac{1}{2}\sigma'(\rho)\sigma(\rho)\nabla\rho\cdot F_2-\frac{1}{2}F_3\sigma^2(\rho)\right)\nonumber\\
&+\int_U\int_0^t\nabla g\cdot\left(\nabla\Phi(\rho)+\alpha\nabla\rho-\sigma(\rho)\dot{\xi}^F -\nu(\rho)+\frac{1}{2}F_1[\sigma'(\rho)]^2\nabla\rho + \frac{1}{2}\sigma'(\rho)\sigma(\rho)F_2\right).
\end{align}
Let us consider each term of \eqref{equation for L2 estimate} in turn. 
The first two terms on the right hand side are precisely those in the estimate, so we move them to the left hand side.
The noise term in both lines vanish after taking an expectation.
The fourth term remains, but we can write it as a boundary integral in the following way. 
As in Assumption \ref{existence assumptions}, for $i=1,\hdots,d$ let $\Theta_{\nu,i}$ denote the anti-derivative of $\nu_i$. We then obtain
\begin{align*}
    \int_U\int_0^t \nabla\rho\cdot\nu(\rho)\,ds\,dx=\sum_{i=1}^d\int_U\int_0^t \partial_i\left(\Theta_{\nu,i}(\rho)\right)\,ds\,dx=\sum_{i=1}^d\int_
    {\partial U}\int_0^t\Theta_{\nu,i}(\Phi^{-1}(\bar{f}))\cdot\hat{\eta_i}\,ds\,dx.
\end{align*}
The fifth term can be bounded by integration by parts, noting either $\nabla\cdot F_2=0$ or it is bounded and we can bound $\sigma^2$ by $c(1+\rho+\Phi(\rho))$, and the fact that $F_2\cdot\hat{\eta}$ is bounded. 
We have
\begin{align*}
    \frac{1}{4}\int_U\int_0^t  \nabla\sigma^2(\rho)\cdot F_2\,ds\,dx&=-\frac{1}{4}\int_U\int_0^t \sigma^2(\rho)\nabla\cdot F_2\,ds\,dx + \frac{1}{4}\int_{\partial U}\int_0^t  \sigma^2(\Phi^{-1}(\bar{f})) F_2\cdot\hat{\eta}\,ds\,dx\\
    &\leq c\int_U\int_0^t(1+\rho+\Phi(\rho))\,ds\,dx + c\int_{\partial U}\int_0^t  \sigma^2(\Phi^{-1}(\bar{f})) \,ds\,dx.
\end{align*}
The final term on the second line of \eqref{equation for L2 estimate} can be bounded by the first term in the above inequality after noting that $F_3$ is bounded and once more bounding $\sigma^2(\rho)$ by $c(1+\rho+\Phi(\rho))$.\\
The terms on the final line of \eqref{equation for L2 estimate} involving $\nabla g$ would all vanish if the boundary condition was constant.
Otherwise they can be bounded precisely as described in points one and two of Remark \ref{remark choice of pdes}.
The first, second and fifth terms involving another gradient are reduced to boundary terms using integration by parts and noting $g$ is harmonic, whereas the remaining terms are handled using Young's inequality and the bounds in point eight of Assumption \ref{existence assumptions}.
Putting everything in \eqref{equation for L2 estimate} together we have
\begin{align*}
    &\frac{1}{2}\mathbb{E}\left(\left.\int_U(\rho(x,s)-g(x))^2\,dx\right|_{s=0}^t\right) + \mathbb{E}\left(\int_U\int_0^t|\nabla\Theta_{\Phi}(\rho)|^2\right) +\mathbb{E}\left(\alpha\int_U\int_0^t|\nabla \rho|^2\right) \\
    &\leq c\left(t+\mathbb{E}\int_0^t\int_U|\rho|+\mathbb{E}\int_0^t\int_U\Phi(\rho)\right)+t\sum_{i=1}^d\int_
    {\partial U}\Theta_{\nu,i}(\Phi^{-1}(\bar{f}))\cdot\hat{\eta}_i+ ct\int_{\partial U} \sigma^2(\Phi^{-1}(\bar{f}))\\
    &+\alpha t\int_{\partial U} \Phi^{-1}(\bar{f}) \frac{\partial g}{\partial \hat{\eta}}+t\int_{\partial U}\bar{f}\frac{\partial g}{\partial \hat{\eta}}
+t\int_{\partial U}\Psi_\sigma(\Phi^{-1}(\bar{f}))\frac{\partial g}{\partial \hat{\eta}}
+t\|\nabla g\|^2_{L^2(U)},
\end{align*}
where $\Psi_\sigma$ was defined in point eleven of Assumption \ref{existence assumptions}.
We used the fact that the boundary terms in the final two lines are deterministic and do not depend on time, and are all well defined due to the final three assumptions in Assumption \ref{existence assumptions}.\\
The first three terms in the final line can be further bounded using Young's inequality point three of Remark \ref{remark choice of pdes} and the final term can also be bounded using points two and three of Remark \ref{remark choice of pdes} as well as H\"older's inequality.\\
To bound the integral involving $\Phi(\rho)$ in the second line we first note the polynomial growth condition $\Phi(\rho)\leq c(1+\rho^m)$. 
We then use Proposition \ref{proposition estimating powers of l1 norm of solution} to bound this term and the term involving the $L^1(U)$ norm of the solution,
\[\int_0^t\int_U\left(|\rho|+|\Phi(\rho)|\right)\leq cT+c\left(\frac{T}{\epsilon}+\epsilon T\|\Theta_\Phi(g)\|_{H^1(U)}+\epsilon\int_0^T\int_U\left|\nabla\Theta_\Phi(\rho)\right|^2\right).\]
Choosing $\epsilon$ so small that the final term can be absorbed into the left hand side of the estimate and taking the supremum over $t\in[0,T]$ we obtain the first estimate \eqref{first energy estimate of regularised solution}.\\
To prove the the second estimate \eqref{second energy estimate of regularised solution}, for the function $S_M:[0,\infty)\to[0,\infty)$ defined in Definition \ref{definition of functions g and h_M}, apply It\^o's formula to a regularised version of $\Psi_{S_M}:U\times[0,\infty)\to\mathbb{R}$ defined by
\[\Psi_{S_M}(x,0)=0, \hspace{10pt} \partial_\xi\Psi_{S_M}(x,\xi)=S'_M(\xi)-h_M(x), \]
where $h_M$ satisfies the PDE in Definition \ref{definition of functions g and h_M} and ensures that $\partial_\xi\Psi_{S_M}(x,\rho)$ vanishes along the boundary.
Grouping the $\nabla h_M$ and $\nabla\rho$ terms, we have
\begin{align}\label{ito formula for second energy estimate}
    &\partial_t\int_U\Psi_{S_M}(x,\rho(x,t))\,dx=\int_U \partial_\xi\Psi_{S_M}(x,\rho(x,t))\partial_t\rho\,dx+\frac{1}{2}\int_U \partial^2_\xi\Psi_{S_M}(x,\rho(x,t))\partial_t\langle\rho\rangle_t\,dx\nonumber\\
     &=-\int_U\int_0^t S''(\rho)\nabla\rho\cdot\left(\nabla\Phi(\rho)+\alpha\nabla\rho-\sigma(\rho)\,d\xi^F-\nu(\rho)-\sigma(\rho)\sigma'(\rho)F_2\right)+\frac{1}{2}\int_U\int_0^t S''(\rho)F_3\sigma^2(\rho)\nonumber\\
     &+\int_U\int_0^t \nabla h_M \cdot\left(\nabla\Phi(\rho)+\alpha\nabla\rho-\sigma(\rho)\,d\xi^F-\nu(\rho)+\frac{1}{2}F_1(\sigma'(\rho))^2\nabla\rho+\sigma(\rho)\sigma'(\rho)F_2\right).
\end{align}
We deal with the terms in the same way as the first estimate. 
The first two terms form part of the estimate so are moved to the left hand side, the noise terms vanish in expectation and the fourth term can be re-written as a boundary integral. 
For the fifth term in the right hand side of \eqref{ito formula for second energy estimate} we use the distributional equality
\[\mathbbm{1}_{M_1<\rho<M_2}\sigma(\rho)\sigma'(\rho)\nabla\rho=\frac{1}{2}\mathbbm{1}_{M_1<\rho<M_2}\nabla\sigma^2(\rho)
=\frac{1}{2}\nabla\left(\sigma^2((\rho\wedge M_2)\vee M_1)-\sigma^2(M_1)\right).
\]
Then using integration by parts gives
\begin{align*}
\frac{1}{2}\int_U\int_0^T S''(\rho)\sigma(\rho)\sigma'(\rho)\nabla\rho\cdot F_2&=-\frac{1}{4}\int_U\int_0^T\left(\sigma^2((\rho\wedge M_2)\vee M_1)-\sigma^2(M_1)\right)\nabla\cdot F_2\\
&+\frac{1}{4}\int_{\partial U}\int_0^T\left(\sigma^2((\Phi^{-1}(\bar{f})\wedge M_2)\vee M_1)-\sigma^2(M_1)\right) F_2\cdot\hat{\eta}.
\end{align*}
This implies by the boundedness of $\nabla\cdot F_2$, $F_3$ and $F_2\cdot\hat{\eta}$ that
\begin{align*}
    \frac{1}{2}\int_U\int_0^T S''(\rho)\left(\sigma(\rho)\sigma'(\rho)\nabla\rho\cdot F_2+ F_3\sigma^2(\rho)\right)&\leq c\int_U\int_0^T\mathbbm{1}_{\rho\geq M_1}\sigma^2(\rho\wedge M_2)\\
    &+c\int_{\partial U}\int_0^T\left(\sigma^2((\Phi^{-1}(\bar{f})\wedge M_2)\vee M_1)-\sigma^2(M_1)\right).
\end{align*}
Again we mention that the terms in \eqref{ito formula for second energy estimate} involving $\nabla h_M$ would vanish if our boundary condition was constant. 
Otherwise they are dealt with in a similar way to the first estimate, except that we use H\"older's inequality everywhere rather than Young's inequality.
In this way we keep the $M$ dependence in these terms through $h_M$.
Explicitly, we have for the terms that involve another gradient using points one and three of Remark \ref{remark choice of pdes},
   \begin{align*}
          \int_U\int_0^T &\nabla h_M \cdot\left(\nabla\Phi(\rho)+\alpha\nabla\rho+\frac{1}{2}F_1(\sigma'(\rho))^2\nabla\rho\right)=T\int_{\partial U}\frac{\partial h_M}{\partial \hat{\eta}}\left(\bar{f}+\alpha\Phi^{-1}(\bar{f})+\Psi_\sigma(\Phi^{-1}(\bar{f}))\right)\\
          &\leq cT\left\|S_M'(\Phi^{-1}(\bar{f}))\right\|_{H^1(\partial U)}\left(\left\|\bar{f}\right\|_{L^2(\partial U)}+\alpha\left\|\Phi^{-1}(\bar{f})\right\|_{L^2(\partial U)}+\left\|\Psi_\sigma(\Phi^{-1}(\bar{f}))\right\|_{L^2(\partial U)}\right).
   \end{align*}
For the remaining two terms, the bound in point eight of Assumption \ref{existence assumptions} alongside H\"older's inequality and point two and three of Remark \ref{remark choice of pdes} gives
\begin{align*}
     \int_U\int_0^T \nabla h_M \cdot&\left(-\nu(\rho)+\sigma(\rho)\sigma'(\rho)F_2\right)\leq c\|\nabla h_M\|_{L^2(U)}\int_U\int_0^T (1+\rho+\Phi(\rho))\\
     &\leq c\left\|S_M'(\Phi^{-1}(\bar{f}))\right\|_{H^1(\partial U)}\int_U\int_0^T (1+\rho+\Phi(\rho)).
\end{align*}
The final two terms can be dealt with using Proposition \ref{proposition estimating powers of l1 norm of solution}, just as in the first estimate. 
Here we \textcolor{red}{cannot} absorb the gradient term that appears, and it stays on the right hand side of the estimate, and it is bounded as a consequence of the first estimate.\\
To complete the estimate we move the first term on the left hand involving $\Psi_{S_M}$ in \eqref{ito formula for second energy estimate} to the right hand side.
To handle it, the definition of $\Psi_{S_M}$ implies that $\Psi_{S_M}(x,\rho)=S_M(\rho)-h(x)\rho$. 
Furthermore the product $h_M\rho_0$ and $\int_U S_M(\rho(x,T))$ are non-negative so can be removed from the estimate.
Using the bound $S_M(\rho_0(x))\leq (\rho_0(x)-M_1)_+$ and Holder's inequality to bound the boundary terms and integral of  $h\rho_T$, noting that the $L^2(U)$ norm of $\rho_T$ is bounded using the first estimate, completes the estimate.
\end{proof}

\begin{remark}[Distinguishing constant and non-constant boundary conditions]\label{remark distinguishing constant and non constant boundary conditions}
    The analysis in the energy estimates above (as well as all subsequent estimates) is far simpler in the case that the boundary condition is non-negative constant, say $\bar{f}=a \geq 0$. 
    In this case the PDEs $g$ and $h_M$ as defined in Definition \ref{definition of functions g and h_M} are solved by constant functions, for instance 
    \[g(x)=\Phi^{-1}(a),\hspace{20pt} x\in\bar{U}.\] 
    Hence the gradient $\nabla g$ as well as the normal derivative $\frac{\partial g}{\partial \hat{\eta}}$ are zero, and so terms involving either of these factors vanish immediately.\\
   \textcolor{red}{ For example, this means that the first energy estimate \eqref{first energy estimate of regularised solution} reads
   \begin{align*}\label{first energy estimate of regularised solution}
    \frac{1}{2}\sup_{t\in[0,T]}&\mathbb{E}\left(\int_U(\rho(x,t)-g(x))^2\right) + \mathbb{E}\left(\int_U\int_0^T|\nabla\Theta_{\Phi}(\rho)|^2\right) +\mathbb{E}\left(\alpha\int_U\int_0^T|\nabla \rho|^2\right) \nonumber\\
    &\leq c\int_0^T\int_U\left(1+\rho+\Phi(\rho)\right)+cT\int_{\partial U}\sigma^2(\Phi^{-1}(a))\\
    &\leq cT\left(1+\Theta_\Phi(\Phi^{-1}(a))+\sigma^2(\Phi^{-1}(a))\right),
\end{align*}
where the above constant is independent of the boundary data $a$.}
\end{remark}

Our goal is to use the above energy estimates to prove fractional Sobolev spatial regularity of the solution. 
We will need the below lemma that will allow us to prove regularity of (fractional) spatial derivatives of the solution. \textcolor{red}{The proof follows \textcolor{red}{\cite[Lem.~5.11]{fehrman2024well}} where we need to use the extension operator theorem \textcolor{red}{\cite[Ch.~5.4]{evans2022partial}} and the fact that $\Theta_\Phi\in H^1(U)$.}

\begin{lemma}\label{lemma on higher order spatial regularity of regularised solution}
    Let $\Phi$ satisfy Assumption \ref{existence assumptions}. Let $z\in H^1(U)$ be non-negative. If $\Phi$ satisfies \eqref{1 of 2 point 4 existence assumption} (allows us to handle $0<m<1$) then
    \[\|\nabla z\|_{L^1(U;\mathbb{R}^d)}\leq \|z\|^\theta_{L^1(U)}\|\nabla\Theta_\Phi\textcolor{red}{(z)}\|_{L^2(U;\mathbb{R}^d)}.\]
    If $\Phi$ satisfies \eqref{2 of 2 point 4 existence assumption} (allows us to handle $m\geq1$) then for every $\beta\in(0,1\wedge 2/q)$, for $c\in(0,\infty)$ depending on $\beta$,
    \[\|z\|_{W^{\beta,1}(U)}\leq c\left(\|z\|_{L^1(U)}+\|\nabla\Theta_\Phi\textcolor{red}{(z)}\|^{\frac{2}{q}}_{L^2(U;\mathbb{R}^d)}\right).\]
\end{lemma}

We have the following regularity estimate.

\begin{corollary}[Regularity of solution to regularised equation]
Let $\xi^F, \Phi,\sigma$ and $\nu$ satisfy Assumptions \ref{Assumption on noise Fi}, \ref{existence assumptions} and \ref{assumption about smoothing sigma in existence}, and fix the regularisation $\alpha\in(0,1)$ and terminal time $T\in[1,\infty)$.
    Let further $\rho_0\in L^2(\Omega;L^2(U))$ be non-negative and $\mathcal{F}_0$ measurable, $\Theta_{\nu,i}$ and $\Psi_\sigma$ be defined as in Assumption \ref{existence assumptions} and $g$ be the PDE defined in Definition \ref{definition of functions g and h_M} satisfying Assumption \ref{assumption on g and h_M}. 
    Let $\rho$ be a weak solution of \eqref{Regularised equation} in the sense of Definition \ref{definition of weak Solution of regularised equation}. 
    \begin{itemize}
        \item If $\Phi$ satisfies \eqref{1 of 2 point 4 existence assumption}, then for $c\in(0,\infty)$ independent of $\alpha$ and $T$, we have
        \begin{align*}
            \mathbb{E}&\left(\|\rho\|_{L^1([0,T];W^{1,1}(U))}\right)\\
            &\leq c\|\rho_0-g\|^2_{L^2(U)}+ cT\left(1+\|\Theta_\Phi(g)\|_{H^1(U)}+\sum_{i=1}^d\int_
    {\partial U}\Theta_{\nu,i}(\Phi^{-1}(\bar{f}))\cdot\hat{\eta}_i\right)\\
    &+cT\left(\|\sigma^2(\Phi^{-1}(\bar{f}))\|_{L^1(\partial U)}+\|\bar{f}\|^2_{L^2(\partial U)}+\|\Psi_\sigma(\Phi^{-1}(\bar{f}))\|^2_{L^2(\partial U)}+(1+\alpha)\|\Phi^{-1}(\bar{f})\|^2_{H^1(\partial U)}\right).
        \end{align*}
        \item If $\Phi$ satisfies \eqref{2 of 2 point 4 existence assumption}, then for all $\beta\in(0,2/q\wedge1)$ there is a constant $c\in(0,\infty)$ depending on $\beta$, but independent of $\alpha$ and  $T$, such that 
        \begin{align*}
            \mathbb{E}&\left(\|\rho\|_{L^1([0,T];W^{\beta,1}(U))}\right)\\
            &\leq c\|\rho_0-g\|^2_{L^2(U)}+ cT\left(1+\|\Theta_\Phi(g)\|_{H^1(U)}+\sum_{i=1}^d\int_
    {\partial U}\Theta_{\nu,i}(\Phi^{-1}(\bar{f}))\cdot\hat{\eta}_i\right)\\
    &+cT\left(\|\sigma^2(\Phi^{-1}(\bar{f}))\|_{L^1(\partial U)}+\|\bar{f}\|^2_{L^2(\partial U)}+\|\Psi_\sigma(\Phi^{-1}(\bar{f}))\|^2_{L^2(\partial U)}+(1+\alpha)\|\Phi^{-1}(\bar{f})\|^2_{H^1(\partial U)}\right).
        \end{align*}
    \end{itemize}
\end{corollary}

\begin{proof}
For both estimates we can use Proposition \ref{proposition estimating powers of l1 norm of solution} to bound the $L^1_tL^1_x$ norms with $\epsilon=1$.
 The first estimate then follows from the first item in Lemma \ref{lemma on higher order spatial regularity of regularised solution}, Young's inequality, the fact that $\theta\in[0,1/2]$ and Proposition \ref{energy estimate of regularised equation proposition}.\\
The second estimate follows from the second point in Lemma \ref{lemma on higher order spatial regularity of regularised solution}, the fact that $q\in[1,\infty)$ and Proposition \ref{energy estimate of regularised equation proposition}.
\end{proof}

The final result we want to show is a higher order space-time regularity of solutions cut away from their zero set.
The result will be useful when proving the existence of a weak solution of equation \eqref{Regularised equation} with smooth and bounded $\sigma$, and will subsequently motivate the introduction of a new metric on $L^1_xL^1_t$, see Definition \ref{definition of new metric on L1L1}.

\begin{definition}[Cutoff away from zero] \label{ definition of cutoff away from 0}
    For $\beta\in(0,1)$ let $\phi_{\beta}$ be the piecewise linear cutoff in Definition \ref{definition of convolution kernels and cutoff functions}. 
    Let $\Tilde{\phi}_\beta\in C^\infty([0,\infty))$ be a smooth approximation of $\phi_\beta$. 
    That is to say $0\leq\Tilde{\phi}_\beta\leq1$ is non-decreasing and satisfies $\Tilde{\phi}_\beta(\xi)=1$ for $\xi\geq\beta$, $\Tilde{\phi}_\beta(\xi)=0$ for $\xi\leq\beta/2$ with $|\Tilde{\phi}'_\beta(\xi)|\leq c/\beta$ for constant $c\in(0,\infty)$ independent of $\beta$.\\
    For each $\delta\in(0,1)$ define $\Phi_\beta\in C^{\infty}([0,\infty))$ by
    \[\Phi_\beta(\xi)=\Tilde{\phi}_\beta(\xi)\xi.\]
\end{definition}

\begin{proposition}\label{proposition higher order time regularity of solution cut away from zero set}
    Let $\xi^F, \Phi,\sigma$ and $\nu$ satisfy Assumptions \ref{Assumption on noise Fi}, \ref{existence assumptions} and \ref{assumption about smoothing sigma in existence}, and let $\alpha\in(0,1)$.
    Let further $\rho_0\in L^2(\Omega;L^2(U))$ be non-negative and $\mathcal{F}_0$ measurable, $\Theta_{\nu,i}$ and $\Psi_\sigma$ be defined as in Assumption \ref{existence assumptions} and $g$ be the PDE defined in Definition \ref{definition of functions g and h_M} satisfying Assumption \ref{assumption on g and h_M}.
        Let $\rho$ be a weak solution of the regularised equation \eqref{Regularised equation} in the sense of Definition \ref{definition of weak Solution of regularised equation}.\\
For every $\delta\in(0,1/2)$ and $s>\frac{d}{2}+1$ there is a constant $c\in(0,\infty)$ depending on $\beta,\delta$ and $s$ but independent of $\alpha$ and $T$ such that
\begin{align*}
    \mathbb{E}&\left(\|\Phi_\beta(\rho)\|_{W^{\delta,1}([0,T];H^{-s}(U))}\right)\leq cT \|\rho_0-g\|^2_{L^2(U)}\\
    &+ cT^2\left(1+\|\Theta_\Phi(g)\|_{H^1(U)}+\sum_{i=1}^d\int_
    {\partial U}\Theta_{\nu,i}(\Phi^{-1}(\bar{f}))\cdot\hat{\eta}_i+ \|\sigma^2(\Phi^{-1}(\bar{f}))\|_{L^1(\partial U)}\right)\nonumber\\
    &+cT^2\left(\|\bar{f}\|^2_{L^2(\partial U)}+\|\Psi_\sigma(\Phi^{-1}(\bar{f}))\|^2_{L^2(\partial U)}+(1+\alpha)\|\Phi^{-1}(\bar{f})\|^2_{H^1(\partial U)}\right).
\end{align*}
\end{proposition}
The proof is analogous to \textcolor{red}{\cite[Ppn.~5.14]{fehrman2024well}}, we just give the main idea below.

\begin{proof}
    Similarly to the derivation of the kinetic equation, it follows by It\^o's formula and then bringing the $\Phi_\beta'(\rho)$ inside the derivative and using the product rule, that
    \begin{align*}
        d\Phi_\beta(\rho)&=\Phi'_\beta(\rho)d\rho+\frac{1}{2}\Phi''_\beta(\rho)d\langle\rho\rangle\\        &=\nabla\cdot\left(\Phi'_\beta(\rho)\nabla\Phi(\rho)+\alpha\Phi'_\beta(\rho)\nabla\rho-\Phi'_\beta(\rho)\sigma(\rho)d\xi^F-\Phi'_\beta(\rho)\nu(\rho)\right)\\
        &+\nabla\cdot\left(\frac{1}{2}F_1\Phi'_\beta(\rho)(\sigma'(\rho))^2\nabla\rho+\frac{1}{2}\Phi'_\beta(\rho)\sigma(\rho)\sigma'(\rho)F_2\right)\\
        &-\Phi''_\beta(\rho)\nabla\rho\cdot\nabla\Phi(\rho)-\alpha\Phi''_\beta(\rho)|\nabla\rho|^2+ \Phi''_\beta(\rho)\sigma(\rho)\nabla\rho\cdot d\xi^F+ \Phi''_\beta(\rho)\nabla\rho\cdot\nu(\rho)\\
        &+\frac{1}{2}\Phi_\beta''(\rho)\left(\sigma(\rho)\sigma'(\rho) F_2\cdot\nabla\rho+F_3\sigma^2(\rho)\right)
    \end{align*}
Doing so will allow us to compute the the fractional Sobolev norm more easily.
Integrating in time, and writing some derivatives in terms of the function $\Theta_\Phi$, we have for fixed $x\in U$ the decomposition $\Phi_\beta(\rho(x,t))=\Phi_\beta(\rho_0(x))+I_t^{f.v.}+I_t^{mart}$ with
\[I_t^{mart}:=-\int_0^t\nabla\cdot\left(\Phi_\beta'(\rho)\sigma(\rho)d\xi^F\right)+\int_0^t\Phi''_\beta(\rho)\sigma(\rho)(\Phi'(\rho))^{-1/2}\nabla\Theta_\Phi(\rho)\cdot d\xi^F,\]
and 
\begin{align*}
    I_t^{f.v.}&:=\int_0^t\nabla\cdot\left(\Phi'_\beta(\rho)(\Phi'(\rho))^{1/2}\nabla\Theta_\Phi(\rho)\right)-\int_0^t\Phi''_\beta(\rho)|\nabla\Theta_\Phi(\rho)|^2+\alpha\int_0^t\nabla\cdot(\Phi'_\beta(\rho)\nabla\rho)\\
    &-\alpha\int_0^t\Phi''_\beta(\rho)|\nabla\rho|^2+\frac{1}{2}\int_0^t\nabla\cdot\left(F_1\Phi'_\beta(\rho)\frac{(\sigma'(\rho))^2}{(\Phi'(\rho))^{1/2}}\nabla\Theta_\Phi(\rho)\right)+\frac{1}{2}\int_0^t\nabla\cdot\left(\Phi'_\beta(\rho)\sigma(\rho)\sigma'(\rho) F_2\right)\\
    &+\frac{1}{2}\int_0^t\Phi_\beta''(\rho)\frac{\sigma(\rho)\sigma'(\rho)}{\Phi'(\rho)^{1/2}} F_2\cdot\nabla\Theta_\Phi(\rho)+\frac{1}{2}\int_0^t\Phi''_\beta(\rho)F_3\sigma^2(\rho)-\int_0^t\nabla\cdot(\Phi'_\beta(\rho)\nu(\rho))\\
    &-\int_0^t\nabla\cdot(\Phi'_\beta(\rho)\nu(\rho))+\int_0^t\Phi''_\beta(\rho)\nabla\rho\cdot\nu.
\end{align*}
We begin by showing the martingale term is in $W^{\delta,2}([0,T];H^{-s}(U))$, that is
\[\|I^{mart}_t\|^2_{W^{\delta,2}([0,T];H^{-s}(U))}:=\int_0^T\|I^{mart}_t\|^2_{H^{-s}(U)}\,dt + \int_0^T\int_0^T\frac{\|I^{mart}_t-I^{mart}_s\|^2_{H^{-s}(U)}}{|t-s|^{1+2\delta}}\,ds\,dt<\infty.\]
Since $s>\frac{d}{2}+1$ the Sobolev embedding theorem tells us that for test functions $\phi$ that we use in the definition of negative fractional Sobolev norm, there is a constant $c\in(0,\infty)$ such that 
\begin{equation}\label{l infinity estimate of test functions}
    \|\phi\|_{L^\infty(U)}+\|\nabla \phi\|_{L^\infty(U;\mathbb{R}^d)}\leq c\|\phi\|_{H^s(U)}
\end{equation}
Using this and an argument similar to \textcolor{red}{\cite[Lem.~2.1]{flandoli1995martingale}} alongside the Burkholder-Davis-Gundy inequality, the fact that $\Phi_\beta''$ is supported on $[\beta/2,\beta]$ as well as the bounds in Assumption \ref{existence assumptions}, that the second term in the definition of the norm satisfies
\begin{align*}
        &\mathbb{E}\left(\int_0^T\int_0^T\frac{\|I^{mart}_t-I^{mart}_s\|^2_{H^{-s}(U)}}{|t-s|^{1+2\delta}}\,ds\,dt\right)\\
        &\leq c\mathbb{E}\left(\int_0^T\|\Theta_{\Phi}(\rho)\mathbbm{1}_{\beta/2\leq\rho\leq \beta}\|^2_{L^2(U)}+\|\sigma(\rho)\|^2_{L^2(U)}\,ds\right)\\
        &\leq c\mathbb{E}\left(|U|T+ \int_0^T\int_U\left(|\rho_t|+|\nabla\Theta_{\Phi}|^2+|\Phi(\rho)|\right)\right).
\end{align*}
Analogously for the first term in the norm we have the same bound
\[\int_0^T\|I^{mart}_t\|^2_{H^{-s}(U)}\,dt\leq c\mathbb{E}\left(|U|T+ \int_0^T\int_U\left(|\rho_t|+|\nabla\Theta_{\Phi}|^2+|\Phi(\rho)|\right)\right). \]
Putting these together, it follows from Proposition \ref{proposition estimating powers of l1 norm of solution} that
\begin{align*}
    \|I^{mart}_t\|^2_{W^{\delta,2}([0,T];H^{-s}(U))}&\leq c\mathbb{E}\left(T+ T\|\Theta_\Phi(g)\|_{H^1(U)}+\int_0^T\int_U|\nabla\Theta_{\Phi}|^2\right).
\end{align*}
We next show that the finite variation term is  in $W^{1,1}([0,T];H^{-1}(U))$, that is
\[\left\|I^{f.v.}\right\|_{W^{1,1}([0,T];H^{-s}(U))}:=\int_0^T\left\|I^{f.v.}\right\|_{H^{-s}(U)}+\int_0^T\left\|\frac{d}{dt}I^{f.v.}\right\|_{H^{-s}(U)}<\infty.\]
It follows from \eqref{l infinity estimate of test functions}, the fact that $\Phi_\beta'$ and $\Phi_\beta''$ are supported on $[\beta/2,\infty)$ and $[\beta/2,\beta]$ respectively and Young's inequality, that we can bound the first term by
\begin{align*}
    \int_0^T\left\|I^{f.v.}\right\|_{H^{-s}(U)}&\leq c\int_0^T\int_U\int_0^t\left( \Phi'(\rho)\mathbbm{1}_{\rho>\beta/2}+|\nabla\Theta_{\Phi}(\rho)|^2\mathbbm{1}_{\rho>\beta/2}+\alpha|\nabla\rho|\mathbbm{1}_{\rho>\beta/2}\right.\\
    &+\frac{\sigma'(\rho)^4}{\Phi'(\rho)}\mathbbm{1}_{\rho>\beta/2}+|\sigma(\rho)\sigma'(\rho)|\mathbbm{1}_{\rho>\beta/2}+\frac{(\sigma(\rho)\sigma'(\rho))^2}{\Phi'(\rho)}\mathbbm{1}_{\rho>\beta/2}+|\nu(\rho)|\mathbbm{1}_{\rho>\beta/2}\\
    &\left.+(1-\alpha)|\nabla\rho|^2\mathbbm{1}_{\beta/2<\rho<\beta}+\sigma^2\mathbbm{1}_{\beta/2<\rho<\beta}+|\nu|^2\mathbbm{1}_{\beta/2<\rho<\beta}\right)\,ds\,dx\,dt.
\end{align*}
Using the fundamental theorem of calculus, the second finite variation term consists of the same terms but without the inner $ds$ integral.
Using that the inner integral is increasing in $t$, the fact that $\alpha\in(0,1)$ and so $\alpha x\leq 1+\alpha x^2$, the fact that the three terms in the final line are bounded over the indicator set, points three six and eight of Assumption \ref{existence assumptions} to bound the various coefficients and using Proposition \ref{proposition estimating powers of l1 norm of solution} analogously to the martingale term, we have
\begin{align*}
   \mathbb{E}\left\|I^{f.v.}\right\|_{W^{1,1}([0,T];H^{-s}(U))} &\leq c(1+T)\mathbb{E}\left(T+ T\|\Theta_\Phi(g)\|_{H^1(U)}+\int_0^T\int_U\left(|\nabla\Theta_{\Phi}|^2+\alpha|\nabla\rho|^2\right)\right).
\end{align*}
Using the trivial fact that there is a constant $c$ such that $(1+T)<cT$, the estimate then follows by the first energy estimate in Proposition \ref{energy estimate of regularised equation proposition} alongside the continuous embeddings $W^{\beta,2}, W^{1,1}\hookrightarrow W^{\beta,1}$ for every $\beta\in(0,1/2)$.
\end{proof}

\subsection{Entropy estimate}\label{section4.2 entropy estimate}

In this section we prove an entropy estimate for weak solutions of the regularised Dean--Kawasaki equation, following \textcolor{red}{\cite[Ppn.~5.18]{fehrman2024well}}. 
The estimates will not be used in the remainder of the work, but they are provided because of their connection with the study of large deviation principles, see for instance the introduction of \cite{fehrman2023non}.
We will need the below assumptions.
\begin{assumption}[Entropy assumptions]\label{Assumptions for entropy space}
    Let $\Phi,\sigma\in C([0,\infty))$ and $\nu\in C([0,\infty);\mathbb{R}^d)$ satisfy the following assumptions from \textcolor{red}{\cite[Asm.~5.17]{fehrman2024well}}.
    \begin{enumerate}    
        
        \item There exists a constant $c\in(0,\infty)$ such that $|\sigma(\xi)|\leq c\Phi^{1/2}(\xi)$ for every $\xi\in[0,\infty)$.
        \item There exists a constant $c\in(0,\infty)$ such that for every $\xi\in[0,\infty)$
        \begin{equation}\label{assumptions for kinetic measure: bounding nu and phi prime}
            \Phi'(\xi)\leq c(1+\xi+\Phi(\xi)).
        \end{equation}
        \item We have $\nabla\cdot F_2=0$.
       
        \item We have that $\log(\Phi)$ is locally integrable on $[0,\infty)$. 
        
        \vspace{10pt}
        Furthermore, we have the two additional new assumptions, analogous to the new assumptions in Assumption \ref{existence assumptions}.
        \item
        Either $F_2=0$ or the unique function $\Theta_{\Phi,\sigma}$ defined by 
        \[\Theta_{\Phi,\sigma}(1)=0, \hspace{100pt} \Theta'_{\Phi,\sigma}(\xi)=\frac{\Phi'(\xi)\sigma'(\xi)\sigma(\xi)}{\Phi(\xi)}\] 
        satisfies $\Theta_{\Phi,\sigma}(\Phi^{-1}(\bar{f}))\in L^1(\partial U)$.
        \item For $i=1,\hdots,d$, the unique functions $\Theta_{\Phi,\nu,i}$ defined by
        \[\Theta_{\Phi,\nu,i}(0)=0, \hspace{100pt} \Theta'_{\Phi,\nu,i}(\xi)=\frac{\Phi'(\xi)\nu_i(\xi)}{\Phi(\xi)}
        \]
        satisfy $\Theta_{\Phi,\nu,i}(\Phi^{-1}(\bar{f}))\cdot \eta_i \in L^1(\partial U)$. 
    \end{enumerate}
\end{assumption}

We give remarks about the new assumptions.

\begin{remark}\label{remark on new assumptions}
\begin{itemize}\label{remark on additional enrtopy assumptions}
    \item  In the model case the function in point five is given by $\Theta_{\Phi,\sigma}(\xi)=m\log(\xi)$. In the case of zero boundary conditions we therefore require $F_2=0$. 
    \item  We can without loss of generality consider $\nu(0)=0$, and since $\nu$ is Lipschitz in  the model case we have that the function in point six is $\Theta_{\Phi,\nu,i}(\xi)=c\xi$ for some constant $c\in(0,\infty)$ and the assumption is satisfied.
\end{itemize}
\end{remark}

We now introduce the family of PDEs $\{v_\delta\}$.

\begin{definition}\label{definition of v delta}
     For every $\delta\in(0,1)$ let $v_\delta$ satisfy the PDE
\begin{equation}\label{equation for v delta}
    \begin{cases}
       -\Delta v_\delta=0,&\text{in}\hspace{5pt} U,\\
       v_\delta=\log(\bar{f}+\delta),&\text{on}\hspace{5pt} \partial U,\\
    \end{cases}
\end{equation}
where $\bar{f}$ is the boundary condition in \eqref{generalised Dean--Kawasaki Equation Stratonovich}. 
\end{definition}

We will need to bound normal derivatives of $v_\delta$ in the $\delta\to0$ limit, motivating the below assumption.

\begin{assumption}[Assumption on $v_\delta$]\label{assumption on the boundary data}
     Either the boundary data $\bar{f}$ is constant or the solution of the PDE $v_0$ defined by
    \begin{equation*}
    \begin{cases}
       -\Delta v_0=0,&\text{in}\hspace{5pt} U,\\
       v_0=\log(\bar{f}),&\text{on}\hspace{5pt} \partial U,
    \end{cases}
\end{equation*}
satisfies $\log(\bar{f})\in H^1(\partial U)$, where the norm is defined in Theorem \ref{thm: definiton of H^1 boundary norm}.
\end{assumption}

We now present the entropy estimate, the proof follows \textcolor{red}{\cite[Prop.~5.18]{fehrman2024well}}. 
We once again repeat the comment in Remark \ref{remark distinguishing constant and non constant boundary conditions} and emphasise that the below analysis is simplified if the boundary condition is constant.

\begin{proposition}[Entropy estimate]\label{entropy estimate}
    Let $\xi^F,\Phi,\sigma$ and $\nu$ satisfy Assumptions \ref{Assumption on noise Fi}, \ref{existence assumptions}, \ref{assumption about smoothing sigma in existence} and \ref{Assumptions for entropy space}, and suppose Assumption \ref{assumption on the boundary data} concerning $v_0$ is satisfied.
    Let the functions $\Theta_{\Phi,\sigma},\{\Theta_{\Phi,\nu,i}\}_{i=1}^d,\Psi_\sigma$ be defined as in Assumption \ref{Assumptions for entropy space}.
    Let $\alpha\in(0,1)$, $T\in[1,\infty)$ and suppose the weak solution of the regularised equation \eqref{Regularised equation} has $\mathcal{F}_0$-measurable initial condition satisfying $\rho_0\in L^1(\Omega;Ent(U))$.
    For the function $\Psi_{\Phi,0}:U\times[0,\infty)\to\mathbb{R}$ defined by
$\Psi_{\Phi,0}(x,0)=0$ and $\partial_\xi\Psi_{\Phi,0}(x,\xi)=\log(\Phi(\xi))-v_0(x)$,  we have, for a constant $c\in(0,\infty)$ independent of $\alpha$ and $T$, the bound
 \begin{align*}
        \mathbb{E}\int_U\Psi_{\Phi,0}&(x,\rho(x,t))\,dx+ 4\mathbb{E}\int_0^T\int_U|\nabla\Phi^{1/2}(\rho)|^2+\alpha \mathbb{E}\int_0^T\int_U\frac{\Phi'(\rho)}{\Phi(\rho)}|\nabla\rho|^2\\
        &\leq \mathbb{E}\int_U\Psi_{\Phi,0}(x,\rho_0(x))\,dx+cT + cT\|\Theta_\Phi(g)\|_{H^1(U)}+c\mathbb{E}\int_0^T\int_U\left|\nabla\Theta_\Phi(\rho)\right|^2\nonumber\\
     &+cT\int_{\partial U} \Theta_{\Phi,\sigma}(\Phi^{-1}(\bar{f}))+T\sum_{i=1}^d \int_{\partial U} \Theta_{\Phi,\nu,i}(\Phi^{-1}(\bar{f}))\cdot\hat{\eta}_i+T\|\bar{f}\|^2_{L^2(\partial U)}\nonumber\\
      & 
      +\alpha T\|\Phi^{-1}(\bar{f})\|_{L^2(\partial U)}^2
      +T\|\Psi_\sigma(\Phi^{-1}(\bar{f}))\|^2_{L^2(\partial U)}+c(1+\alpha)T\|\log(\Phi^{-1}(\bar{f}))\|^2_{H^1(\partial U)}.
\end{align*}
\end{proposition}

\begin{proof}
To obtain the bound, apply It\^o's formula to the regularised function $\Psi_{\Phi,\delta}:U\times[0,\infty)\to\mathbb{R}$ defined by
$\Psi_{\Phi,\delta}(x,0)=0$ and $\partial_\xi\Psi_{\Phi,\delta}(x,\xi)=\log(\Phi(\xi)+\delta)-v_\delta(x)$, 
where $v_\delta$ satisfies the PDE \eqref{equation for v delta} in Definition \ref{definition of v delta} and once again ensures that $\partial_\xi\Psi_{\Phi,\delta}(x,\rho)$ vanishes along the boundary.
We get after integrating the first order term by parts, that
\begin{align}\label{equation applying ito in kinetic measure proposition}
     &   \left.\int_U\Psi_{\Phi,\delta}(x,\rho(x,t))\,dx\right|_{s=0}^T\nonumber\\
     &=\int_0^T\int_U-\frac{|\nabla\Phi(\rho)|^2}{\Phi(\rho)+\delta}-\frac{\alpha\Phi'(\rho)|\nabla\rho|^2}{\Phi(\rho)+\delta}+\frac{\sigma(\rho) \Phi'(\rho)\nabla\rho \cdot d{\xi}^F}{\Phi(\rho)+\delta}+\frac{\Phi'(\rho)\nabla\rho\cdot \nu(\rho)}{\Phi(\rho)+\delta}\nonumber\\
     &\hspace{170pt} + \frac{\Phi'(\rho)\sigma'(\rho)\sigma(\rho)\nabla\rho\cdot F_2}{2(\Phi(\rho)+\delta)}
      +\frac{F_3\Phi'(\rho)\sigma^2(\rho)}{2(\Phi(\rho)+\delta)}\nonumber\\
      &+\int_0^T\int_U\nabla v_\delta\cdot\left(\nabla\Phi(\rho)+\alpha\nabla\rho-\sigma(\rho)d\xi^F -\nu(\rho)+\frac{1}{2}F_1[\sigma'(\rho)]^2\nabla\rho + \frac{1}{2}\sigma'(\rho)\sigma(\rho)F_2\right).
\end{align}
The terms are handled in an analogous way to energy estimates already seen thus far in Propositions \ref{energy estimate of regularised equation proposition} and \ref{proposition about kinetic measure at 0}, so we are brief.
We move the first two terms to the left and side of the estimate, noting that the distributional inequality allows us to rewrite the first term as
$
    \nabla\Phi^{1/2}(\rho)=\frac{\Phi'(\rho)}{2\Phi^{1/2}(\rho)}\nabla\rho
$
\[\int_U\frac{|\nabla\Phi(\rho)|^2}{\Phi(\rho)+\delta}=\int_U\frac{4\Phi(\rho)}{\Phi(\rho)+\delta}|\nabla\Phi^{1/2}(\rho)|^2.\]
After taking expectation the third term is killed as well as the noise term in the second line.
The fourth and fifth terms can be re-written as boundary integrals. 
For the fourth, we use the functions $\Theta_{\Phi,\nu,\delta,i}$ for $i=1,\hdots,d$ defined by $\Theta_{\Phi,\nu,\delta,i}(0)=0, \Theta'_{\Phi,\nu,\delta,i}(\xi)=\frac{\Phi'(\xi)\nu_i(\xi)}{\Phi(\xi)+\delta}$,
and for the fifth we define the unique function $\Theta_{\Phi,\sigma,\delta}$ satisfying $\Theta_{\Phi,\sigma,\delta}(0)=0,\Theta'_{\Phi,\sigma,\delta}(\xi)=\frac{\Phi'(\xi)\sigma'(\xi)\sigma(\xi)}{\Phi(\xi)+\delta}$,
and note that either $F_2=0$, or use integration by parts alongside the assumptions $\nabla\cdot F_2=0$ and the fact that $F_2 \cdot \hat{\eta}$ is bounded.
For the final term in the first line of \eqref{equation applying ito in kinetic measure proposition}, by the assumption $\sigma\leq c \Phi^{1/2}$, the fact that $\frac{x}{x+\delta}<1$ for every $\delta>0$ and the assumption $\Phi'(\xi)\leq c(1+\xi+\Phi(\xi))$, we obtain
\[\frac{1}{2}\int_0^T\int_U  \frac{F_3\Phi'(\rho)\sigma^2(\rho)}{\Phi(\rho)+\delta}\leq c \int_0^T\int_U  \Phi'(\rho)\leq c\left(|U|T + \int_0^T\int_U\left(\rho+ \Phi(\rho)\right)\right). \]
The terms involving $\nabla v_\delta$ in \eqref{equation applying ito in kinetic measure proposition} would all vanish if the boundary condition was constant.
Otherwise they are handled in the way described in points one and two of Remark \ref{remark choice of pdes}.
The first, second and fifth terms in \eqref{equation applying ito in kinetic measure proposition} with other gradient terms are handled using integrate by parts and turn into boundary terms.
As for the terms without derivatives, using the $L^2(U)$ integrability of $\nabla v_\delta$ implied by point two of Remark \ref{remark choice of pdes}, Young's inequality, the boundedness of $F_2$, the equation \eqref{equation for v delta} satisfied by $v_\delta$ and Assumption \ref{existence assumptions}, we have
\begin{align*}
 \int_0^T\int_U\nabla v_\delta\cdot(-\nu(\rho) &+ \frac{1}{2}\sigma'(\rho)\sigma(\rho)F_2) 
\leq cT\|\log(\bar{f}+\delta)\|^2_{H^1(\partial U)}+c\int_0^T\int_U\left(1+\rho+\Phi(\rho)\right)\,dx\,dt.
\end{align*}
Putting everything together we get
\begin{align}\label{long espression in proof of kinetic measure at 0}
        \mathbb{E}\int_U&\Psi_{\Phi,\delta}(x,\rho(x,t))\,dx+ 4\mathbb{E}\int_0^T\int_U\frac{\Phi(\rho)}{\Phi(\rho)+\delta}|\nabla\Phi^{1/2}(\rho)|^2+\alpha \mathbb{E}\int_0^T\int_U\frac{\Phi'(\rho)}{\Phi(\rho)+\delta}|\nabla\rho|^2\nonumber\\
     &\leq \mathbb{E}\int_U\Psi_{\Phi,\delta}(x,\rho_0(x))\,dx+c\left(|U|T +\mathbb{E}\int_0^T\int_U\left(|\rho|+|\Phi(\rho)|\right)\right)\nonumber\\
     & +cT\int_{\partial U} \Theta_{\Phi,\sigma,\delta}(\Phi^{-1}(\bar{f})) +T\sum_{i=1}^d\int_{\partial U} \Theta_{\Phi,\nu,\delta,i}(\Phi^{-1}(\bar{f}))\cdot\hat{\eta}_i\nonumber\\
      & +T\left(\int_{\partial U}\bar{f} \frac{\partial v_\delta}{\partial \hat{\eta}}
      +\alpha \int_{\partial U}\Phi^{-1}(\bar{f}) \frac{\partial v_\delta}{\partial \hat{\eta}}
      +\int_{\partial U}\Psi_\sigma(\Phi^{-1}(\bar{f}))\frac{ \partial v_\delta}{\partial \hat{\eta}}+c\|\log(\bar{f}+\delta)\|^2_{H^1(\partial U)}\right).
\end{align}
Once again we used the fact that the boundary terms in the final two lines are deterministic and do not depend on time.
To obtain the desired estimate, we wish to take the $\delta\to 0$ limit in equation \eqref{long espression in proof of kinetic measure at 0}, and therefore we need a handle over the boundary terms which all depend on $\delta$.
Alongside Young's inequality, the new assumptions in Assumption \ref{Assumptions for entropy space} and Assumption \ref{existence assumptions} precisely allow us to do this.
Finally, to bound the integral of $\rho$ and $\Phi(\rho)$ in the second line, we use Proposition \ref{proposition estimating powers of l1 norm of solution} in the same manner as the energy estimates, noting that we \textcolor{red}{cannot} absorb the term involving $\left|\nabla\Theta_\Phi(\rho)\right|^2$ into the left hand side, but the term is bounded by the first energy estimate \eqref{first energy estimate of regularised solution}. 
The estimate is proven.
\end{proof}

\begin{remark}[Comparing entropy estimate with proof of kinetic measure at zero]\label{Comparing entropy estimate with proof of kinetic measure at zero}
    Note that the entropy estimate could be used to prove a statement about the kinetic measure at zero if we include the assumption that $\frac{\Phi(\xi)}{\Phi'(\xi)}\leq c\xi$.
    We have 
    \begin{align*}
2\beta^{-1}\mathbb{E}\left(q(U\times[\beta/2,\beta]\times[0,T])\right) \leq \liminf_{\alpha\to0}\mathbb{E}\left(\int_0^T\int_U\int_\mathbb{R}\frac{1}{\xi}\mathbbm{1}_{\beta/2\leq\xi\leq\beta}dq^\alpha\right).
\end{align*}
Subsequently the assumption gives
\begin{align*}
    \frac{1}{\xi}dq^a&=\frac{1}{\xi}\delta_0(\xi-\rho^\alpha)\left(\frac{4\Phi(\rho^\alpha)}{\Phi'(\rho^\alpha)}|\nabla\Phi^{1/2}(\rho^\alpha)|^2+\alpha|\nabla\rho^\alpha|^2\right)\nonumber\\
    &\leq c\delta_0(\xi-\rho^\alpha)\left(4|\nabla\Phi^{1/2}(\rho^\alpha)|^2+\alpha\frac{\Phi'(\rho)}{\Phi(\rho)}|\nabla\rho^\alpha|^2\right).
\end{align*}
One sees that this is the precise quantity which we showed was bounded in the entropy estimate above.
Consequently by dominated convergence theorem, with the indicator present, the kinetic measures go to zero.\\
The reason we do not use this estimate is due to the first term on the right hand side of the estimate, which requires $\rho_0\in L^1(\Omega;Ent(U))$. 
For the definition of stochastic kinetic solution, Definition \ref{definition of stochastic kinetic solution of generalised Dean--Kawasaki equation}, we only have $\rho_0\in L^1(\Omega;L^1(U))$. 
We circumvent this in the sequel by choosing a test function that cuts off the logarithm at $1$.
\end{remark}

\subsection{Decay of kinetic measure at zero}\label{section 4.3 decay of kinetic measure at zero}
In this section we will prove the decay of the kinetic measure at zero required in the uniqueness proof,
    \[\liminf_{\beta\to0}\left(\beta^{-1}q(U\times[\beta/2,\beta]\times[0,T])\right)=0.\]
First of all we begin with a remark that illustrates why we \textcolor{red}{are unable to} bound the decay of the kinetic measure using the kinetic equation as in \textcolor{red}{\cite[Prop.~4.6]{fehrman2024well}}.

\begin{remark}\label{remark outlining why we can't follow the proof of FG on bounding kinetic equation at zero}
Adapting the proof of \textcolor{red}{\cite[Prop.~4.6]{fehrman2024well}}, test the kinetic equation \eqref{kinetic equation} against smooth approximations of the product $\zeta_M\phi_\beta\iota_\gamma$ and subsequently take the limit in the approximations. 
We end up with an additional term when the spatial gradient hits the cutoff $\iota_\gamma$.
    Taking expectation to kill the noise term and taking the limit as $\gamma\to0$ (which we need to take before $M,\beta$ limits due to Remark \ref{remark highlighting singularity in ito correction}) one needs to consider the term
\begin{align*}
    &-\lim_{\gamma\to 0}\mathbb{E}\left(\int_0^t \int_{U}\zeta_M(\rho)\phi_\beta(\rho)\left(\Phi'(\rho)\nabla\rho+\frac{1}{2}F_1[\sigma'(\rho)]^2\nabla\rho +\frac{1}{2} \sigma'(\rho)\sigma(\rho)F_2\right)\cdot \nabla\iota_\gamma(x)\,dx\,ds\right)\nonumber.
\end{align*}
We cannot make sense of this limit because we do not have sufficient regularity of the first two terms. 
For example to take the limit in  the first term we would need to use the trace theorem and therefore require $\nabla\Phi(\rho)\in H^1_{loc}(U)$. 
However, we only have $\Phi(\rho)\in H^1_{loc}(U)$, i.e. $\nabla\Phi(\rho)\in L^2_{loc}(U)$.\\
Note that similar terms arose in the uniqueness proof, for instance recall equation \eqref{new terms in uniqueness}. 
Crucially, there the terms appear as a difference of two solutions and can only be handled because the first two terms above have a sign for every fixed $\gamma>0$.
\end{remark}

We consider the following PDE, where the boundary condition is the logarithm cutoff at $1$.

\begin{definition}[The PDE $v$]\label{definition of v}
    Let $\bar{f}$ be the boundary condition of the regularised equation \eqref{Regularised equation}.
    Define the function $S:[0,\infty)\to[0,\infty)$ by $S(0)=0$ and $S''(\xi)=\frac{1}{\xi}\mathbbm{1}_{0\leq\xi\leq1}$.
Define the harmonic PDE $v:U\to\mathbb{R}$ by 
\begin{equation*}
    \begin{cases}
            -\Delta v = 0 & \text{on} \hspace{5pt}U,\\
    v=S'(\Phi^{-1}(\bar{f}))& \text{on} \hspace{5pt}\partial U.
    \end{cases}
\end{equation*}
\end{definition}

By integrating we have that 
$S'(\xi)=\log(\xi\wedge1)$ and $S(\xi)=(\xi\wedge1)\log(\xi\wedge1)-(\xi\wedge1)$.

\begin{assumption}[Assumptions for kinetic measure]\label{assumption on kinetic measure}
\begin{enumerate}
\item  Either the boundary data $\bar{f}$ is constant or the solution of PDE $v$ satisfies  $S'(\Phi^{-1}(\bar{f}))\in H^1(\partial U)$, where the norm is defined in Theorem \ref{thm: definiton of H^1 boundary norm}.
\item We have that $v\in L^2(U)$.
        \item Either $F_2=0$, or the unique function $\Theta_\sigma$ defined by $\Theta_\sigma(1)=0$, $\Theta_\sigma'(\xi)=\frac{\sigma(\xi)\sigma'(\xi)}{\xi}$ satisfies
        $\Theta_\sigma(\Phi^{-1}(\bar{f})\wedge 1)\in L^1(\partial U)$.
        \item Either $\bar{f}$ is constant, or for the unique function $\Psi_\sigma$ defined by $\Psi_{\sigma}(0)=0$, $\Psi_\sigma'(\xi)=F_1[\sigma'(\xi)]^2$, we have $\Psi_\sigma(\Phi^{-1}(\bar{f}))\in L^2(\partial U)$.
\end{enumerate}
\end{assumption}

For point three, note that the $L^1(\partial U)$ integrability of $\Theta_\sigma$ is guaranteed if, for example, $\sigma(\xi)\sigma'(\xi)\leq c\xi^\alpha$, where $\alpha>0$ and $c\in(0,\infty)$. 

\begin{proposition}\label{proposition about kinetic measure at 0}
     Let $\xi^F, \Phi,\sigma$ and $\nu$ satisfy Assumption \ref{Assumption on noise Fi}, \ref{assumptions on coefficients for uniqueness} and Assumption \ref{assumption on kinetic measure}.
    Let further $\rho_0\in L^1(\Omega;L^1(U))$ be non-negative and $\mathcal{F}_0$ measurable and $v$ be defined as in Definition \ref{definition of v}. 
    If $\rho$ is a stochastic kinetic solution of \eqref{generalised Dean--Kawasaki equation Ito} in the sense of Definition \ref{definition of stochastic kinetic solution of generalised Dean--Kawasaki equation}, then it follows almost surely that
    \[\liminf_{\beta\to0}\left(\beta^{-1}q(U\times[\beta/2,\beta]\times[0,T])\right)=0.\]
\end{proposition}

\begin{proof}
Let us begin by noting that, whilst we do not know the precise form of the limiting measure $q$ due to the presence of a parabolic defect measure in the limit, see \cite{chen2003well}, for the regularised equation we have a precise equation for the kinetic measures, see proof of \textcolor{red}{\cite[Prop.~5.12]{fehrman2024well}} given by
\begin{align*}
   dq^\alpha=\delta_0(\xi-\rho^\alpha)\left(\Phi'(\rho^\alpha)|\nabla\rho^\alpha|^2+\alpha|\nabla\rho^\alpha|^2\right).
\end{align*}
However, by Fatou's lemma for measures, we have
\begin{align*}
   \liminf_{\beta\to0}  2\beta^{-1}\mathbb{E}\left(q(U\times[\beta/2,\beta]\times[0,T])\right) &\leq  \liminf_{\beta\to0} \mathbb{E}\left(\int_0^T\int_U\int_{\beta/2}^\beta \frac{1}{\xi}\,dq\right)\\
   &=\liminf_{\beta\to0} \mathbb{E}\left(\int_0^T\int_U\int_{\mathbb{R}}\frac{1}{\xi}\mathbbm{1}_{\beta/2\leq\xi\leq\beta}\,dq\right)\\
   &\leq \liminf_{\alpha\to0}\liminf_{\beta\to0}\mathbb{E}\left(\int_0^T\int_U\int_\mathbb{R}\frac{1}{\xi}\mathbbm{1}_{\beta/2\leq\xi\leq\beta}dq^\alpha\right).
\end{align*}
Analogous to the energy estimates of Proposition \ref{energy estimate of regularised equation proposition} and to the entropy estimate of Proposition \ref{entropy estimate}, apply It\^o's formula to a regularised version of the function $\Psi:U\times[0,\infty)\to\mathbb{R}$ defined by
$\Psi(x,0)=0, \partial_\xi\Psi(x,\xi)=S'(\xi)-v(x).$
One obtains, for the  functions $\Theta_\sigma$ and $\Psi_\sigma$ as in Assumption \ref{assumption on kinetic measure}, $\Theta_{\nu,i}$ defined by $\Theta_{\nu,i}(0)=0$, $\Theta'_{\nu,i}(\xi)=\nu_i(\xi)/\xi$, that there is a constant $c\in(0,\infty)$ such that
\begin{align}\label{equation in kinetic measure at 0 ppn}
   & \mathbb{E}\left(\int_U\int_0^t\int_{\mathbb{R}} \frac{1}{\xi}\mathbbm{1}_{0\leq\xi\leq1}\,dq^\alpha\right)=\mathbb{E}\left(\int_U\int_0^t \frac{1}{\rho}\mathbbm{1}_{0\leq\rho\leq1}\left(\Phi'(\rho)|\nabla\rho|^2+\alpha|\nabla\rho|^2\right)\right)\nonumber\\
   &\leq \mathbb{E}\int_U\left(\Psi(x,\rho_0)-\Psi(x,\rho_t)\right)+ 
   cT+ T\sum_{i=1}^d\int_{\partial U}\Theta_{\nu,i}(\Phi^{-1}(\bar{f})\wedge 1)\cdot\hat{\eta}_i+ cT\int_{\partial U}\Theta_\sigma(\Phi^{-1}(\bar{f})\wedge 1)\nonumber\\
   &+T\|\log(\Phi^{-1}(\bar{f})\wedge 1)\|_{H^1(\partial U)}\left(\|\bar{f}\|_{L^2(\partial U)}+\alpha\|\Phi^{-1}(\bar{f})\|_{L^2(\partial U)}+\|\Psi_\sigma({\Phi^{-1}(\bar{f})})\|_{L^2(\partial U)}\right)\nonumber\\
   &+T\|\nabla v\|_{L^2(U)}\left(c\|\sigma(\rho)\sigma'(\rho)\|_{L^2(U)}+\|\nu(\rho)\|_{L^2(U;\mathbb{R}^d)}\right)
    \end{align}
    Since $\int_{\partial U}|\nabla v|^2=\int_{\partial U}v\frac{\partial v}{\partial \hat{\eta}}=\int_{\partial U}\log(\Phi^{-1}(\bar{f})\wedge 1)\frac{\partial v}{\partial \hat{\eta}}$, if $\Phi^{-1}(\bar{f})>1$ then the terms in the final two lines vanish since $\log(1)=0$.\\    
    Furthermore, we have for the first term on the right hand side, that
    \[\Psi(x,\rho_0)=S(\rho_0)-\rho_0v(x)=(\rho_0\wedge1)\log(\rho_0\wedge1)-(\rho_0\wedge1)-\rho_0v(x).\]
    And so by the non-negativity of the solution, the initial condition and $v$, we have by disposing of the negative terms,
\[\mathbb{E}\int_U\left(\Psi(x,\rho_0)-\Psi(x,\rho_t)\right)\leq \mathbb{E}\int_U\left((\rho_0\wedge1)\log(\rho_0\wedge1)+(\rho_t\wedge1)-\rho_tv(x)\right).\]
Using H\"older's inequality, the second assumption of Assumption \ref{assumption on kinetic measure} and the $L^2(U)$ energy estimate \eqref{first energy estimate of regularised solution}, this term is bounded.\\
Hence putting everything together, we showed that the right hand side of \eqref{equation in kinetic measure at 0 ppn} the final term in the above inequality is  bounded.
However, by working along the dyadic scale $\beta^{(i)}=2^{-i}$ for $i=0,1,\hdots$, we have
\begin{align*}
&\sum_{i=0}^\infty\mathbb{E}\left(\int_U\int_0^t\int_{\mathbb{R}} \frac{1}{\beta^{(i)}}\mathbbm{1}_{\beta^{(i)}/2\leq\xi\leq\beta^{(i)}}\,dq^\alpha\right)\leq \mathbb{E}\left(\int_U\int_0^t\int_{\mathbb{R}} \frac{1}{\xi}\mathbbm{1}_{0\leq\xi\leq1}\,dq^\alpha\right)\leq c.
\end{align*}
The infinite sum being bounded by a constant (that is decreasing in $\alpha$) implies that the individual elements of the sum converge to zero, which proves the claim.
\end{proof}

\subsection{Existence of solution to generalised Dean--Kawasaki equation}\label{section 4.2 existence of solution}
In what follows the arguments are identical to that on the torus and so follow \textcolor{red}{\cite[Sec.~5]{fehrman2024well}}. 
We are therefore brief and just provide the main ideas for completeness.\\
In this subsection we start in Proposition \ref{proposition on existence of weak solution to regularised equation with smooth and bounded sigma} by proving the existence of a weak solution of the regularised Dean--Kawasaki equation with smooth and bounded $\sigma$ (in the sense of Assumption \ref{assumption about smoothing sigma in existence}).
We then show that the constructed weak solution is also a stochastic kinetic solution, in the sense that it satisfies a kinetic equation similar to equation \eqref{kinetic equation}.\\
The goal will then subsequently be in Lemma \ref{lemma on approximating sigma by smooth function} to remove the assumption that $\sigma$ is smooth and bounded, which will be done with an approximation argument.
To take the regularisation $\alpha$ limit, showing the existence of a solution to the generalised Dean--Kawasaki equation \eqref{generalised Dean--Kawasaki equation Ito} is done in Theorem \ref{existence theorem}.

\begin{proposition}[Existence of weak solution to regularised equation \eqref{Regularised equation} with smooth and bounded $\sigma$]\label{proposition on existence of weak solution to regularised equation with smooth and bounded sigma}
Let $\xi^F, \Phi,\sigma$ and $\nu$ satisfy Assumptions \ref{Assumption on noise Fi}, \ref{existence assumptions} and \ref{assumption about smoothing sigma in existence}, and let $\alpha\in(0,1)$.
Let further $\rho_0\in L^2(\Omega;L^2(U))$ be non-negative and $\mathcal{F}_0$ measurable.\\
Then there exists a weak solution of the regularised equation \eqref{Regularised equation} in the sense of Definition \ref{definition of weak Solution of regularised equation}. Additionally the solution satisfies the energy estimates of Proposition \ref{energy estimate of regularised equation proposition}.
 \end{proposition}
 
\begin{proof}
The idea is to approximate all coefficients by regular ones, use Galerkin projection argument to show existence and then take limits in the correct order.\\
Start by considering a sequence $\{\Phi_n\}_{n\in\mathbb{N}}$ of smooth bounded non-decreasing functions starting at zero such that $\Phi_n$ and $\Phi_n'$ converge locally uniformly to $\Phi$ and $\Phi'$ as $n\to\infty$.\\
Next consider sequence $\{\nu_n\}_{n\in\mathbb{N}}$ of smooth approximations of $\nu$ that converge to $\nu$ locally uniformly as $n\to\infty$.\\
For $K\in\mathbb{N}$ consider the finite dimensional approximation of the noise $\xi^{F,K}:=\sum_{k=1}^Kf_k(x)B_t^k$, and also the truncated coefficients $F_1^K:=\sum_{k=1}^K f_k^2$, $F_2^K:=\sum_{k=1}^K f_k\nabla f_k$. \\
For $\{e_k\}_{k\in\mathbb{N}}$ an orthonormal basis in $L^2(U)$ which is orthogonal in $H^1(U)$ and $M\in\mathbb{N}$, define $\Pi_M:L^2(U\times[0,T])\to L^2(U\times[0,T])$ be the projection onto the first $M$ orthonormal basis vectors.
That is to say, for any $g\in L^2(U\times[0,T])$,
\[\Pi_M g(x,t):=\sum_{k=1}^M g_k(t)e_k(x)\]
where $g_k(t):=\int_U g(x,t)e_k(x)\,dx$.
For brevity denote the space of projected $L^2$ function by $L^2_M:=\Pi_M(L^2(U\times[0,T]))$. 
Consider the below projected equation with regularised coefficients posed on the space $L^2(\Omega;L^2_M)$
\begin{align*}
        d\rho=&\Pi_M\left(\Delta\Phi_n (\rho)\,dt +\alpha\Delta\rho\,dt-\nabla\cdot (\sigma(\rho) \,d{\xi}^{F,K} + \nu_n(\rho)\,dt)\right)\\
        &+ \Pi_M\left( \frac{1}{2}\nabla\cdot(F_1^K[\sigma'(\rho)]^2\nabla\rho + \sigma'(\rho)\sigma(\rho)F_2^K)\,dt\right).
\end{align*}
The equation is equivalent to a finite dimensional stochastic differential equation. Since $\Phi_n, \nu_n,\sigma$ are all smooth and bounded functions, the system has a unique strong solution.\\
Then we pass to the various limits, first as $M\to \infty$, followed by $K\to\infty$ and finally the limit as $n\to\infty$. 
We rely on simpler versions of the energy estimates of the previous section (e.g. equation \eqref{energy estimate of regularised equation proposition} and Proposition \ref{proposition higher order time regularity of solution cut away from zero set}) and the compact embedding given by Aubin--Lions--Simon lemma (\cite{aubin1963theoreme,lions1969quelques,simon1986compact}) to do this.\\
The resulting solution is continuous in $L^2(U)$ as a consequence of It\^o's formula. The non-negativity of solution follows by applying It\^o's formula to approximations of $min(0,\rho)$, similar to what was done in estimate \eqref{second energy estimate of regularised solution}.
 \end{proof}

 \begin{proposition}[Stochastic kinetic solution of regularised DK equation \eqref{Regularised equation} with smooth and bounded $\sigma$]\label{proposition on stochastic kinetic solution of regularised DK equation with smooth and bounded sigma}
  Let $\xi^F, \Phi,\sigma$ and $\nu$ satisfy Assumptions \ref{Assumption on noise Fi}, \ref{existence assumptions} and \ref{assumption about smoothing sigma in existence}, and let $\alpha\in(0,1)$.
     Let $\rho_0\in L^{m+1}(\Omega;L^1(U))\cap L^2(\Omega;L^2(U))$ be non-negative and $\mathcal{F}_0$ measurable.\\
     Let $\rho$ be a weak solution of \eqref{Regularised equation} with smooth and bounded $\sigma$ in the sense of Definition \ref{definition of weak Solution of regularised equation}, and $\chi(x,\xi,t)=\mathbbm{1}_{0<\xi<\rho(x,t)}$ be the kinetic function on $U\times(0,\infty)\times[0,T]$. Then $\rho$ is a stochastic kinetic solution in the sense that, almost surely, for every $\psi\in C_c^{\infty}(U\times\mathbb{R})$, $t\in[0,T]$:
     \small
     \begin{align*}\label{kinetic equation of regularised solution}
\int_{\mathbb{R}}\int_{U}\chi(x,\xi,t)\psi(x,\xi)\,dx\,d\xi&=\int_{\mathbb{R}}\int_{U}\chi(x,\xi,t=0)\psi(x,\xi)\,dx\,d\xi -\alpha\int_0^t\int_U\nabla\rho\cdot\nabla\psi(x,\xi)|_{\xi=\rho}\,\nonumber\\
&-\int_0^t \int_{U}\left(\Phi'(\rho)\nabla(\rho)+\frac{1}{2}F_1[\sigma'(\rho)]^2\nabla\rho +\frac{1}{2} \sigma'(\rho)\sigma(\rho)F_2\right)\cdot\nabla\psi(x,\xi)|_{\xi=\rho}\,dx\,dt\nonumber\\
    &-\int_0^t\int_{U}\partial_\xi\psi(x,\rho)\Phi'(\xi)|\nabla\rho|^2-\alpha\int_0^t\int_U\partial_\xi\psi (x,\rho)|\nabla\rho|^2\\
    &+\frac{1}{2}\int_0^t \int_{U}\left(
    \sigma'(\rho)\sigma(\rho)\nabla\rho\cdot F_2 + \sigma(\rho)^2F_3 \right)\partial_\xi\psi(x,\rho)\,dx\,dt\nonumber\\
    & -\int_0^t \int_{U}\psi(x,\rho)\nabla\cdot (\sigma(\rho) \,d{\xi}^F)\,dx -        \int_0^t \int_{U}\psi(x,\rho)\nabla\cdot\nu(\rho).
\end{align*}
\normalsize
The derivatives of the test function are again interpreted in the sense of Remark \ref{remark about notation in kinetic equation}.
 \end{proposition}
 
 \begin{proof}[Proof (idea)]
 The proof follows precisely the steps for deriving the stochastic kinetic equation \eqref{kinetic equation}. 
 Begin by using It\^o's formula to derive an equation for $S(\rho)$ for a smooth and bounded function $S:\mathbb{R}\to\mathbb{R}$. Secondly derive a formula for the integral
 \[\int_U S(\rho)\psi(x)\]
 for test function $\psi\in C_c^\infty(U)$ using Definition \ref{definition of weak Solution of regularised equation} of a weak solution. 
 Finally the kinetic equation is derived, noting the density of linear combinations of functions of the form $S'(\xi)\psi(x)$ in $C_c^\infty(U\times\mathbb{R})$.\\
 One noteworthy point is that, as mentioned above, the kinetic measure corresponding to the solution $\rho$ constructed above is
 \[q=\delta_0(\xi-\rho)\Phi'(\xi)|\nabla\rho|^2+\alpha\delta_0(\xi-\rho)|\nabla\rho|^2=\delta_0(\xi-\rho)\left(|\nabla\Theta_\Phi(\rho)|^2+\alpha|\nabla\rho|^2\right).\]
 The measure is finite due to the estimates of Proposition \ref{energy estimate of regularised equation proposition} and satisfies the other assumptions of a kinetic measure as in Definition \ref{definition of stochastic kinetic solution of generalised Dean--Kawasaki equation} due to Assumption \ref{existence assumptions}.\\
 For further details see proof of \textcolor{red}{\cite[Prop.~5.21]{fehrman2024well}}.
 \end{proof}
 
 Next we wish to extend the well-posedness to the generalised Dean--Kawasaki equation \eqref{generalised Dean--Kawasaki equation Ito}. The first step is to dispense of the regularity assumption on $\sigma$ of Assumption \ref{assumption about smoothing sigma in existence}.
 
\begin{lemma}[Approximating $\sigma$ in $C^1_{loc}$]\label{lemma on approximating sigma by smooth function}
Let $\sigma$ satisfy Assumption \ref{existence assumptions}. 
Then one can find a sequence $\{\sigma_n\}_{n\in\mathbb{N}}$ such that for each $n$, $\sigma_n$ satisfies Assumption \ref{assumption about smoothing sigma in existence}. 
Further, the sequence uniformly satisfy Assumption \ref{existence assumptions} and $\sigma_n\to\sigma$ in $C^1_{loc}((0,\infty))$.     
\end{lemma}

The proof follows from constructing smooth bounded approximations by convolution which can be done due to the regularity of $\sigma$ from Assumption \ref{existence assumptions}.\\
 The difficulty in extending the well-posedness to \eqref{generalised Dean--Kawasaki equation Ito} is that the weak solution constructed in Proposition \ref{proposition on existence of weak solution to regularised equation with smooth and bounded sigma} does not have a stable $W^{\beta,1}_t H^{-s}$ energy estimate.
 We only have stable $W^{\beta,1}_t H^{-s}$ for the solution bounded away from its zero set, as in Proposition \ref{proposition higher order time regularity of solution cut away from zero set}.
 We deal with this by defining the below metric on $L^1_tL^1_x$. 
 Tightness of the cutoff solution $\Phi_\delta(\rho)$ as in Definition \eqref{ definition of cutoff away from 0} will be proved with respect to this metric.
 
\begin{definition}[New metric on $L^1_tL^1_x$]\label{definition of new metric on L1L1} For $\delta\in(0,1)$ let $\Psi_\delta$ be defined as in Definition \eqref{ definition of cutoff away from 0}. Define $D:L^1([0,T]; L^1(U))\to[0,\infty)$ by
   \begin{align*}
       D(f,g)=\sum_{k=1}^{\infty}2^{-k}\left(\frac{\|\Psi_{1/k}(f)-\Psi_{1/k}(g)\|_{L^1([0,T];L^1(U))}}{1+\|\Psi_{1/k}(f)-\Psi_{1/k}(g)\|_{L^1([0,T];L^1(U))}}\right).
   \end{align*} 
\end{definition}

\begin{lemma}
    The function $D$ defined above is a metric on $L^1([0,T];L^1(U))$. The metric topology defined by $D$ coincides with the strong norm topology on $L^1([0,T];L^1(U))$.
\end{lemma}

The proof of the above lemma can be found in \textcolor{red}{\cite[Lem.~5.24]{fehrman2024well}}.
Instead of assuming $\sigma$ in the regularised equation \eqref{Regularised equation} satisfies Assumption \ref{assumption about smoothing sigma in existence}, we define an approximate equation with $\sigma_n$ as defined in Lemma \ref{lemma on approximating sigma by smooth function} approximating $\sigma$.

\begin{definition}[Regularised and smoothed $\sigma$ equation]\label{definiton of regularised and smoothed sigma equation}
      Let $\xi^F, \Phi,\sigma$ and $\nu$ satisfy Assumptions \ref{Assumption on noise Fi} and \ref{existence assumptions}, and let $T\in[1,\infty)$.
     Let $\rho_0\in L^2(\Omega;L^2(U))$ be non-negative and $\mathcal{F}_0$ measurable, and $\sigma_n$ as in Lemma \ref{lemma on approximating sigma by smooth function}.\\
     For every $n\in\mathbb{N}$ and $\alpha\in(0,1)$, define $\rho^{a,n}$ to be the stochastic kinetic solution of
     \begin{align}\label{regularised and sigma smoothed dean kawaski equation}
    d\rho^{\alpha,n}&=\Delta\Phi (\rho^{\alpha,n})\,dt +\alpha\Delta\rho^{\alpha,n}\,dt-\nabla\cdot (\sigma_n(\rho^{\alpha,n}) \,d{\xi}^F + \nu(\rho^{\alpha,n})\,dt)\nonumber\\
    &+ \frac{1}{2}\nabla\cdot(F_1[\sigma_n'(\rho^{\alpha,n})]^2\nabla\rho^{\alpha,n} + \sigma_n'(\rho^{\alpha,n})\sigma_n(\rho^{\alpha,n})F_2)\,dt,
\end{align}
in $U\times(0,T)$ with initial data $\rho_0$ and boundary condition $\Phi(\rho)=\bar{f}$ as constructed in Proposition \ref{proposition on stochastic kinetic solution of regularised DK equation with smooth and bounded sigma}.
\end{definition}

The below proposition is a key element of the existence proof. The proof can be found in \textcolor{red}{\cite[Prop.~5.26, Prop.~5.27]{fehrman2024well}}.

\begin{proposition}[Tightness of laws of $\rho^{\alpha,n}$ in $L^1_tL^1_x$ and of martingale term in $C^\gamma_t$]\label{proposition on tightness of rho and maringale term}
    Let $\xi^F, \Phi,\sigma$ and $\nu$ satisfy Assumptions \ref{Assumption on noise Fi} and \ref{existence assumptions}.
     Let $\rho_0\in L^2(\Omega;L^2(U))$ be non-negative and $\mathcal{F}_0$ measurable, $\sigma_n$ as in Lemma \ref{lemma on approximating sigma by smooth function} and kinetic solutions $\rho^{\alpha,n}$ be as in Definition \ref{definiton of regularised and smoothed sigma equation}.
     \begin{enumerate}
         \item The laws of $\{\rho^{\alpha,n}\}_{\alpha\in(0,1), n\in\mathbb{N}}$ are tight on $L^1([0,T];L^1(U))$ with respect to the strong norm topology.
         \item For each test function $\psi\in C_c^{\infty}(U\times(0,\infty))$, $\gamma\in(0,1/2)$ the laws of the martingales
     \[M_t^{\alpha,n,\psi}:=\int_0^t\int_U\psi(x,\rho^{\alpha,n})\nabla\cdot(\sigma_n(\rho^{\alpha,n})d\xi^F)\]
     are tight on $C^\gamma([0,T])$.
     \end{enumerate}     
\end{proposition}

One of the main results to prove existence will come from the below technical lemma, see \textcolor{red}{\cite[Lem.~1.1]{gyongy1996existence}} for proof. 

\begin{lemma}\label{covergence in probability lemma}
Let $(\Omega,\mathcal{F},\mathbb{P})$ be a probability space and $\bar{X}$ be a complete separable metric space. Then a sequence $\{X_n:\Omega\to\bar{X}\}$ of $\bar{X}$ valued random variables converges in probability as $n\to\infty$ if and only if for every pair of sequences $(n_k,m_k)_{k=1}^\infty$ with $n_k,m_k\to\infty$ as $k\to\infty$, there is a further sub-sequence $(n_{k'},m_{k'})_{k=1}^\infty$ such that the joint laws $(X_{n_{k'}},X_{m_{k'}})$ converge weakly as $k'\to\infty$ to a probability measure $\mu$ on $\bar{X}\times\bar{X}$ satisfying $\mu(\{(x,y)\in\bar{X}\times\bar{X}:x=y\})=1$.
\end{lemma}

We state the main existence result, which is stated as \textcolor{red}{\cite[Thm.~5.29]{fehrman2024well}}. 
The full proof can be found there, we will just explain the main idea by putting all the previous results from this section together. \textcolor{red}{Similar arguments in a simpler setting can be found in \cite{debussche2014scalar}. 
See also \textcolor{red}{\cite[Thm.~6.2]{wang2022dean}} for a similar proof.}
\begin{theorem}[Existence of solution to \eqref{generalised Dean--Kawasaki equation Ito}]\label{existence theorem}
     Let $\xi^F, \Phi,\sigma$ and $\nu$ satisfy Assumptions \ref{Assumption on noise Fi} and \ref{existence assumptions}.
     Let $\rho_0\in L^1(\Omega;L^1(U))$ be non-negative and $\mathcal{F}_0$ measurable.\\
     Then there exists a stochastic kinetic solution to the generalised Dean--Kawasaki equation \eqref{generalised Dean--Kawasaki equation Ito} in the sense of Definition \ref{definition of stochastic kinetic solution of generalised Dean--Kawasaki equation}. Furthermore, the solution satisfies the estimates of Proposition \ref{energy estimate of regularised equation proposition}.
\end{theorem}

\begin{proof}
We provide the main steps of the proof and omit the technical details.
\begin{enumerate}
    \item \textbf{Tightness.}\\Recall  the stochastic kinetic solutions $\{\rho^{\alpha,n}\}_{\alpha\in(0,1),n\in\mathbb{N}}$ as defined in Definition \ref{definiton of regularised and smoothed sigma equation}, martingales $M^{\alpha,n,\psi}$ as introduced in Proposition \ref{proposition on tightness of rho and maringale term} and introduce the measures 
    \[q^{\alpha,n}:=\delta_0(\xi-\rho^{\alpha,n})\left(|\nabla\Theta_{\Phi}(\rho^{\alpha,n})|^2+\alpha|\rho^{\alpha,n}|^2\right).\]
     Proposition \ref{proposition on stochastic kinetic solution of regularised DK equation with smooth and bounded sigma} gives us existence of the solutions and the energy estimate of Proposition \ref{energy estimate of regularised equation proposition} allows us to deduce that $\{q^{\alpha,n}\}_{\alpha\in(0,1),n\in\mathbb{N}}$ are finite kinetic measures.\\
    Using the kinetic equation given in Proposition \ref{proposition on stochastic kinetic solution of regularised DK equation with smooth and bounded sigma}, one can write an equation for the kinetic functions $\chi^{\alpha,n}$ of $\rho^{\alpha,n}$. 
    Fixing a dense sequence of functions $\{\psi_j\}_{j\in\mathbb{N}}$ of $C_c^{\infty}(U\times(0,\infty))$ in the strong $H^s(U\times (0,\infty))$ topology (for $s>d/2+1$), we consider the random variables coming from the kinetic equation of $\chi^{\alpha,n}$:
    \[X^{\alpha,n}:=(\rho^{\alpha,n},\nabla\Theta_{\Phi}(\rho^{\alpha,n}), \alpha\nabla\rho^{\alpha,n},q^{\alpha,n},(M^{\alpha,n,\psi_j})_{j\in\mathbb{N}})\]
    on the space
    \[\Bar{X}:=L^1(U\times(0,T))\times L^2\left(U\times(0,T);\mathbb{R}^d\right)^2\times  \mathcal{M}(U\times(0,\infty)\times[0,T])\times C([0,T])^\mathbb{N}.\]
    Equip $\bar{X}$ with the product topology, with the strong topology on $L^1(U\times(0,T))$, the weak topologies on $L^2(U\times(0,T);\mathbb{R}^d)$ and $\mathcal{M}(U\times(0,\infty)\times[0,T])$ and topology of component wise convergence in the strong norm on $C([0,T])^\mathbb{N}$, in particular using the norm constructed before:
     \begin{align*}
       D_C((f_k),(g_k))=\sum_{k=1}^{\infty}2^{-k}\left(\frac{\|f_k-g_k\|_{C([0,T])}}{1+\|f_k-g_k\|_{C([0,T])}}\right).
   \end{align*} 
   To show convergence in probability of the random variables $X^{\alpha,n}$ we try to use Lemma \ref{covergence in probability lemma}.
   To this end, we consider two subsequences $(\alpha_k,n_k)$, $(\beta_k,m_k)$ such that $\alpha_k,\beta_k\to 0$ and $n_k,m_k\to \infty$ as $k\to\infty$
   Consider the laws on $\bar{Y}:=\bar{X}\times\bar{X}\times C([0,T])^\mathbb{N}$ of 
   \[(X^{\alpha_k,n_k},X^{\beta_k,m_k}, B),\]
   where $B=(B^j)_{j\in\mathbb{N}}$ are the Brownian motions defined in the noise $\xi^F$ in Definition \ref{definition noise}.\\
   The energy estimate in Proposition \ref{energy estimate of regularised equation proposition} alongside the two tightness results in Proposition \ref{proposition on tightness of rho and maringale term}
 show that the laws of $(X^{\alpha,n})$ are tight on $\bar{X}$.
\item \textbf{Skorokhod representation theorem.}\\
By Prokhrov's theorem, passing to a sub-sequence still denoted by $k\to\infty$, there is a probability measure $\mu$ on $\bar{Y}$ such that $(X^{\alpha_k,n_k},X^{\beta_k,m_k}, B)\to\mu$ as $k\to\infty$.\\
$\bar{X}$ being separable implies that $\bar{Y}$ is separable, so we can apply Skorokhod representation theorem. It tells us that there is an auxiliary probability space $(\Tilde{\Omega},\Tilde{\mathcal{F}},\tilde{\mathbb{P}})$ such that for every $k$,
\[(\Tilde{Y}^k,\Tilde{Z}^k,\Tilde{\beta}^k)=(X^{\alpha_k,n_k},X^{\beta_k,m_k}, B) \hspace{20pt}\text{in law on }\hspace{3pt}\bar{Y},\]
\[(\Tilde{Y},\Tilde{Z},\Tilde{\beta})=\mu \hspace{20pt}\text{in law on }\hspace{3pt}\bar{Y},\]
and we have the almost sure convergence as $k\to\infty$:
\[(\Tilde{Y}^k,\Tilde{Z}^k,\Tilde{\beta}^k)\to (\Tilde{Y},\Tilde{Z},\Tilde{\beta})\]
in the space $\bar{X}$ and $C([0,T])$. To apply Lemma \ref{covergence in probability lemma} we will show $\Tilde{Y}=\Tilde{Z}$.
\item \textbf{Characterising $\Tilde{Y}$.}\\
It follows from the equality in law of $\Tilde{Y}^k$ and $X^{\alpha_k,n_k}$ that there is a $\Tilde{\rho}^k\in L^\infty(\Omega\times[0,T];L^1(U))$ and $\Tilde{G}^k_1,\Tilde{G}^k_2,\Tilde{q}^k,(\Tilde{M}^{k,\psi_j})_{j\in\mathbb{N}}$ in the appropriate spaces such that 
\[\Tilde{Y}^k=(\Tilde{\rho}^k,\Tilde{G}^k_1,\Tilde{G}^k_2,\Tilde{q}^k,(\Tilde{M}^{k,\psi_j})_{j\in\mathbb{N}}).\]
By converting various expectations $\Tilde{\mathbb{E}}$ into expectations $\mathbb{E}$ by using the equalities in law above, and further using that $\alpha_k\nabla\rho^k\rightharpoonup0$ in $L^2_tL^2_x$ by energy estimates (Proposition \ref{energy estimate of regularised equation proposition}) tells us that in the limit as $k\to\infty$
\[\Tilde{Y}=(\Tilde{\rho},\nabla\Theta_{\Phi}(\Tilde{\rho}),0,\Tilde{q},(\Tilde{M}^j)_{j\in\mathbb{N}}),\]
where $\Tilde{p}\in L^1(\tilde{\Omega}\times[0,T];L^1(U))$ and $\Tilde{q}$ is the corresponding kinetic measure.\\
It remains to characterise $\Tilde{M}^j$, and to do this we first need to characterise $\Tilde{\beta}^k$.
\item \textbf{The path $\Tilde{\beta}$ is a Brownian Motion.}\\
Writing for each $k$, $\Tilde{\beta}^k:=(\Tilde{\beta}^{k,j})_{j\in\mathbb{N}}$, and the limiting process $\Tilde{\beta}=(\Tilde{\beta}^j)_{j\in\mathbb{N}}$, one obtains using the same trick of interchanging expectations $\Tilde{\mathbb{E}}$ and $\mathbb{E}$ using equalities in law that, by proving first for $\Tilde{\beta}^{k,j}$ then passing to the limit in $k$, that for any $F:\Bar{Y}\to\mathbb{R}$, $0\leq s\leq t\leq T$, $j\in\mathbb{N}$
\[\tilde{\mathbb{E}}\left(F\left(\tilde{Y}|_{[0,s]},\tilde{Z}|_{[0,s]},\tilde{\beta}|_{[0,s]}\right)\left(\tilde{\beta}_t^j-\tilde{\beta}_s^j\right)\right)=0.\]
Identically for $i,j\in\mathbb{N}$, $0\leq s\leq t\leq T$,
\[\tilde{\mathbb{E}}\left(F\left(\tilde{Y}|_{[0,s]},\tilde{Z}|_{[0,s]},\tilde{\beta}|_{[0,s]}\right)\left(\tilde{\beta}_t^j\tilde{\beta}_t^i-\tilde{\beta}_s^j\tilde{\beta}_s^i-\delta_{ij}(t-s)\right)\right)=0,\]
where $\delta_{ij}$ is the Kronecker delta. 
Using these and the fact that $\tilde{\beta}^j$ has almost surely continuous paths,we conclude using L\'evy's characterisation that $\tilde{\beta}^j$ are independent one dimensional Brownian motions with respect to the filtration
\[\mathcal{G}_t=\sigma(\tilde{Y}|_{[0,t]},\tilde{Z}|_{[0,t]},\tilde{\beta}|_{[0,t]}).\]
By continuity and  uniform integrability $\tilde{\beta}$ is a Brownian motion with respect to the augmented filtration $\bar{\mathcal{G}}$ of $\mathcal{G}$.
\item\textbf{$(\tilde{M}^j)_{j\in\mathbb{N}}$ are $\bar{\mathcal{G}}_t$ martingales.}\\
The statement follows using a similar technique as the previous point. First showing for $j\in\mathbb{N}$, $0\leq s\leq t\leq T$ and $k\in\mathbb{N}$,
\[\tilde{\mathbb{E}}\left(F\left(\tilde{Y}^k|_{[0,s]},\tilde{Z}^k|_{[0,s]},\tilde{\beta}^k|_{[0,s]}\right)\left(\tilde{M}_t^{k,\psi_j}-\tilde{M}_s^{k,\psi_j}\right)\right)=0.\]
The result follows by taking the limit as $k\to\infty$ using the uniform integrability of the martingales.
\item \textbf{$(\tilde{M}^j)_{j\in\mathbb{N}}$ are stochastic integrals with respect to $\tilde{\beta}$.}\\
Again this follows from the same techniques as before. First proving the results for the approximations and then taking a limit as $k\to\infty$, we can prove that 
\[\tilde{\mathbb{E}}\left(F\left(\tilde{Y}|_{[0,s]},\tilde{Z}|_{[0,s]},\tilde{\beta}|_{[0,s]}\right)\left(\tilde{M}^j_t\tilde{\beta}^i_t - \tilde{M}^j_s\tilde{\beta}^i_s -\int_s^t\int_U\psi_j(x,\tilde{\rho})\nabla\cdot(\sigma(\tilde{\rho})f_i)\right)\right)=0,\]
where recall $f_i$ are defined as the spatial components of the noise $\xi^F$.
Hence this shows for each $i\in\mathbb{N}$,
\[\tilde{M}^j_t\tilde{\beta}^i_t -\int_s^t\int_U\psi_j(x,\tilde{\rho})\nabla\cdot(\sigma(\tilde{\rho})f_i)\hspace{15pt} \text{is a}\hspace{5pt} \mathcal{G}-\text{martingale}.\]
It is easy to see by uniform integrability and continuity that the process is also a $\bar{\mathcal{G}}_t$ martingale. Identical arguments show that for $j\in\mathbb{N}$
\[(\tilde{M}^j_t)^2-\int_0^t\sum_{k=1}^\infty \left(\int_U\psi_j(x,\tilde{\rho})\nabla\cdot(\sigma(\tilde{\rho})f_k)\right)^2\]
is a continuous $\bar{\mathcal{G}}_t$ martingale. Putting everything together, due to an explicit calculation using the quadratic variation of Brownian motion, for every $j\in\mathbb{N}$, $t\in[0,T]$,
\[\tilde{E}\left(\left(\tilde{M}^j_t -\int_0^t\int_U\psi_j(x,\tilde{\rho})\nabla\cdot(\sigma(\tilde{\rho})\tilde{\xi}^F)\right)^2\right)=0,\]
where $\tilde{\xi}^F$ is defined analogously to $\xi^F$ but with Brownian Motion $\tilde{\beta}$ on $\tilde{\Omega}$. It follows that $\tilde{M}^j_t =\int_0^t\int_U\psi_j(x,\tilde{\rho})\nabla\cdot(\sigma(\tilde{\rho})\tilde{\xi}^F)$.
\item \textbf{Tying up loose ends.}\\
One needs to show the following technical steps in order to finish the proof.
\begin{enumerate}
    \item Show the limiting kinetic measure $\tilde{q}$ is almost surely a kinetic measure for $\tilde{p}$.
    \item Show that $\sigma(\tilde{\rho})$ is in $L^2$.
    \item Remove the set $\mathcal{A}:=\{t\in[0,T]:\tilde{q}(\{t\}\times U \times \mathbb{R})\neq 0\}$. Outside $\mathcal{A}$ there is no ambiguity when writing the kinetic equation for the kinetic function $\tilde{\chi}$.
    \item Show that $\tilde{\rho}\in L^1([0,T];L^1(U))$.\\
    This is quite a technical step, it involves looking at left and right continuous representations of $\tilde{\rho}$.
    One needs to study properties of the left and right kinetic functions $\chi^{\pm}$.
    Conclude by showing that the measure $\tilde{q}$ almost surely has no atoms in time.
\end{enumerate}
\item\textbf{Conclusion.}\\
We showed the existence of $\tilde{\rho}$ with representative in $L^1(\tilde{\Omega}\times[0,T];L^1(U))$ 
$\tilde{\rho}$ is a stochastic kinetic solution in the sense of Definition \ref{definition of stochastic kinetic solution of generalised Dean--Kawasaki equation} with respect to Brownian Motion $\tilde{\beta}$ and filtration $(\bar{\mathcal{G}}_t)_{t\in[0,T]}$. 
That is to say, we showed the existence of a probabilistically weak solution.
We now explain how to extend this to a probabilistically strong solution.\\
Repeating all steps from Step 3 of the above, it follows that we can characterise $\tilde{Z}$ as 
\[\Tilde{Z}=(\bar{\rho},\nabla\Theta_{\Phi}(\bar{\rho}),0,\bar{q},(\bar{M}^j)_{j\in\mathbb{N}}).\]
Continuing, one shows there is an $L^1$ continuous representation of $\bar{\rho}$ which is a stochastic kinetic solution in the sense of Definition \ref{definition of stochastic kinetic solution of generalised Dean--Kawasaki equation} with respect to Brownian Motion $\tilde{\beta}$ and filtration $(\bar{\mathcal{G}}_t)_{t\in[0,T]}$.\\
The uniqueness theorem, Theorem \ref{main uniqueness theorem} tells us $\tilde{\rho}=\bar{\rho}$ almost surely in $L^1_tL^1_x$.\\
By Lemma \ref{covergence in probability lemma}, it follows that after passing $\{\rho^{a,n}\}$ to a sub-sequence $\alpha_k,n_k$ on the original probability space $(\Omega,\mathcal{F},\mathbb{P})$, there is a random variable $\rho\in L^1(\Omega\times[0,T];L^1(U))$ such that $\rho^{\alpha_k,n_k}$ converges to $\rho$ in probability.\\
Working along a further sub-sequence we have that $\rho^{\alpha_k,n_k}$ converges almost surely to $\rho$.  
Repeating the steps again above we can show $\rho$ is a stochastic kinetic solution of the generalised Dean--Kawasaki equation \eqref{generalised Dean--Kawasaki equation Ito} in the sense of Definition \ref{definition of stochastic kinetic solution of generalised Dean--Kawasaki equation}.
Noting that $\rho^{\alpha,n}$ are all probabilistically strong solutions, $\rho$ is also a probabilistically strong solution.
The energy estimates follow from the estimates for the regularised equations and the weak lower semicontinuity of the Sobolev norm.
\end{enumerate}
\end{proof}

\bibliographystyle{alpha}  
\newpage
\Large{\textbf{Acknowledgements}}\\
\vspace{0pt}
\normalsize

The author is grateful to his PhD advisor Benjamin Fehrman for the constant academic support and guidance.
This work would not have been possible without him.
\textcolor{red}{The author is also grateful to the reviewers for helpful feedback on the work.}\\
This publication is based on work supported by the EPSRC Centre for Doctoral Training in Mathematics of Random Systems: Analysis, Modelling and Simulation (EP/S023925/1).

\end{document}